\documentclass[10pt]{article}

\usepackage{amsfonts}
\usepackage{amssymb}
\usepackage{amsthm}
\usepackage{amsmath}
\usepackage{bm}
\usepackage{amscd} 
\usepackage{tikz-cd}

\usepackage{latexsym}

\usepackage{comment}
\usepackage{amsfonts}
\usepackage{amssymb}
\usepackage{amsthm}
\usepackage{amsmath}
\usepackage{bm}
\usepackage{amscd} 
\usepackage{graphicx}
\usepackage{mathrsfs}
\usepackage{amsrefs}
\usepackage{MnSymbol}
\usepackage{hyperref}

\usepackage[all]{xy}

\usepackage{hyperref}

\makeatletter
\newcommand*{\doublerightarrow}[2]{\mathrel{
  \settowidth{\@tempdima}{$\scriptstyle#1$}
  \settowidth{\@tempdimb}{$\scriptstyle#2$}
  \ifdim\@tempdimb>\@tempdima \@tempdima=\@tempdimb\fi
  \mathop{\vcenter{
    \offinterlineskip\ialign{\hbox to\dimexpr\@tempdima+1em{##}\cr
    \rightarrowfill\cr\noalign{\kern.5ex}
    \rightarrowfill\cr}}}\limits^{\!#1}_{\!#2}}}
\newcommand*{\triplerightarrow}[1]{\mathrel{
  \settowidth{\@tempdima}{$\scriptstyle#1$}
  \mathop{\vcenter{
    \offinterlineskip\ialign{\hbox to\dimexpr\@tempdima+1em{##}\cr
    \rightarrowfill\cr\noalign{\kern.5ex}
    \rightarrowfill\cr\noalign{\kern.5ex}
    \rightarrowfill\cr}}}\limits^{\!#1}}}
\makeatother

\numberwithin{equation}{section} \DeclareMathSizes{2}{10}{12}{13}
\newtheorem{thm}{Proposition}[section]
\newtheorem{Thm}[thm]{Theorem}
\newtheorem{rem}[thm]{Remark}

\newtheorem{cor}[thm]{Corollary}
\newtheorem{lem}[thm]{Lemma}
\newtheorem{defn}[thm]{Definition}

\newtheorem{Thrm}{Theorem}

\title{Entwined modules over representations of categories
}

\smallskip 

\smallskip 
\author{Abhishek Banerjee \footnote{Department of Mathematics, Indian Institute of Science, Bangalore,  India. Email: abhishekbanerjee1313@gmail.com} \footnote{partially supported by SERB Matrics fellowship MTR/2017/000112} }
\date{ }

\begin{document}

\maketitle

\medskip

\begin{abstract}
We introduce  a theory of modules over a representation of a small category taking values in entwining structures over a semiperfect coalgebra. This takes forward the aim of developing  categories of entwined modules to the same extent as that of module categories as well as the philosophy of Mitchell of working with rings with several objects. The representations are motivated by work of Estrada and Virili, who developed a theory of modules over a representation taking values in small preadditive categories, which were then studied in the same spirit as sheaves of modules over a scheme. We also describe, by means of Frobenius and separable functors, how our theory relates
to that of modules over the underlying representation taking values in small $K$-linear categories.
\end{abstract}

\medskip

{\emph{{\bf\emph{MSC(2020) Subject Classification:} } 16T15, 18E05, 18E10}}

\medskip

{ \emph{{\bf \emph{Keywords:}} Rings with several objects, entwined modules, separable functors, Frobenius pairs}} 

\medskip

\section{Introduction}

The purpose of this paper is to study modules over representations of a small category taking values in spaces that behave like quotients of  categorified fiber bundles. Let $H$ be a Hopf algebra having a coaction $\rho:A\longrightarrow A\otimes H$ on an algebra $A$ such that $A$ becomes an $H$-comodule algebra. Let $B$ denote the algebra of coinvariants of this coaction. 
Suppose that the inclusion $B\hookrightarrow A$ is faithfully flat and the canonical morphism
\begin{equation*}
can: A\otimes_BA\longrightarrow A\otimes H \qquad x\otimes y\mapsto x\cdot \rho(y)
\end{equation*}
is an isomorphism. This  datum is the algebraic counterpart of  a principal fiber bundle given by the quotient of an affine algebraic group scheme acting freely on an affine scheme over a field $K$ (see, for instance, \cite{MF}, \cite{Schn}). If $H$ has bijective antipode, then modules over the algebra $B$ of coinvariants may be recovered as ``$(A,H)$-Hopf modules'' (see Schneider \cite{Schn}). 

\smallskip
These $(A,H)$-Hopf modules may be rolled into the more general concept of modules over an `entwining structure' consisting of an algebra $R$, a coalgebra $C$
and a morphism $\psi:C\otimes R\longrightarrow R\otimes C$ satisfying certain conditions. Entwining structures were introduced by Brzezi\'{n}ski and Majid \cite{BrMj}. It was soon realized (see 
Brzezi\'{n}ski \cite{Brx1}) that  entwining structures provide a single formalism that unifies relative Hopf modules, Doi-Hopf modules, Yetter-Drinfeld modules and several other concepts such as coalgebra Galois extensions. As pointed out in  Brzezi\'{n}ski \cite{Brx2}, an entwining structure $(R,C,\psi)$ behaves like a single bialgebra, or more generally a comodule algebra
over a bialgebra. Accordingly, the investigation of entwining structures as well as the modules over them has emerged as an object of study in its own right (see, for instance, \cite{Abu}, \cite{BBR0}, \cite{BBR}, \cite{Brx1}, \cite{Brx3}, \cite{BuTa2}, \cite{BuTa1}, \cite{CaDe}, \cite{HP}, \cite{Jia}, \cite{Schb}). 

\smallskip
We consider an entwining structure consisting of a small $K$-linear category $\mathcal R$, a coalgebra $C$ and a family of morphisms
\begin{equation*}
\psi=\{\psi_{rs}:C\otimes \mathcal R(r,s)\longrightarrow \mathcal R(r,s)\otimes C\}_{r,s\in \mathcal R}
\end{equation*} satisfying certain conditions (see Definition \ref{entcatx}). This is in keeping with the general philosophy of Mitchell \cite{Mit}, where a small
$K$-linear category is viewed as a $K$-algebra with several objects. In fact, we consider the category $\mathscr Ent$ of such entwining structures. When the coalgebra $C$ is fixed, we have the subcategory $\mathscr Ent_C$. Given an entwining structure $(\mathcal R,C,\psi)$, we have a category $\mathbf M^C_{\mathcal R}(\psi)$ of  modules over it (see our earlier work in
\cite{BBR}). These entwined modules over $(\mathcal R,C,\psi)$ may be seen as modules over a certain categorical quotient space of $\mathcal R$, which need not exist in an explicit
sense, but is studied only through its category of modules.

\smallskip
We work with representations $\mathscr R:\mathscr X\longrightarrow \mathscr Ent_C$ of a small category $\mathscr X$ taking values in $\mathscr Ent_C$, where $C$ is a fixed coalgebra. This is motivated by the work of Estrada and  Virili \cite{EV}, who introduced a theory of modules over a representation $\mathscr A:\mathscr X\longrightarrow Add$, where 
$Add$ is the category of small preadditive categories. The modules over $\mathscr A:\mathscr X\longrightarrow Add$ were studied in the spirit of sheaves of modules over a scheme, or more generally, a ringed space. By considering small preaditive categories, the authors in \cite{EV} also intended to take Mitchell's idea one step forward: from replacing rings with small preadditive categories to replacing ring representations by representations taking values in small preadditive categories. In this paper, we develop a theory of modules over a representation
$\mathscr R:\mathscr X\longrightarrow \mathscr Ent_C$ taking values in entwining structures. We also describe, by means of Frobenius and separable functors, how this theory relates
to that of modules over the underlying representation taking values in small $K$-linear categories.

\smallskip
This paper has two parts. In the first part, we introduce and develop the properties of the category $Mod^C-\mathscr R$ of modules over $\mathscr R:\mathscr X\longrightarrow \mathscr Ent_C$. For this, we have to combine techniques on comodules along with adapting the methods of Estrada and Virili \cite{EV}.   When $\mathscr R:\mathscr X\longrightarrow \mathscr Ent_C$ is a flat representation (see Section 6), we also consider the subcategory $Cart-\mathscr R$ of cartesian entwined modules over $\mathscr R$. In the analogy with sheaves of modules over a scheme, the cartesian objects may be seen as similar to quasi-coherent sheaves. 

\smallskip
Let $\mathscr Lin$ be the category of small $K$-linear categories. In the second part, we consider the underlying representation $\mathscr R:\mathscr X\longrightarrow \mathscr Ent_C
\longrightarrow \mathscr Lin$, which we continue to denote by $\mathscr R$. Accordingly, we have a category $Mod-\mathscr R$ of modules over $\mathscr R:\mathscr X\longrightarrow \mathscr Ent_C
\longrightarrow \mathscr Lin$ in the sense of Estrada and Virili \cite{EV}.  We study the relation between $Mod^C-\mathscr R$ and $Mod-\mathscr R$ by describing Frobenius and separability conditions for a pair of adjoint functors between them (see Section 7)
\begin{equation*}
\mathscr F:Mod^C-\mathscr R\longrightarrow Mod-\mathscr R\qquad\qquad\qquad \mathscr G:Mod-\mathscr R\longrightarrow Mod^C-\mathscr R
\end{equation*} Here, the left adjoint $\mathscr F$ may be thought of as an `extension of scalars' and the right adjoint $\mathscr G$ as  a `restriction of scalars.'

\smallskip
The idea is as follows: as mentioned before,  modules  over an entwining structure $(\mathcal R,C,\psi)$ may be seen as modules over a certain categorical quotient space of $\mathcal R$, which behaves like a subcategory of $\mathcal R$. Again, this ``subcategory'' of $\mathcal R$ need not exist in an explicit sense, but is studied only
through the category of modules $\mathbf M^C_{\mathcal R}(\psi)$. Accordingly, a representation $\mathscr R:\mathscr X\longrightarrow \mathscr Ent_C$ taking values
in $\mathscr Ent_C$ may be thought of as a subfunctor of the underlying representation $ \mathscr R:\mathscr X\longrightarrow \mathscr Ent_C
\longrightarrow \mathscr Lin$. We want to understand the properties of the inclusion of this ``subfunctor'': in particular, whether it behaves like a separable, split or Frobenius
extension of rings. We recall here (see \cite[Theorem 1.2]{uni}) that if $R\longrightarrow S$ is an extension of rings, these properties may be expressed in terms of the functors
$F:Mod-R\longrightarrow Mod-S$ (extension of scalars) and $G:Mod-S\longrightarrow Mod-R$ (restriction of scalars) as follows
\begin{equation*}
\mbox{
\begin{tabular}{ccc}
$R\longrightarrow S$ split extension & $\qquad\qquad \Leftrightarrow\qquad\qquad$ & $F:Mod-R\longrightarrow Mod-S$ separable\\
$R\longrightarrow S$ separable extension & $\qquad\qquad \Leftrightarrow\qquad\qquad$  & $G:Mod-S\longrightarrow Mod-R$ separable\\
$R\longrightarrow S$ Frobenius extension  &  $\qquad\qquad \Leftrightarrow\qquad\qquad$  & $(F,G)$ Frobenius pair of functors\\
\end{tabular}}
\end{equation*}

\smallskip
We now describe the paper in more detail. Throughout,we let $K$ be a field. We begin in Section 2 by describing the categories of entwining structures and entwined modules. For a morphism
$(\alpha,\gamma):(\mathcal R,C,\psi)\longrightarrow (\mathcal S,D,\psi')$ of entwining structures, we describe `extension of scalars' and `restriction of scalars' on categories
of entwined modules. Our first result is as follows.

\begin{Thrm}\label{resulta} (see \ref{P2.2}, \ref{P2.3} and \ref{T2.5}) Let  $(\alpha,\gamma):(\mathcal R,C,\psi)\longrightarrow (\mathcal S,D,\psi')$ be a morphism of entwining structures.

\smallskip
(1) There is a  functor $(\alpha,\gamma)^\ast : \mathbf M_{\mathcal R}^C(\psi)\longrightarrow \mathbf M_{\mathcal S}^D(\psi')$ of extension of scalars.

\smallskip
(2)  Suppose that the coalgebra map $\gamma:C\longrightarrow D$ is also a monomorphism of vector spaces. Then, there is a  functor $(\alpha,\gamma)_\ast : \mathbf M_{\mathcal S}^D(\psi')\longrightarrow \mathbf M_{\mathcal R}^C(\psi)$ of restriction of scalars. Further, there is an adjunction of functors which is given by natural isomorphisms \begin{equation*}
\mathbf M_{\mathcal S}^D(\psi')((\alpha,\gamma)^*\mathcal M,\mathcal N)=\mathbf M_{\mathcal R}^C(\psi)(\mathcal M,(\alpha,\gamma)_*\mathcal N)
\end{equation*} for any $\mathcal M\in \mathbf M_{\mathcal R}^C(\psi)$ and $\mathcal N\in \mathbf M_{\mathcal S}^D(\psi')$.

\end{Thrm}

In Section 3, we give conditions for the category $\mathbf M^C_{\mathcal R}(\psi)$ of modules over an entwining structure $(\mathcal R,C,\psi)$ to have projective generators. We recall that a $K$-coalgebra $C$ is said to be right semiperfect if the category of right $C$-comodules has enough projectives. 

\begin{Thrm}\label{resultb} (see \ref{T3.5}) Let $(\mathcal R,C,\psi)$ be an entwining structure and let $C$ be a right semiperfect $K$-coalgebra. Then, the category $\mathbf M_{\mathcal R}^C(\psi)$ of entwined modules is a Grothendieck category with a set of projective generators.
\end{Thrm}

In Section 4, we fix a coalgebra $C$. We introduce the category $Mod^C-\mathscr R$ of modules over a representation $\mathscr R:\mathscr X\longrightarrow \mathscr Ent_C$, which is our main object of study. Our first purpose is to show that $Mod^C-\mathscr R$  is a Grothendieck category.

\begin{Thrm}\label{resultc} (see \ref{T4.9}) Let $C$ be a right semiperfect coalgebra over a field $K$. Let $\mathscr R:\mathscr X\longrightarrow\mathscr Ent_C$ be an entwined $C$-representation of a small category $\mathscr X$. Then, the category $Mod^C-\mathscr R$
of entwined modules over $\mathscr R$ is a Grothendieck category.

\end{Thrm}

Given $\mathscr R:\mathscr X\longrightarrow \mathscr Ent_C$, we have an entwining structure $(\mathscr R_x,C,\psi_x)$ for each $x\in \mathscr X$. Our next aim is to give conditions
for $Mod^C-\mathscr R$ to have projective generators. For this, we will construct an extension functor $ex_x^C$ and an evaluation functor $ev_x^C$ relating the categories
$Mod^C-\mathscr R$ and $\mathbf M_{\mathscr R_x}^C(\psi_x)$ at each $x\in \mathscr X$. 

\begin{Thrm}\label{resultd} (see \ref{P5.3} and \ref{T5.5})  Let $C$ be a right semiperfect coalgebra over a field $K$. Let $\mathscr X$ be a poset and let 
$\mathscr R:\mathscr X\longrightarrow \mathscr Ent_C$ be an entwined $C$-representation of $\mathscr X$.

\smallskip
(1) For each $x\in \mathscr X$, there is an extension functor $ex_x^C:\mathbf M_{\mathscr R_x}^C\longrightarrow Mod^C-\mathscr R$ which is left adjoint to an evaluation
functor $ev_x^C:Mod^C-\mathscr R\longrightarrow \mathbf M^C_{\mathscr R_x}(\psi_x)$. 

\smallskip
(2) The family $\{\mbox{$ex_x^C(V\otimes H_r)$ $\vert$ $x\in \mathscr X$,  $r\in \mathscr R_x$,  $V\in Proj^f(C)$}\}$  is a set of projective generators for $Mod^C-\mathscr R$, where $Proj^f(C)$ is the set of isomorphism classes of finite dimensional projective $C$-comodules. 

\end{Thrm}

We introduce the category of cartesian entwined modules in Section 6. Here, we will assume that $\mathscr X$ is a poset and $\mathscr R:\mathscr X\longrightarrow \mathscr Ent_C$ is a flat representation, i.e., for any morphism $\alpha:x\longrightarrow y$ in $\mathscr X$, the functor $\alpha^\ast:=\mathscr R_\alpha^\ast:\mathbf M^C_{\mathscr R_x}(\psi_x)
\longrightarrow \mathbf M^C_{\mathscr R_y}(\psi_y)$ is exact. We then apply  induction on $\mathbb N\times Mor(\mathscr X)$ to show that any cartesian entwined module may be expressed as a sum of submodules
whose cardinality is $\leq \kappa :=sup\{
\mbox{$|\mathbb N|$, $|C|$, $|K|$, $|Mor(\mathscr X)|$,  $|Mor(\mathscr R_x)|$, $x\in \mathscr X$}\}$.

\begin{Thrm}\label{resulte} (see \ref{T6.10}) Let $C$ be a right semiperfect coalgebra over a field $K$. Let $\mathscr X$ be a poset and let 
$\mathscr R:\mathscr X\longrightarrow \mathscr Ent_C$ be an entwined $C$-representation of $\mathscr X$. Suppose that $\mathscr R$
is  flat. Then, $Cart^C-\mathscr R$ is a Grothendieck category.

\end{Thrm} 

In the next three sections, we study separability and Frobenius conditions for functors relating $Mod^C-\mathscr R$ to the category $Mod-\mathscr R$ of modules over the underlying representation $\mathscr R:\mathscr X\longrightarrow \mathscr Ent_C\longrightarrow \mathscr Lin$. For this, we have to adapt the techniques from \cite{uni} as well as our earlier work 
in \cite{BBR}. For more on Frobenius and separability conditions for Doi-Hopf modules and  modules over entwining structures of algebras, we refer the reader to \cite{Brx5}, \cite{X13}, \cite{X14}, \cite{X15}. 

\smallskip
At each $x\in \mathscr X$, we have functors $\mathscr F_x:\mathbf M^C_{\mathscr R_x}(\psi_x)\longrightarrow \mathbf M_{\mathscr R_x}$ and $\mathscr G_x:\mathbf M_{\mathscr R_x}
\longrightarrow \mathbf M_{\mathscr R_x}^C(\psi_x)$ which combine to give functors $\mathscr F:Mod^C-\mathscr R\longrightarrow Mod-\mathscr R$ and $\mathscr G:
Mod-\mathscr R\longrightarrow Mod^C-\mathscr R$ respectively. We will also need to consider a space $V_1$ of elements $\theta=\{\theta_x(r):C\otimes C\longrightarrow \mathscr R_x(r,r)\}_{x\in \mathscr X,r\in \mathscr R_x}$  and a space $W_1$ of elements  $\eta=\{\eta_x(s,r):\mathscr R_x(s,r)\longrightarrow \mathscr R_x(s,r)\otimes C\}_{x\in \mathscr X,r,s\in \mathscr R_x}$ satisfying certain conditions (see Sections 7 and 8). 

\begin{Thrm}\label{resultf} (see \ref{P7.2}, \ref{P7.25} and \ref{P7.6}) Let $\mathscr X$ be a poset, $C$ be a right semiperfect $K$-coalgebra and $\mathscr R:\mathscr X\longrightarrow \mathscr Ent_C$
 be an entwined $C$-representation. 
 
 \smallskip
 (1) The forgetful functor $\mathscr F:Mod^C-\mathscr R\longrightarrow Mod-\mathscr R$ has a right adjoint $\mathscr G:
Mod-\mathscr R\longrightarrow Mod^C-\mathscr R$. 

\smallskip
(2)  A  natural transformation $\upsilon\in Nat(\mathscr G\mathscr F,1_{Mod^C-\mathscr R})$ corresponds to  a collection 
of natural transformations $\{\upsilon_x\in Nat(\mathscr G_x\mathscr F_x,1_{\mathbf M^C_{\mathscr R_x}(\psi_x)})\}_{x\in \mathscr X}$ such that for any $\alpha:x\longrightarrow y$ in
$\mathscr X$ and object $\mathscr M\in Mod^C-\mathscr R$, we have $\mathscr M_\alpha\circ \upsilon_x(\mathscr M_x)=\alpha_\ast\upsilon_y(\mathscr M_y)\circ \mathscr G_x\mathscr F_x(\mathscr M_\alpha) $. 

\smallskip
(3) The space  $Nat(\mathscr G\mathscr F,1_{Mod^C-\mathscr R})$ is isomorphic to $V_1$.
\end{Thrm}

The main results in Sections 7 and 8 give necessary and sufficient conditions for  the forgetful functor $\mathscr F:Mod^C-\mathscr R\longrightarrow Mod-\mathscr R$ and its right
adjoint $\mathscr G:
Mod-\mathscr R\longrightarrow Mod^C-\mathscr R$ to be separable. In Section 9, we give necessary and sufficient conditions for  $(\mathscr F,\mathscr G)$ to be a Frobenius pair, i.e., $\mathscr G$ is both a left and a right adjoint of $\mathscr F$. 

\begin{Thrm}\label{resultg} (see \ref{T7.7}, \ref{P7.8} and \ref{P7.9}) Let $\mathscr X$ be a partially ordered set.  Let $C$ be a right semiperfect $K$-coalgebra and let $\mathscr R:\mathscr X\longrightarrow 
\mathscr Ent_C$ be an entwined $C$-representation. 

\smallskip
(1) The functor $\mathscr F:Mod^C-\mathscr R\longrightarrow Mod-\mathscr R$ is separable if and only if there exists
$\theta\in V_1$ such that 
$
\theta_x(r)(c_1\otimes c_2)=\varepsilon_C(c)\cdot id_r
$ for every $x\in \mathscr X$, $r\in\mathscr R_x$ and $c\in C$. 

\smallskip
(2) Suppose additionally that the representation $\mathscr R:\mathscr X\longrightarrow \mathscr Ent_C$ is flat. Then, we have
 
 \smallskip
\begin{itemize}
\item[(a)]  The functor $\mathscr F:Mod^C-\mathscr R\longrightarrow Mod-\mathscr R$ restricts to a functor $\mathscr F^c:Cart^C-\mathscr R\longrightarrow Cart-\mathscr R$. Moreover, $\mathscr F^c$ has a right adjoint $\mathscr G^c:Cart-\mathscr R\longrightarrow Cart^C-\mathscr R$. 

\smallskip
\item[(b)]   Suppose there exists
$\theta\in V_1$ such that 
$
\theta_x(r)(c_1\otimes c_2)=\varepsilon_C(c)\cdot id_r
$ for every $x\in \mathscr X$, $r\in\mathscr R_x$ and $c\in C$. 
Then, $\mathscr F^c:Cart^C-\mathscr R\longrightarrow Cart-\mathscr R$ is separable.
\end{itemize}
\end{Thrm}

\begin{Thrm}\label{resulth} (see \ref{Pro8.2} and \ref{T8.3})  Let $\mathscr X$ be a partially ordered set, $C$ be a right semiperfect $K$-coalgebra and let $\mathscr R:\mathscr X\longrightarrow \mathscr Ent_C$ be an entwined $C$-representation.

\smallskip (1) The spaces $Nat(1_{Mod-\mathscr R},\mathscr F\mathscr G)$ and $ W_1$ are isomorphic.

\smallskip
(2) The functor
$\mathscr G:Mod-\mathscr R\longrightarrow Mod^C-\mathscr R$  is separable if and only if there exists $\eta\in W_1$ such that 
$
 id=(id\otimes \varepsilon_C)\circ \eta_x(s,r)
$
for each $x\in \mathscr X$ and $s$, $r\in \mathscr R_x$.

\end{Thrm}

\begin{Thrm}\label{resulti} (see \ref{T9.1}, \ref{P9.4}, \ref{C9.5})  Let $\mathscr X$ be a partially ordered set, $C$ be a right semiperfect $K$-coalgebra and let $\mathscr R:\mathscr X\longrightarrow \mathscr Ent_C$ be an entwined $C$-representation. 

\smallskip
(1)  $(\mathscr F,\mathscr G)$ is a Frobenius pair if and only if 
there exist $\theta\in V_1$ and $\eta\in W_1$ such that
$
\varepsilon_C(d)f=\sum \widehat{f}\circ \theta_x(r)(c_f\otimes d)$ and $\varepsilon_C(d)f=\sum \widehat{f_{\psi_x}} 
\circ \theta_x(r)(d^{\psi_x}\otimes c_f)
$ for every $x\in \mathscr X$, $r\in \mathscr R_x$, $f\in \mathscr R_x(r,s)$ and $d\in C$, where $\eta_x(r,s)(f)=\widehat{f}\otimes c_f$. 
 
\smallskip
(2) Suppose additionally that the representation $\mathscr R:\mathscr X\longrightarrow \mathscr Ent_C$ is flat. Then, $\mathscr G:Mod-\mathscr R\longrightarrow Mod^C-\mathscr R$ restricts to a functor
$\mathscr G^c:Cart-\mathscr R\longrightarrow Cart^C-\mathscr R$. 
Further, $(\mathscr F^c,\mathscr G^c)$ is a Frobenius pair of adjoint functors between $Cart^C-\mathscr R$ and
$Cart-\mathscr R$.

\end{Thrm}

We conclude in Section 10 by giving examples of how to construct entwined representations and describe modules over them. In particular, we show how to construct entwined representations using $B$-comodule categories, where $B$ is a bialgebra. 

\section{Category of entwining structures}

Let $K$ be a field and let $Vect_K$ be the category of vector spaces over $K$. Let $\mathcal R$ be a small $K$-linear category. The category of right $\mathcal R$-modules will be denoted by $\mathbf M_{\mathcal R}$. For any object $r\in \mathcal R$, we denote by $H_r:\mathcal R^{op}\longrightarrow 
Vect_K$ the right $\mathcal R$-module represented by $r$ and by $_rH:\mathcal R\longrightarrow Vect_K$ the left $\mathcal R$-module represented
by $r$. Given a $K$-coalgebra $C$, the category of right $C$-comodules will be denoted by $Comod-C$. 

\begin{defn}\label{entcatx} (see \cite[$\S$ 2]{BBR}) Let $\mathcal R$ be a small $K$-linear category and $C$ be a $K$-coalgebra. An entwining structure
$(\mathcal R,C,\psi)$ over $K$ is a collection of $K$-linear morphisms
\begin{equation*}
\psi=\{\psi_{rs}:C\otimes \mathcal R(r,s)\longrightarrow \mathcal R(r,s)\otimes C\}_{r,s\in \mathcal R}
\end{equation*} satisfying the following conditions
\begin{equation}
\begin{array}{c}
  (gf)_\psi \otimes c^\psi = g_\psi f_\psi \otimes {c^\psi}^\psi  \qquad 
 \varepsilon_C(c^\psi)(f_\psi) = \varepsilon_C(c)f \\
 f_\psi \otimes \Delta_C(c^\psi) =  {f_\psi}_\psi \otimes {c_1}^\psi \otimes {c_2}^\psi \qquad 
\psi(c \otimes id_r)= id_r \otimes c \\
\end{array}
\end{equation}
for each $f \in \mathcal{R}(r,s)$, $g \in \mathcal{R}(s,t)$ and $c \in C$. Here, we have  suppressed the summation and written $\psi(c\otimes f)$ simply as 
$f_\psi\otimes c^\psi$.

\smallskip
A morphism $(\alpha,\gamma):(\mathcal R,C,\psi)\longrightarrow (\mathcal S,D,\psi')$ of  entwining structures consists of a functor $\alpha:\mathcal{R}\longrightarrow \mathcal{S}$ and a counital coalgebra map $\gamma: C\longrightarrow D$ such that $\alpha({f}_{\psi})\otimes \gamma({c}^{\psi}) = \alpha(f)_{\psi'} \otimes \gamma(c)^{\psi'}$ for any $c\otimes f \in C\otimes  \mathcal{R}(r,s)$, where $r,s \in \mathcal R$. 

\smallskip We will denote by
$\mathscr Ent$ the category of entwining structures over $K$. 
\end{defn} 

If $\mathcal M$ is a right $\mathcal R$-module, $m\in \mathcal M(r)$ and $f\in \mathcal R(s,r)$, the element $\mathcal M(f)(r)\in \mathcal M(s)$
will often be denoted by $mf$. 

\smallskip
If $\alpha:\mathcal{R}\longrightarrow \mathcal{S}$ is a functor of small $K$-linear categories, there is an obvious functor $\alpha_*: \mathbf M_{\mathcal S}
\longrightarrow \mathbf M_{\mathcal R}$ of restriction of scalars. For the sake of convenience, we briefly recall here the well known extension of scalars 
$\alpha^*:\mathbf M_{\mathcal R}\longrightarrow \mathbf M_{\mathcal S}$. For $\mathcal M\in \mathbf M_{\mathcal R}$, the module $\alpha^*(\mathcal M)\in \mathbf M_{\mathcal S}$ is determined by setting
\begin{equation}\label{ke2.2}
\alpha^*(\mathcal M)(s):=\left(\underset{r\in \mathcal R}{\bigoplus}\mathcal M(r)\otimes \mathcal S(s,\alpha(r))\right)/V
\end{equation} for $s\in \mathcal S$, where $V$ is the subspace  generated by elements of the form
\begin{equation}
(m'\otimes \alpha(g)f)- (m'g\otimes f)
\end{equation} for $m'\in \mathcal M(r')$, $g\in \mathcal R(r,r')$, $f\in \mathcal S(s,\alpha(r))$ and $r$, $r'\in \mathcal R$.

\smallskip
On the other hand, if $\gamma : C\longrightarrow D$ is a morphism of coalgebras and $N$ is a right $C$-comodule, there is an obvious corestriction of scalars $\gamma^*:Comod-C\longrightarrow Comod-D$. The functor $\gamma^*$ has a well known right adjoint $\gamma_*:Comod-D\longrightarrow
Comod-C$, known as the coinduction functor, given by the cotensor product $N\mapsto N\Box_D C$ (see, for instance, \cite[$\S$ 11.10]{Wibook}). In general, we recall that the cotensor product $N\Box_DN'$ of a right $D$-comodule $(N,\rho:N\longrightarrow N\otimes D)$ with a left $D$-comodule $(N',\rho':N'\longrightarrow D\otimes N')$ is given by the equalizer
\begin{equation}
N\Box_DN':=Eq\left(N\otimes N'\doublerightarrow{\rho\otimes id}{id\otimes \rho'} N\otimes D\otimes N'\right)
\end{equation} In other words, an element $\sum n_i\otimes n'_i\in N\otimes N'$ lies in $N\Box_DN'$ if and only if $\sum n_{i0}\otimes n_{i1}\otimes 
n'_i=\sum n_i\otimes n'_{i0}\otimes n'_{i1}$.  However, we will continue to suppress the summation and write an element of $N\Box_DN'$ simply
as $n\otimes n'$. 
We will now consider modules over an entwining structure $(\mathcal R,C,\psi)$. 

\begin{defn}\label{D2.2} (see \cite[Definition 2.2]{BBR}) Let $\mathcal{M}$ be a right $\mathcal{R}$-module with a given right $C$-comodule structure $\rho_{\mathcal M(s)}:\mathcal M(s)\longrightarrow \mathcal M(s)\otimes C$ on $\mathcal{M}(s)$ for each $s\in \mathcal{R}$. Then, $\mathcal{M}$ is said to be an entwined module over  $(\mathcal{R},C,\psi)$ if the following compatibility condition holds:
\begin{equation}\label{comp 2}
 \rho_{\mathcal{M}(s)}(mf)= \big(mf\big)_0 \otimes \big(mf\big)_{1}=m_0f_\psi\otimes  {m_1}^\psi
\end{equation}
for every $f \in \mathcal{R}(s,r)$  and $m \in \mathcal{M}(r).$ 

\smallskip
A morphism $\eta:\mathcal M\longrightarrow \mathcal N$  of entwined modules is a morphism $\eta:\mathcal M\longrightarrow \mathcal N$ in 
$\mathbf M_{\mathcal R}$ such that $\eta(r):\mathcal M(r)\longrightarrow\mathcal N(r)$ is $C$-colinear for each $r\in\mathcal R$. The category of entwined modules over $(\mathcal R,C,\psi)$ will be denoted
by $\mathbf M_{\mathcal R}^C(\psi)$. 

\end{defn}

\begin{thm} \label{P2.2} Let  $(\alpha,\gamma):(\mathcal R,C,\psi)\longrightarrow (\mathcal S,D,\psi')$ be a morphism of entwining structures.  Then, there is a  functor $(\alpha,\gamma)^\ast : \mathbf M_{\mathcal R}^C(\psi)\longrightarrow \mathbf M_{\mathcal S}^D(\psi')$.
\end{thm}
\begin{proof}
We take $\mathcal M\in \mathbf M_{\mathcal R}^C(\psi)$. Then, $\mathcal M\in \mathbf M_{\mathcal R}$  and we consider $\mathcal N:=\alpha^*(\mathcal M)\in \mathbf M_{\mathcal S}$. For $s\in S$, we consider an element $m\otimes f\in \mathcal N(s)$, where $m\in \mathcal M(r)$ and $f\in \mathcal S(s,\alpha(r))$ for
some $r\in \mathcal R$. We claim that the morphism
\begin{equation}\label{eq2.6}
\rho_{\mathcal N(s)}:\mathcal N(s)\longrightarrow \mathcal N(s)\otimes D \qquad (m\otimes f)\mapsto (m\otimes f)_0\otimes (m\otimes f)_1:=(m_0\otimes f_{\psi'})\otimes \gamma(m_1)^{\psi'}
\end{equation}  makes $\mathcal N(s)$ a right $D$-comodule. Here, the association $m\mapsto m_0\otimes m_1$ comes from the $C$-comodule
structure $\rho_{\mathcal M(r)}:\mathcal M(r)\longrightarrow \mathcal M(r)\otimes C$ of $\mathcal M(r)$. 

\smallskip
First, we show that $\rho_{\mathcal N(s)}$ is well defined. For this, we consider $m'\in\mathcal M(r')$, $g\in \mathcal R(r,r')$ and $f\in \mathcal S(s,
\alpha(r))$.  We have
\begin{equation}
\begin{array}{ll}
(m'g\otimes f)_0\otimes (m'g\otimes f)_1 & = ((m'g)_0\otimes f_{\psi'})\otimes \gamma((m'g)_1)^{\psi'}\\
& = (m'_0g_\psi\otimes f_{\psi'})\otimes \gamma(m'^\psi_1)^{\psi'}\\
&= (m'_0\otimes \alpha(g_\psi)f_{\psi'})\otimes \gamma(m'^\psi_1)^{\psi'}\\
&= (m'_0\otimes \alpha(g)_{\psi'}f_{\psi'})\otimes \gamma(m'_1)^{\psi'\psi'}\\
&= (m'_0\otimes (\alpha(g)f)_{\psi'})\otimes \gamma(m'_1)^{\psi'}\\
\end{array}
\end{equation} From the properties of entwining structures, it may be easily verified that the structure maps in \eqref{eq2.6} are coassociative and counital, giving a right $D$-comodule
structure on $\mathcal N(s)$. We now consider $f'\in \mathcal S(s',s)$. Then, we have
\begin{equation}
\begin{array}{ll}
(m\otimes ff')_0\otimes (m\otimes ff')_1 & =(m_0\otimes (ff')_{\psi'})\otimes \gamma(m_1)^{\psi'}\\
&=(m_0\otimes f_{\psi'})f'_{\psi'}\otimes  \gamma(m_1)^{\psi'\psi'}\\
& = (m\otimes f)_0f'_{\psi'}\otimes (m\otimes f)_1^{\psi'}\\
\end{array}
\end{equation} This shows that $\mathcal N\in  \mathbf M_{\mathcal S}^D(\psi')$. 
\end{proof}

\begin{thm} \label{P2.3} Let  $(\alpha,\gamma):(\mathcal R,C,\psi)\longrightarrow (\mathcal S,D,\psi')$ be a morphism of entwining structures.  Suppose additionally that $\gamma:C\longrightarrow D$ is a monomorphism of vector spaces. Then, there is a  functor $(\alpha,\gamma)_\ast : \mathbf M_{\mathcal S}^D(\psi')\longrightarrow \mathbf M_{\mathcal R}^C(\psi)$.
\end{thm}

\begin{proof}
We take $\mathcal N\in \mathbf M_{\mathcal S}^D(\psi')$ and set $\mathcal M(r):=\mathcal N(\alpha(r))\Box_DC$ for each $r\in \mathcal R$. For $f\in 
\mathcal R(r',r)$, we define
\begin{equation}\label{eq2.9}
\mathcal M(f):\mathcal M(r)\longrightarrow \mathcal M(r')\qquad n\otimes c\mapsto (n\otimes c)\cdot f :=n\alpha(f_\psi)\otimes c^\psi
\end{equation} To show that this morphism is well defined, we need to check that $\mathcal M(f)(n\otimes c)\in \mathcal M(r')=\mathcal N(\alpha(r'))\Box_DC$. Since $n\otimes c\in \mathcal N(\alpha(r))\Box_DC$, we know that
\begin{equation}\label{eq2.10}
n_0\otimes n_1\otimes c= n\otimes \gamma(c_1)\otimes c_2
\end{equation} In particular, it follows that 
\begin{equation}\label{eq2.11}
n\otimes c\otimes f\in Eq\left(
\begin{CD}
\begin{tikzcd}
\mathcal N(\alpha(r))\otimes C\otimes \mathcal R(r',r) \ar[d,xshift = 5pt]\ar[d,xshift=-5pt]\\
\mathcal N(\alpha(r))\otimes D\otimes C\otimes \mathcal R(r',r) \\
\end{tikzcd}\\
@Vid\otimes id\otimes \psi VV\\
\mathcal N(\alpha(r))\otimes D\otimes \mathcal R(r',r)\otimes C \\
@Vid\otimes id \otimes \alpha \otimes id VV\\
\mathcal N(\alpha(r))\otimes D\otimes \mathcal S(\alpha(r'),\alpha(r))\otimes C \\
@Vid\otimes \psi'\otimes idVV\\
\mathcal N(\alpha(r))\otimes\mathcal S(\alpha(r'),\alpha(r))\otimes D\otimes C\\
@VVV\\
\mathcal N(\alpha(r'))\otimes D\otimes C\\
\end{CD}\right)
\end{equation} 
From \eqref{eq2.11}, it follows that
\begin{equation}\label{eq2.12}
n_0\alpha(f_\psi)_{\psi'}\otimes n_1^{\psi'}\otimes c^\psi=n\alpha(f_\psi)_{\psi'}\otimes \gamma(c_1)^{\psi'}\otimes {c_2}^{\psi}
\end{equation}
 Applying \eqref{eq2.10} and \eqref{eq2.12}, we now see that
\begin{equation}\label{eq2.13}
\begin{array}{ll}
(n\alpha(f_\psi))_0\otimes (n\alpha(f_\psi))_1 \otimes c^\psi & =n_0\alpha(f_\psi)_{\psi'}\otimes n_1^{\psi'}\otimes c^\psi \\
&=n\alpha(f_\psi)_{\psi'}\otimes \gamma(c_1)^{\psi'}\otimes {c_2}^{\psi}\\
&=n\alpha(f_{\psi\psi})\otimes \gamma({c_1}^{\psi})\otimes {c_2}^{\psi}\\
& =n\alpha(f_\psi)\otimes \gamma({c^{\psi}}_1)\otimes  {c^{\psi}}_2\\
\end{array}
\end{equation} From the definition, we may easily verify that the structure maps in \eqref{eq2.9} make $\mathcal M$ into a right $\mathcal R$-module. To show that $\mathcal M$ is entwined, it remains to check that
\begin{equation}\label{eq2.14}
n\alpha(f_\psi)\otimes {c^\psi}_1\otimes {c^\psi}_2=((n\otimes c)\cdot f)_0\otimes ((n\otimes c)\cdot f)_1= (n\otimes c)_0\cdot f_\psi\otimes (n\otimes c)_1^\psi=n\alpha(f_{\psi\psi})\otimes c_1^\psi\otimes c_2^\psi
\end{equation}  in $\mathcal N(\alpha(r'))\otimes C\otimes C$. Since $\gamma: C\longrightarrow D$ is a monomorphism and all tensor products are
taken over the field $K$, it suffices to show that
\begin{equation}\label{eq2.15}
n\alpha(f_\psi)\otimes \gamma({c^\psi}_1)\otimes {c^\psi}_2=n\alpha(f_{\psi\psi})\otimes \gamma( {c_1}^\psi)\otimes {c_2}^\psi\in \mathcal N(\alpha(r'))\otimes D\otimes C
\end{equation} Using \eqref{eq2.13} and the fact that $(\alpha,\gamma)$ is a morphism of entwining structures, the right hand side of
\eqref{eq2.15} becomes
\begin{equation}\label{eq2.16}
n\alpha(f_{\psi\psi})\otimes \gamma( {c_1}^\psi)\otimes {c_2}^\psi=n\alpha(f_\psi)_{\psi'}\otimes \gamma(c_1)^{\psi'}\otimes {c_2}^{\psi}=n_0\alpha(f_\psi)_{\psi'}\otimes n_1^{\psi'}\otimes c^\psi 
\end{equation} From \eqref{eq2.13}, we already know that $n\alpha(f_\psi)\otimes c^\psi\in \mathcal N(\alpha(r'))\Box_DC$. As such, we have
\begin{equation}\label{eq2.17}
n\alpha(f_\psi)\otimes \gamma({c^\psi}_1)\otimes {c^\psi}_2=(n\alpha(f_\psi))_0\otimes (n\alpha(f_\psi))_1\otimes c^\psi=n_0\alpha(f_\psi)_{\psi'}\otimes n_1^{\psi'}\otimes c^\psi
\end{equation} where the second equality follows from \eqref{eq2.13}. From \eqref{eq2.16} and \eqref{eq2.17}, the result of \eqref{eq2.15} is now clear. 
\end{proof}

\begin{Thm}\label{T2.5}   Let  $(\alpha,\gamma):(\mathcal R,C,\psi)\longrightarrow (\mathcal S,D,\psi')$ be a morphism of entwining structures such that   $\gamma:C\longrightarrow D$ is a monomorphism of vector spaces. Then, 
there is an adjuction of functors
\begin{equation}
\mathbf M_{\mathcal S}^D(\psi')((\alpha,\gamma)^*\mathcal M,\mathcal N)=\mathbf M_{\mathcal R}^C(\psi)(\mathcal M,(\alpha,\gamma)_*\mathcal N)
\end{equation} for $\mathcal M\in \mathbf M_{\mathcal R}^C(\psi)$ and $\mathcal N\in \mathbf M_{\mathcal S}^D(\psi')$.

\end{Thm}

\begin{proof} We consider a morphism $\eta :(\alpha,\gamma)^*\mathcal M\longrightarrow\mathcal N$ in $\mathbf M_{\mathcal S}^D(\psi')$. Then, $\eta$ corresponds
to a morphism $\eta:\alpha^*\mathcal M\longrightarrow \mathcal N$ in $\mathbf M_{\mathcal S}$ such that  $\eta(s):\alpha^*\mathcal M(s)\longrightarrow \mathcal N(s)$ is $D$-colinear for each $s\in S$. Accordingly, we have $\eta':\mathcal M\longrightarrow \mathcal N$ in $\mathbf M_{\mathcal R}$ such that
$\eta'(r):\mathcal M(r)\longrightarrow \mathcal N(\alpha(r))$ is $D$-colinear for each $r\in \mathcal R$. Here, $\mathcal M(r)$ is treated
as a $D$-comodule via corestriction of scalars. Therefore, we have morphisms $\eta''(r):\mathcal M(r)\longrightarrow \mathcal N(\alpha(r))\Box_DC$ 
of $C$-comodules for each $r\in \mathcal R$. Together, these determine a morphism $\mathcal M\longrightarrow (\alpha,\gamma)_*\mathcal N$ 
in $\mathbf M_{\mathcal R}$. These arguments can be easily reversed and hence the result. 
\end{proof}

\section{Projective generators and entwined modules}

Let $(\mathcal R,C,\psi)$ be an entwining structure. In \cite[Proposition 2.9]{BBR}, it was shown that the category 
$\mathbf M_{\mathcal R}^C(\psi)$ of entwined modules is a Grothendieck category. In this section, we will refine this result to give conditions for
$\mathbf M_{\mathcal R}^C(\psi)$ to have a collection of projective generators.

\begin{lem}\label{L3.1} Let $\mathcal G$ be a Grothendieck category. Fix a set of generators $\{ G_k\}_{k\in K}$ for $\mathcal G$. Let $Z\in \mathcal G$ be an object. Let $i_X:X\hookrightarrow Z$, $i_Y:Y\hookrightarrow Z$ be two subobjects  of $Z$ such that for any $k\in K$ and any morphism $f_k:G_k\longrightarrow X$, there exists $g_k:G_k\longrightarrow Y$ such that $i_Y\circ g_k=i_X\circ f_k$. Then, $i_X:X\hookrightarrow Z$ factors through $i_Y:Y\hookrightarrow Z$, i.e., $X$ is a subobject of $Y$.
\end{lem}

\begin{proof}
Since  $\{ G_k\}_{k\in K}$ is a set of generators for $\mathcal G$, we can choose (see \cite[Proposition 1.9.1]{Tohoku}) an epimorphism
$f:\underset{j\in J}{\bigoplus}\textrm{ }G_j\longrightarrow X$, corresponding to a collection of maps
$f_j:G_j\longrightarrow X$, with each $G_j$ a generator from the collection $\{ G_k\}_{k\in K}$. Accordingly, we can choose morphisms
$g_j:G_j\longrightarrow Y$ such that  $i_Y\circ g_j=i_X\circ f_j$ for each $j\in J$. Together, these $\{g_j\}_{j\in J}$ determine
a morphism $g:\underset{j\in J}{\bigoplus}\textrm{ }G_j\longrightarrow Y$ satisfying $i_Y\circ g=i_X\circ f$. Since $i_X$, $i_Y$ are monomorphisms and $f$ is an epimorphism, we have
\begin{equation}
X=Im(i_X)=Im(i_X\circ f)=Im(i_Y\circ g)=Im(i_Y|Im(g))\subseteq Im(i_Y)=Y
\end{equation}
\end{proof}

\begin{lem}\label{L3.2} Let $\mathcal G$ be a Grothendieck category having a set of projective generators $\{ G_k\}_{k\in K}$. Let $f:X\longrightarrow Y$ be a morphism in $\mathcal G$. Let $i:X'\hookrightarrow X$ and $j:Y'\hookrightarrow Y$ be monomorphisms. Suppose that for any $k\in K$ and 
any morphism $f_k:G_k\longrightarrow X'$, there exists a morphism $g_k:G_k\longrightarrow Y'$ such that $f\circ i\circ f_k=j\circ g_k:G_k
\longrightarrow Y$. Then, there exists $f':X'\longrightarrow Y'$ such that $j\circ f'=f\circ i$. 
\end{lem}

\begin{proof}
It suffices to show that $Im(f\circ i)\subseteq Y'$. We choose any $k\in K$ and a morphism $h_k:G_k\longrightarrow Im(f\circ i)\hookrightarrow Y$. Since $G_k$ is projective, we can choose $f_k:G_k\longrightarrow X'$ such that $f\circ i\circ f_k=h_k$.  By assumption, we can now find $g_k:G_k\longrightarrow Y'$ such that $f\circ i\circ f_k=j\circ g_k:G_k
\longrightarrow Y$.  In particular, $j\circ g_k=h_k$. Applying Lemma \ref{L3.1}, we obtain $Im(f\circ i)\subseteq Y'$.
\end{proof}

\begin{lem}\label{L3.3} Let $(\mathcal R,C,\psi)$ be an entwining structure. Let $V$ be a right $C$-comodule. Then, for any $r\in \mathcal R$, the module
$V\otimes H_r$ given by 
\begin{equation}
\begin{array}{c}
(V\otimes H_r)(r')=V\otimes \mathcal R(r',r)\\
(V\otimes H_r)(f):(V\otimes H_r)(r')\longrightarrow (V\otimes H_r)(r'')\qquad v\otimes g\mapsto v\otimes gf
\end{array}
\end{equation} for $r'\in \mathcal R$, $f\in \mathcal R(r'',r')$ is an entwined module in $\mathbf M^C_{\mathcal R}(\psi)$. Here, the right $C$-comodule structure on
$(V\otimes H_r)(r')$ is given by taking $v\otimes g$ to $v_0\otimes g_\psi\otimes v^\psi_1$. 

\end{lem}

\begin{proof}
See \cite[Lemma 2.5]{BBR}. 
\end{proof}

For the rest of this section, we will assume that the coalgebra $C$ is such that the category $Comod-C$ of right $C$-comodules has enough projective objects. In other words, the coalgebra $C$ is right semiperfect (see \cite[Definition 3.2.4]{book3}). 

\begin{thm}\label{P3.4} Let $(\mathcal R,C,\psi)$ be an entwining structure with $C$ a right semiperfect coalgebra. Let $V$ be a projective right $C$-comodule. Then, for any $r\in \mathcal R$, the module
$V\otimes H_r$ is a projective object of $\mathbf M_{\mathcal R}^C(\psi)$.
\end{thm}

\begin{proof}
We begin with  a morphism $\zeta: V\otimes 
H_r\longrightarrow \mathcal M$ and an epimorphism $\eta:\mathcal N\longrightarrow \mathcal M$ in $\mathbf M_{\mathcal R}^C(\psi)$. In particular, we consider the composition
\begin{equation}\label{eq3.3e}
V\longrightarrow V\otimes H_r(r)\longrightarrow \mathcal M(r) \qquad v\mapsto v\otimes id_r\mapsto \zeta(r)(v\otimes id_r)
\end{equation} which is a morphism in $Comod-C$. Since $V$ is projective, we can lift the map in \eqref{eq3.3e} to a map $T:V\longrightarrow \mathcal N(r)$ in $Comod-C$ such that
$(\eta(r)(T(v))=\zeta(r)(v\otimes id_r)$ for each $v\in V$. 

\smallskip
We now define $\xi:V\otimes H_r\longrightarrow \mathcal N$ by setting for each $s\in \mathcal R$
\begin{equation}\label{eq3.4e}
\xi(s):V\otimes H_r(s)\longrightarrow \mathcal N(s)\qquad v\otimes g\mapsto \mathcal N(g)(T(v))
\end{equation}  We first check that $\xi:V\otimes H_r\longrightarrow \mathcal N$ is a morphism in $\mathbf M_{\mathcal R}$. Given $g'\in \mathcal R(s',s)$, we have
\begin{equation}\label{pf3.5}
\mathcal N(g')(\xi(s)(v\otimes g))=\mathcal N(gg')(T(v))=\xi(s')(v\otimes gg')=\xi(s')((V\otimes H_r)(g')(v\otimes g))
\end{equation} We also have, for $v\otimes g\in V\otimes H_r(s)$,
\begin{equation}
\begin{array}{ll}
\xi(s)(v\otimes g)_0\otimes \xi(s)(v\otimes g)_1 & =\mathcal N(g)(T(v))_0\otimes \mathcal N(g)(T(v))_1\\
& = T(v)_0g_\psi\otimes T(v)_1^\psi \\
&= T(v_0)g_\psi \otimes v_1^\psi\\
& =\mathcal N(g_\psi)(T(v_0))\otimes v_1^\psi=(\xi(s)\otimes id_C)(v_0\otimes g_\psi\otimes v_1^\psi)\\
\end{array}
\end{equation} This shows that $\xi(s):V\otimes H_r(s)\longrightarrow \mathcal N(s)$ is a morphism in $Comod-C$.  Together with \eqref{pf3.5}, it follows that $\xi:V\otimes H_r\longrightarrow \mathcal N$ is a morphism in $\mathbf M_{\mathcal R}^C(\psi)$.  Finally, we see that for $v\otimes g\in V\otimes H_r(s)$, we have
\begin{equation}
\begin{array}{ll}
(\eta(s)\circ \xi(s))(v\otimes g)& =\eta(s)(\mathcal N(g)(T(v)))\\
&= \mathcal M(g)(\eta(r)(T(v)))\\
& = \mathcal M(g)(\zeta(r)(v\otimes id_r))\\
& =\zeta(s)((V\otimes H_r)(g)(v\otimes id_r))\\
&= \zeta(s)(v\otimes g)\\
\end{array}
\end{equation} This gives us $\eta\circ \xi=\zeta:V\otimes H_r\longrightarrow \mathcal M$. Hence the result.
\end{proof}

\begin{Thm}\label{T3.5} Let $(\mathcal R,C,\psi)$ be an entwining structure and let $C$ be a right semiperfect $K$-coalgebra. Then, the category $\mathbf M_{\mathcal R}^C(\psi)$ of entwined modules is a Grothendieck category with a set of projective generators.
\end{Thm}

\begin{proof}
From \cite[Proposition 2.9]{BBR}, we know that $\mathbf M_{\mathcal R}^C(\psi)$ is a Grothendieck category.  Let $\mathcal M$ be an object of $\mathbf M_{\mathcal R}^C(\psi)$. From the proof of \cite[Proposition 2.9]{BBR}, we know that there exists an epimorphism
\begin{equation}
\eta' : \underset{i\in I}{\bigoplus}V'_i\otimes H_{r_i}\longrightarrow \mathcal M
\end{equation} where each $r_i\in \mathcal R$ and each $V'_i$ is a finite dimensional $C$-comodule. Since $Comod-C$ has enough projectives, it follows from \cite[Corollary 2.4.21]{book3} that we can choose for each $V'_i$ an epimorphism $V_i\longrightarrow V'_i$ in $Comod-C$ such that $V_i$ is a finite dimensional projective in $Comod-C$.  This induces an epimorphism
\begin{equation}
\eta : \underset{i\in I}{\bigoplus}V_i\otimes H_{r_i}\longrightarrow \mathcal M
\end{equation} The collection $\{V\otimes H_r\}$ now gives a set of projective generators for $\mathbf M_{\mathcal R}^C(\psi)$, where
$r\in \mathcal R$ and $V$ ranges over (isomorphism classes of) finite dimensional projective $C$-comodules. 
\end{proof}

\section{Modules over an entwined representation}

We fix a $K$-coalgebra $C$ which is right semiperfect. We consider the category $\mathscr Ent_C$ whose objects are entwining structures 
$(\mathcal R,C,\psi)$. A morphism in $\mathscr Ent_C$ is a map $(\alpha,id):(\mathcal R,C,\psi)\longrightarrow (\mathcal R',C,\psi')$ of entwining structures, which we will denote simply by $\alpha$. From Section 2, it follows that we have adjoint functors 
\begin{equation}\label{eq4.1}
\begin{array}{c}
\alpha^\ast=(\alpha,id_C)^\ast: \mathbf M^C_{\mathcal R}(\psi)\longrightarrow \mathbf M^C_{\mathcal R'}(\psi')\qquad 
\alpha_\ast=(\alpha,id_C)_\ast :\mathbf M^C_{\mathcal R'}(\psi')\longrightarrow \mathbf M^C_{\mathcal R}(\psi)\\ 
\end{array} 
\end{equation} We note in particular that the functors $\alpha_\ast=(\alpha,id_C)_\ast$ are exact. In fact, the functors $\alpha_\ast$ preserve both limits and colimits. 

\smallskip
\begin{defn}\label{D4.1}
Let $\mathscr X$ be a small category. Let $C$ be a right semiperfect coalgebra over the field $K$. By an entwined $C$-representation of a small category, we will mean a functor $\mathscr R:\mathscr X\longrightarrow \mathscr Ent_C$.

\smallskip In particular, for each object $x\in \mathscr X$, we have an entwining structure
$(\mathscr R_x,C,\psi_x)$. Given a morphism $\alpha : x\longrightarrow y$ in $\mathscr X$, we have a morphism $\mathscr R_\alpha=(\mathscr R_\alpha,id_C):(\mathscr R_x,C,\psi_x)
\longrightarrow (\mathscr R_y,C,\psi_y)$ of entwining structures. 
\end{defn}

By abuse of notation, if  $\mathscr R:\mathscr X\longrightarrow \mathscr Ent_C$ is an entwined $C$-representation, we will write
\begin{equation}
\alpha^\ast=\mathscr R_\alpha^\ast:  \mathbf M_{\mathscr R_x}^C(\psi_x)\longrightarrow  \mathbf M_{\mathscr R_y}^C(\psi_y)
\qquad \alpha_\ast=\mathscr R_{\alpha\ast}: \mathbf M_{\mathscr R_y}^C(\psi_y)\longrightarrow  \mathbf M_{\mathscr R_x}^C(\psi_x)
\end{equation} for any morphism $\alpha:x\longrightarrow y$ in $\mathscr X$. Also by abuse of notation, if $f:r'\longrightarrow r$ is a morphism in $\mathscr R_x$, we will often denote
$\mathscr R_\alpha(f):\mathscr R_\alpha(r')\longrightarrow \mathscr R_\alpha(r)$ in $\mathscr R_y$ simply as $\alpha(f):\alpha(r')\longrightarrow \alpha(r)$. 
We will now consider modules over an entwined $C$-representation.

\begin{defn}\label{D4.2} Let  $\mathscr R:\mathscr X\longrightarrow \mathscr Ent_C$ be an entwined $C$-representation of a small category
$\mathscr X$. An entwined module $\mathscr M$ over $\mathscr R$ will consist of the following data 

\smallskip
(1) For each object $x\in \mathscr X$, an entwined module $\mathscr M_x\in \mathbf M_{\mathscr R_x}^C(\psi_x)$. 

\smallskip
(2) For each morphism $\alpha : x\longrightarrow y$ in $\mathscr X$, a morphism $\mathscr M_\alpha: \mathscr M_x\longrightarrow \alpha_\ast\mathscr M_y$ in $\mathbf M_{\mathscr R_x}^C(\psi_x)$ (equivalently, a morphism $\mathscr M^\alpha: \alpha^\ast\mathscr M_x\longrightarrow \mathscr M_y$ in $\mathbf M_{\mathscr R_y}^C(\psi_y)$).

\smallskip
Further, we suppose that $\mathscr M_{id_x}=id_{\mathscr M_x}$ for each $x\in \mathscr X$ and that for any  composable
morphisms $x\overset{\alpha}{\longrightarrow}y\overset{\beta}{\longrightarrow}z$, we have $\alpha_\ast(\mathscr M_\beta)\circ \mathscr M_\alpha=\mathscr M_{\beta\alpha}:\mathscr M_x\longrightarrow \alpha_\ast\mathscr M_y\longrightarrow \alpha_\ast\beta_\ast\mathscr M_z=(\beta\alpha)_\ast\mathscr M_z$. The latter condition may be expressed in any of two equivalent ways
\begin{equation}\label{4.25d}
\mathscr M_{\beta\alpha}=\alpha_\ast(\mathscr M_\beta)\circ \mathscr M_\alpha\qquad\Leftrightarrow\qquad\mathscr M^{\beta\alpha}=\mathscr M^\beta\circ \beta^\ast(\mathscr M^\alpha)
\end{equation}

\smallskip
A morphism $\eta:\mathscr M\longrightarrow \mathscr N$ of entwined modules over $\mathscr R$ consists of morphisms $\eta_x:\mathscr M_x\longrightarrow
\mathscr N_x$ in each $\mathbf M^C_{\mathscr R_x}(\psi_x)$ such that the following diagram commutes
\begin{equation}
\begin{CD}
\mathscr M_x @>\eta_x>>
\mathscr N_x\\
@V\mathscr M_\alpha VV @VV\mathscr N_\alpha V \\
\alpha_\ast\mathscr M_y @>\alpha_\ast\eta_y>> \alpha_\ast\mathscr N_y \\
\end{CD}
\end{equation} for each $\alpha:x\longrightarrow y$ in $\mathscr R$. The category of entwined modules over $\mathscr R$ will be denoted by
$Mod^C-\mathscr R$.
\end{defn}

\begin{thm}\label{P4.3}   Let  $\mathscr R:\mathscr X\longrightarrow \mathscr Ent_C$ be an entwined $C$-representation of a small category
$\mathscr X$. Then, $Mod^C-\mathscr R$ is an abelian category. 
\end{thm}

\begin{proof}
Let $\eta:\mathscr M\longrightarrow \mathscr N$ be a morphism $Mod^C-\mathscr R$. We define the kernel and cokernel of $\eta$ by setting
\begin{equation}
Ker(\eta)_x:=Ker(\eta_x:\mathscr M_x\longrightarrow \mathscr N_x)\qquad Cok(\eta)_x:=Cok(\eta_x:\mathscr M_x\longrightarrow \mathscr N_x)
\end{equation}
for each $x\in \mathscr X$. For $\alpha:x\longrightarrow y$ in $\mathscr X$, the morphisms $Ker(\eta)_\alpha$ and $Cok(\eta)_\alpha$ are induced in the obvious manner, using the fact that $\alpha_\ast:\mathbf M^C_{\mathscr R_y}(\psi_y)\longrightarrow \mathbf M_{\mathscr R_x}^C(\psi_x)$ is exact. From this, it is also clear that $Cok(Ker(\eta)\hookrightarrow \mathscr M)=Ker(\mathscr N\twoheadrightarrow Cok(\eta))$. 
\end{proof}

We now let $\mathscr R:\mathscr X\longrightarrow \mathscr Ent_C$ be an entwined $C$-representation of a small category
$\mathscr X$ and let $\mathscr M$ be an entwined module over $\mathscr R$. We consider some $x\in \mathscr X$ and a morphism
\begin{equation}
\eta: V\otimes H_r\longrightarrow \mathscr M_x
\end{equation} in $\mathbf M_{\mathscr R_x}^C(\psi_x)$,  where $V$ is a finite dimensional projective in $Comod-C$ and $r\in \mathscr R_x$. For each $y\in \mathscr X$, we now set $\mathscr N_y\subseteq \mathscr M_y$ to be the image of the family of maps
\begin{equation}\label{eq4.6}
\begin{array}{ll}
\mathscr N_y&=Im\left(\underset{\beta\in \mathscr X(x,y)}{\bigoplus}\textrm{ }\begin{CD}\beta^\ast (V\otimes H_r)@>\bigoplus \beta^\ast\eta>>\beta^\ast \mathscr M_x@>\mathscr M^\beta>>\mathscr M_y\end{CD}\right)\\
&=\underset{\beta\in \mathscr X(x,y)}{\sum}\textrm{ }Im\left(\begin{CD}\beta^\ast (V\otimes H_r)@>\beta^\ast\eta>>\beta^\ast \mathscr M_x@>\mathscr M^\beta>>\mathscr M_y\end{CD}\right)\\
\end{array}
\end{equation} We denote by $\iota_y$ the inclusion $\iota_y:\mathscr N_y\hookrightarrow\mathscr M_y$. For each $\beta\in \mathscr X(x,y)$, we denote by $\eta'_\beta:\beta^\ast (V\otimes H_r)
\longrightarrow \mathscr N_y$ the canonical morphism induced from \eqref{eq4.6}. 

\begin{lem}\label{L4.4} For any $\alpha\in \mathscr X(y,z)$, $\beta\in \mathscr X(x,y)$, the following composition 
\begin{equation}
\begin{CD}
\beta^\ast (V\otimes H_r)@>\eta'_\beta>> \mathscr N_y @>\iota_y>> \mathscr M_y @>\mathscr M_\alpha>> \alpha_\ast\mathscr M_z
\end{CD}
\end{equation}
factors through $\alpha_\ast(\iota_z):\alpha_\ast\mathscr N_z\longrightarrow\alpha_\ast\mathscr M_z$.
\end{lem} 

\begin{proof}
Since $(\alpha^\ast,\alpha_\ast)$ is an adjoint pair, it suffices to show that the composition
\begin{equation}
\begin{CD}
\alpha^\ast\beta^\ast (V\otimes H_r)@>\alpha^\ast(\eta'_\beta)>> \alpha^\ast\mathscr N_y @>\alpha^\ast(\iota_y)>> \alpha^\ast\mathscr M_y @>\mathscr M^\alpha>> \mathscr M_z
\end{CD} 
\end{equation}factors through $\iota_z:\mathscr N_z\longrightarrow \mathscr M_z$. By definition, we know that the composition $\beta^\ast (V\otimes H_r)\xrightarrow{\eta'_\beta}\mathscr N_y\xrightarrow{\iota_y}\mathscr M_y$ factors through $\beta^\ast\mathscr M_x$, i.e., we have
\begin{equation}
\iota_y\circ \eta'_\beta=\mathscr M^\beta\circ \beta^\ast\eta
\end{equation}
Applying $\alpha^\ast$, composing with $\mathscr M^\alpha$ and using \eqref{4.25d}, we get
\begin{equation}
\mathscr M^\alpha\circ \alpha^\ast(\iota_y)\circ \alpha^\ast(\eta'_\beta)=\mathscr M^\alpha\circ \alpha^\ast(\mathscr M^\beta)\circ \alpha^\ast(\beta^\ast\eta)=\mathscr M^{\alpha\beta}\circ \alpha^\ast\beta^\ast\eta
\end{equation} From the definition in \eqref{eq4.6}, it is now clear that the composition $\mathscr M^\alpha\circ \alpha^\ast(\iota_y)\circ \alpha^\ast(\eta'_\beta)=\mathscr M^{\alpha\beta}\circ \alpha^\ast\beta^\ast\eta$ factors  through $\iota_z:\mathscr N_z\longrightarrow \mathscr M_z$ as $\mathscr M^\alpha\circ \alpha^\ast(\iota_y)\circ \alpha^\ast(\eta'_\beta)=\iota_z\circ \eta'_{\alpha\beta}$. 
\end{proof}

\begin{thm}\label{P4.5}
For any $\alpha\in \mathscr X(y,z)$, the morphism $\mathscr M_\alpha:\mathscr M_y\longrightarrow \alpha_\ast\mathscr M_z$ restricts to a morphism
 $\mathscr N_\alpha:\mathscr N_y\longrightarrow \alpha_\ast\mathscr N_z$, giving us a commutative diagram
 \begin{equation}\label{cd4}
 \begin{CD}
 \mathscr M_y @>\mathscr M_\alpha>>\alpha_\ast\mathscr M_z\\
 @A\iota_yAA @AA\alpha_\ast(\iota_z)A \\
 \mathscr N_y @>\mathscr N_\alpha >> \alpha_\ast\mathscr N_z\\
 \end{CD}
 \end{equation}
\end{thm}
\begin{proof}
We already know that $\iota_z:\mathscr N_z\longrightarrow \mathscr M_z$ is a monomorphism. Since $\alpha_\ast$ is a right adjoint, it follows that
$\alpha_\ast(\iota_z)$ is also a monomorphism. Since $C$ is right semiperfect, we know from Theorem \ref{T3.5} that $\mathbf M^C_{\mathscr R_y}(\psi_y)$ is a Grothendieck category
with projective generators $\{G_k\}_{k\in K}$.Using Lemma \ref{L3.2}, it suffices to show that for any $k\in K$ and any morphism $\xi_k : G_k
\longrightarrow \mathscr N_y$, there exists $\xi'_k:G_k\longrightarrow \alpha_\ast\mathscr N_z$ such that $\alpha_\ast(\iota_z)\circ \xi'_k=\mathscr M_\alpha\circ \iota_y\circ \xi_k$. 

\smallskip
From \eqref{eq4.6}, we have an epimorphism 
\begin{equation}\underset{\beta\in \mathscr X(x,y)}{\bigoplus}\textrm{ }\eta'_\beta:\underset{\beta\in \mathscr X(x,y)}{\bigoplus}\textrm{ }\beta^\ast(V\otimes H_r)\longrightarrow \mathscr N_y
\end{equation} Since $G_k$ is projective, we can lift $\xi_k : G_k
\longrightarrow \mathscr N_y$ to a morphism $\xi''_k:G_k\longrightarrow \underset{\beta\in \mathscr X(x,y)}{\bigoplus}\textrm{ }\beta^\ast(V\otimes H_r)$ such that 
\begin{equation}\xi_k=\left(\underset{\beta\in \mathscr X(x,y)}{\bigoplus}\textrm{ }\eta'_\beta\right)\circ \xi''_k
\end{equation} From  Lemma \ref{L4.4}, we know that  $\mathscr M_\alpha\circ \iota_y\circ \eta'_\beta$ factors through
$\alpha_\ast(\iota_z):\alpha_\ast\mathscr N_z\longrightarrow\alpha_\ast\mathscr M_z$ for each $\beta\in \mathscr X(x,y)$. The result is now clear.
\end{proof}

Using the adjointness of $(\alpha^\ast,\alpha_\ast)$, we can also obtain a morphism $\mathscr N^\alpha:\alpha^\ast\mathscr N_y\longrightarrow \mathscr N_z$ for each $\alpha\in \mathscr X(y,z)$, corresponding to the morphism $\mathscr N_\alpha:\mathscr N_y\longrightarrow\alpha_\ast\mathscr N_z$ in \eqref{cd4}. The objects $\{\mathscr N_y\in \mathbf M_{\mathscr R_y}^C(\psi_y)\}_{y\in \mathscr X}$, together with the morphisms $\{\mathscr N_\alpha\}_{\alpha\in Mor(\mathscr X)}$ determine an object of $Mod^C-\mathscr R$ that we denote by $\mathscr N$. Additionally, Proposition \ref{P4.5} shows that we have an inclusion $\iota:\mathscr N\hookrightarrow \mathscr M$ in $Mod^C-\mathscr R$.
Before we proceed further, we will describe the object $\mathscr N$ in a few more ways.

\begin{lem}\label{P4.6}
Let $\eta'_1:V\otimes H_r\longrightarrow \mathscr N_x$ be the canonical morphism corresponding to the identity map in $\mathscr X(x,x)$. Then, for any $y\in \mathscr X$, we have
\begin{equation}
\mathscr N_y=Im\left(\underset{\beta\in \mathscr X(x,y)}{\bigoplus}\textrm{ }\begin{CD}\beta^\ast (V\otimes H_r)@>\bigoplus\beta^\ast\eta'_1>>\beta^\ast \mathscr N_x@>\mathscr N^\beta>>\mathscr N_y\end{CD}\right)
\end{equation}
\end{lem}

\begin{proof}
For any $\beta\in\mathscr X(x,y)$, we consider the commutative diagram
\begin{equation}\label{cd4.16}
\begin{CD}
\beta^\ast(V\otimes H_r)@>\beta^\ast\eta'_1>>\beta^\ast\mathscr N_x @>\mathscr N^\beta>>\mathscr N_y\\
@. @V\beta^\ast(\iota_x)VV @VV\iota_yV\\
@. \beta^\ast\mathscr M_x @>\mathscr M^\beta>> \mathscr M_y\\
\end{CD}
\end{equation} By definition, we know that $\iota_x\circ \eta'_1=\eta$, which gives $\beta^\ast(\iota_x)\circ \beta^\ast(\eta'_1)=\beta^\ast(\eta)$. Composing with $\mathscr M^\beta$, we get
\begin{equation}
Im(\mathscr M^\beta\circ \beta^\ast(\eta))=Im(\mathscr M^\beta\circ \beta^\ast(\iota_x)\circ \beta^\ast(\eta'_1))=Im(\iota_y\circ \mathscr N^\beta\circ \beta^\ast\eta'_1)\cong Im(\mathscr N^\beta\circ \beta^\ast\eta'_1)
\end{equation} where the last isomorphism follows from the fact that $\iota_y$ is monic. The result is now clear from the definition in \eqref{eq4.6}.
\end{proof}

\begin{lem}\label{L4.7} For any $y\in \mathscr X$, we have
\begin{equation}\label{ny}
\mathscr N_y=\underset{\beta\in \mathscr X(x,y)}{\sum}\textrm{ }Im\left(\begin{CD}\beta^\ast \mathscr N_x@>\beta^\ast(\iota_x)>>\beta^\ast \mathscr M_x@>\mathscr M^\beta>>\mathscr M_y\end{CD}\right)
\end{equation} 
\end{lem}
\begin{proof} For the sake of convenience, we set 
\begin{equation*}
\mathscr N'_y:=\underset{\beta\in \mathscr X(x,y)}{\sum}\textrm{ }Im\left(\begin{CD}\beta^\ast \mathscr N_x@>\beta^\ast(\iota_x)>>\beta^\ast \mathscr M_x@>\mathscr M^\beta>>\mathscr M_y\end{CD}\right)
\end{equation*}
 From the commutative diagram in \eqref{cd4.16}, we see that each of the morphisms $\begin{CD}\beta^\ast \mathscr N_x@>\beta^\ast(\iota_x)>>\beta^\ast \mathscr M_x@>\mathscr M^\beta>>\mathscr M_y\end{CD}$ factors through the subobject $\mathscr N_y\subseteq \mathscr M_y$. Hence,  $\mathscr N_y'\subseteq \mathscr N_y$. On the other hand, it is clear that
\begin{equation*}
Im\left(\begin{CD}\beta^\ast(V\otimes H_r)@>\beta^\ast\eta_1'>>\beta^\ast \mathscr N_x@>\beta^\ast(\iota_x)>>\beta^\ast \mathscr M_x@>\mathscr M^\beta>>\mathscr M_y\end{CD}\right)\subseteq Im\left(\begin{CD}\beta^\ast \mathscr N_x@>\beta^\ast(\iota_x)>>\beta^\ast \mathscr M_x@>\mathscr M^\beta>>\mathscr M_y\end{CD}\right)
\end{equation*} Applying Lemma \ref{P4.6}, it is now clear that $\mathscr N_y\subseteq \mathscr N_y'$. This proves the result.
\end{proof}

We now make a few conventions : if $\mathcal M$ is a module over a small $K$-linear category $\mathcal R$, we denote by $el(\mathcal M)$ the union $\underset{r\in \mathcal R}{\bigcup}\textrm{ }\mathcal M(r)$. The cardinality of $el(\mathcal M)$ will be denoted by $|\mathcal M|$. If $\mathscr M$ is a module
over an entwined $C$-representation $\mathscr R:\mathscr X\longrightarrow \mathscr Ent_C$, we denote by $el_{\mathscr X}(\mathscr M)$ the union $\underset{x\in 
\mathscr X}{\bigcup}\textrm{ }el(\mathscr M_x)$. The cardinality of $el_{\mathscr X}(\mathscr M)$ will be denoted by $|\mathscr M|$. It is evident that if $\mathscr M\in Mod^C-\mathscr R$ and $\mathscr N$ is either a quotient or a subobject of $\mathscr M$, then $|\mathscr N|\leq |\mathscr M|$.  

\smallskip
We now define the following cardinality
\begin{equation}
\kappa =sup\{
\mbox{$|\mathbb N|$, $|C|$, $|K|$, $|Mor(\mathscr X)|$,  $|Mor(\mathscr R_x)|$, $x\in \mathscr X$}\}
\end{equation} 
We observe that  $|\beta^\ast(V\otimes H_r)|\leq\kappa$, where $V$ is any finite dimensional $C$-comodule and $\beta\in \mathscr X(x,y)$. 

\begin{lem}\label{L4.8}
We have $|\mathscr N|\leq \kappa$.
\end{lem}

\begin{proof}
We choose $y\in \mathscr X$. From Lemma \ref{P4.6}, we have
\begin{equation}
\mathscr N_y=Im\left(\underset{\beta\in \mathscr X(x,y)}{\bigoplus}\textrm{ }\begin{CD}\beta^\ast (V\otimes H_r)@>\bigoplus\beta^\ast\eta'_1>>\beta^\ast \mathscr N_x@>\mathscr N^\beta>>\mathscr N_y\end{CD}\right)
\end{equation} Since $\mathscr N_y$ is an epimorphic image of $\underset{\beta\in \mathscr X(x,y)}{\bigoplus}\textrm{ }\beta^\ast (V\otimes H_r)$, we have
\begin{equation}|\mathscr N_y|\leq | \underset{\beta\in \mathscr X(x,y)}{\bigoplus}\textrm{ }\beta^\ast (V\otimes H_r)|\leq \kappa
\end{equation} It follows that
$
|\mathscr N|=\underset{y\in \mathscr X}{\sum}\textrm{ }|\mathscr N_y|\leq \kappa
$.
\end{proof}

\begin{Thm}\label{T4.9}
Let $C$ be a right semiperfect coalgebra over a field $K$. Let $\mathscr R:\mathscr X\longrightarrow\mathscr Ent_C$ be an entwined $C$-representation of a small category $\mathscr X$. Then, the category $Mod^C-\mathscr R$
of entwined modules over $\mathscr R$ is a Grothendieck category.
\end{Thm}

\begin{proof}
Since filtered  colimits and finite limits in $Mod^C-\mathscr R$ are computed pointwise, it is clear that they commute with each other. 

\smallskip
We now consider an object $\mathscr M$ in $Mod^C-\mathscr R$ and an element $m\in el_{\mathscr X}(\mathscr M)$. Then, $m\in \mathscr M_x(r)$ for
some $x\in \mathscr X$ and $r\in \mathscr R_x$. By \cite[Lemma 2.8]{BBR}, we can find a finite dimensionsal $C$-subcomodule $V'\subseteq \mathscr M_x(r)$
containing $m$ and a morphism $\eta': V'\otimes H_r\longrightarrow \mathscr M_x$ in $\mathbf M_{\mathscr R_x}^C(\psi_x)$ such that
$\eta'(r)(m\otimes id_r)=m$. Since $C$ is semiperfect, we can choose a finite dimensional projective $V$ in $Comod-C$ along with an epimorphism
$V\longrightarrow V'$. This induces a morphism $\eta:V\otimes H_r\longrightarrow \mathscr M_x$  in $\mathbf M_{\mathscr R_x}^C(\psi_x)$. Corresponding to $\eta$, we now define the subobject $\mathscr N\subseteq \mathscr M$ as in \eqref{eq4.6}. It is clear that $m\in el_{\mathscr X}(\mathscr N)$. By Lemma \ref{L4.8}, we know that $|\mathscr N|\leq \kappa$. 

\smallskip
We now consider the set of isomorphism classes of objects  in $Mod^C-\mathscr R$ having cardinality $\leq \kappa$. From the above, it is clear that any object in $Mod^C-\mathscr R$ may be expressed as a sum of such objects. By choosing one object from each such isomorphism class, we obtain a set of
generators for $Mod^C-\mathscr R$. 
\end{proof}

\section{Entwined representations of a poset and projective generators}

In this section, the small category $\mathscr X$ will always be a partially ordered set. If $x\leq y$ in $\mathscr X$, we will say that there is a single morphism
$x\longrightarrow y$ in $\mathscr X$. We continue with $C$ being a right semiperfect coalgebra over the field $K$ and $\mathscr R:\mathscr X\longrightarrow
\mathscr Ent_C$ being an entwined $C$-representation of $\mathscr X$. From Theorem \ref{T4.9}, we know that $Mod^C-\mathscr R$ is a Grothendieck category. 

\smallskip In this section, we will show that $Mod^C-\mathscr R$ has projective generators. For this, we will construct a pair of adjoint functors
\begin{equation}\label{adjexev}
ex_x^C:\mathbf M_{\mathscr R_x}^C(\psi_x)\longrightarrow Mod^C-\mathscr R \qquad ev_x^C:Mod^C-\mathscr R\longrightarrow \mathbf M_{\mathscr R_x}^C(\psi_x)
\end{equation} for each $x\in \mathscr X$. 

\begin{lem}\label{L5.1} Let $\mathscr X$ be a poset. Fix $x\in \mathscr X$. Then, there is a functor $ex_x^C:\mathbf M_{\mathscr R_x}^C(\psi_x)\longrightarrow Mod^C-\mathscr R$ defined by setting
\begin{equation}
ex_x^C(\mathcal M)_y:=\left\{
\begin{array}{ll}
\alpha^\ast\mathcal M & \mbox{if $\alpha\in \mathscr X(x,y)$}\\
0 & \mbox{if $\mathscr X(x,y)=\phi$}\\
\end{array}\right.
\end{equation} for each $y\in \mathscr X$.
\end{lem}

\begin{proof}
It is immediate that each $ex_x^C(\mathcal M)_y\in \mathbf M_{\mathscr R_y}^C(\psi_y)$. We consider $\beta:y\longrightarrow y'$ in $\mathscr X$. If $x\not\leq y$, we have  $0=ex_x^C(\mathcal M)^\beta:0=\beta^\ast ex_x^C(\mathcal M)_y\longrightarrow ex_x^C(\mathcal M)_{y'}$ in $\mathbf M_{\mathscr R_{y'}}^C(\psi_{y'})$. Otherwise, we consider
$\alpha:x\longrightarrow y$ and $\alpha':x\longrightarrow y'$. Then, we have
\begin{equation*}
id=ex_x^C(\mathcal M)^\beta :\beta^\ast ex_x^C(\mathcal M)_y=\beta^\ast\alpha^\ast \mathcal M\longrightarrow \alpha'^\ast\mathcal M=ex_x^C(\mathcal M)_{y'}
\end{equation*} which follows from the fact that $\beta\circ \alpha=\alpha'$. Given composable morphisms $\beta$, $\gamma$ in $\mathscr X$, it is now  clear from the definitions that $ex_x^C(\mathcal M)^{\gamma\beta}=ex_x^C(\mathcal M)^\gamma\circ \gamma^\ast(ex_x^C(\mathcal M)^\beta)$. 
\end{proof}

\begin{lem}\label{L5.2} Let $\mathscr X$ be a poset. Fix $x\in \mathscr X$. Then, there is a functor
\begin{equation} ev_x^C:Mod^C-\mathscr R\longrightarrow \mathbf M_{\mathscr R_x}^C(\psi_x)\qquad \mathscr M\mapsto \mathscr M_x
\end{equation} Additionally, $ev_x^C$ is exact.
\end{lem}

\begin{proof}
It is immediate that $ev_x^C$ is a functor. Since finite limits and finite colimits in $Mod^C-\mathscr R$ are computed pointwise, it follows that $ev_x^C$ is exact. 
\end{proof}

\begin{thm}\label{P5.3}Let $\mathscr X$ be a poset. Fix $x\in \mathscr X$. Then, $(ex_x^C,ev_x^C)$ is a pair of adjoint functors. 
\end{thm}

\begin{proof}
For any $\mathcal M\in \mathbf M_{\mathscr R_x}^C(\psi_x)$ and $\mathscr N\in Mod^C-\mathscr R$, we will show that
\begin{equation}
Mod^C-\mathscr R(ex_x^C(\mathcal M),\mathscr N)\cong \mathbf M_{\mathscr R_x}^C(\psi_x)(\mathcal M,ev_x^C(\mathscr N))
\end{equation} We begin with a morphism $f:\mathcal M\longrightarrow \mathscr N_x$ in $\mathbf M_{\mathscr R_x}^C(\psi_x)$. Corresponding to $f$, we define $\eta^f:ex_x^C(\mathcal M)\longrightarrow\mathscr N$ in $Mod^C-\mathscr R$ by setting
\begin{equation}
\eta^f_y:ex^C_x(\mathcal M)_y=\alpha^\ast\mathcal M\xrightarrow{\alpha^\ast f}\alpha^\ast\mathscr N_x\xrightarrow{\mathscr N^\alpha}\mathscr N_y
\end{equation} whenever $x\leq y$ and $\alpha\in \mathscr X(x,y)$. Otherwise, we set $0=\eta^f_y:0=ex^C_x(\mathcal M)_y\longrightarrow \mathscr N_y$. For $\beta:y\longrightarrow y'$ in $\mathscr X$, we have to show that the following diagram is commutative. 
\begin{equation}\label{5.6cd}
\begin{CD}
\beta^\ast ex_x^C(\mathcal M)_y @>\beta^\ast\eta^f_y>> \beta^\ast\mathscr N_y \\
@Vex_x^C(\mathcal M)^\beta VV @VV\mathscr N^\beta V \\
ex_x^C(\mathcal M)_{y'} @>\eta^f_{y'}>> \mathscr N_{y'} \\
\end{CD}
\end{equation} If $x\not\leq y$, then $ex_x^C(\mathcal M)_y=0$ and the diagram commutes. Otherwise, we consider $\alpha:x\longrightarrow y$ and
$\alpha'=\beta\circ \alpha:x\longrightarrow y'$. Then, \eqref{5.6cd} reduces to the commutative diagram
\begin{equation}
\begin{CD}
\beta^\ast \alpha^\ast\mathcal M @>\beta^\ast(\mathscr N^\alpha\circ \alpha^\ast f)>> \beta^\ast\mathscr N_y \\
@VidVV @VV\mathscr N^\beta V \\
\beta^\ast\alpha^\ast\mathcal M=\alpha'^\ast\mathcal M @>\mathscr N^{\alpha'}\circ \alpha'^\ast(f)=\mathscr N^\beta\circ\beta^\ast(\mathscr N^\alpha)\circ \beta^\ast\alpha^\ast f>>\mathscr N_{y'} \\
\end{CD}
\end{equation} Conversely, we take $\eta:ex_x^C(\mathcal M)\longrightarrow \mathscr N$ in $Mod^C-\mathscr R$. In particular, this determines
$f^\eta=\eta_x:\mathcal M\longrightarrow \mathscr N_x$ in $ \mathbf M_{\mathscr R_x}^C(\psi_x)$. It may be easily verified that these two associations are inverse to each other. This proves the result. 
\end{proof}

\begin{cor}\label{C5.4}
The functor $ex_x^C:\mathbf M_{\mathscr R_x}^C(\psi_x)\longrightarrow Mod^C-\mathscr R $ preserves projectives.
\end{cor}

\begin{proof}
From Proposition \ref{P5.3}, we know that  $(ex_x^C,ev_x^C)$ is a pair of adjoint functors. From Lemma \ref{L5.2}, we know that the right adjoint $ev_x^C$
is exact. It follows therefore that its left adjoint $ex_x^C$ preserves projective objects. 
\end{proof}

\begin{Thm}\label{T5.5}  Let $C$ be a right semiperfect coalgebra over a field $K$. Let $\mathscr X$ be a poset and let 
$\mathscr R:\mathscr X\longrightarrow \mathscr Ent_C$ be an entwined $C$-representation of $\mathscr X$. Then, $Mod^C-\mathscr R$ has projective generators.
\end{Thm}

\begin{proof} We denote by $Proj^f(C)$ the set of isomorphism classes of finite dimensional projective $C$-comodules. We will show that  the family 
\begin{equation}\mathcal G=\{\mbox{$ex_x^C(V\otimes H_r)$ $\vert$ $x\in \mathscr X$,  $r\in \mathscr R_x$,  $V\in Proj^f(C)$}\}
\end{equation} is a set of projective generators for $Mod^C-\mathscr R$. From Proposition \ref{P3.4}, we know that $V\otimes H_r$ is projective in
$\mathbf M_{\mathscr R_x}^C(\psi_x)$, where $r\in \mathscr R_x$ and $V\in Proj^f(C)$.  It now follows from Corollary \ref{C5.4}  that each $ex_x^C(V\otimes H_r)$
is projective in $Mod^C-\mathscr R$.

\smallskip
It remains to show that $\mathcal G$ is a set of generators for $Mod^C-\mathscr R$. For this, we consider a monomorphism $\iota:\mathscr N\hookrightarrow \mathscr M$ in $Mod^C-\mathscr R$ such that $\mathscr N\subsetneq\mathscr M$. Since kernels and cokernels in $Mod^C-\mathscr R$ are taken pointwise, it follows that there is some 
$x\in \mathscr X$ such that $\iota_x:\mathscr N_x\hookrightarrow\mathscr M_x$ is a monomorphism with $\mathscr N_x\subsetneq\mathscr M_x$.  

\smallskip From the proof of Theorem \ref{T3.5}, we know that $\{V\otimes H_r\}_{r\in \mathscr R_x,V\in Proj^f(C)}$ is a set of generators for $\mathbf M^C_{\mathscr R_x}(\psi_x)$. Accordingly, we can choose a morphism $f:V\otimes H_r\longrightarrow \mathscr M_x$ with $r\in \mathscr R_x$ and $V\in Proj^f(C)$ such that $f$ does not factor through $ev_x^C(\iota)=\iota_x:\mathscr N_x\hookrightarrow\mathscr M_x$. Applying the adjunction $(ex^C_x,ev_x^C)$,
we now obtain a morphism $\eta:ex_x^C(V\otimes H_r)\longrightarrow \mathscr M$ corresponding to $f$, which does not factor through $\iota:
\mathscr N\longrightarrow \mathscr M$. It now follows (see, for instance, \cite[$\S$ 1.9]{Tohoku}) that the family $\mathcal G$ is a set of generators
for $Mod^C-\mathscr R$. 

\end{proof}

\section{Cartesian modules over entwined representations}

We continue with $\mathscr X$ being a poset, $C$ being a right semiperfect $K$-coalgebra and $\mathscr R:\mathscr X\longrightarrow \mathscr Ent_C$ being an entwined $C$-representation of $\mathscr X$. In this section, we will introduce the category of cartesian modules over $\mathscr R$.

\smallskip
Given a morphism $\alpha:(\mathcal R,C,\psi)\longrightarrow (\mathcal S,C,\psi')$ in $\mathscr Ent_C$, we already know that the left adjoint $\alpha^\ast$  is right exact. We will say that   $\alpha:(\mathcal R,C,\psi)\longrightarrow (\mathcal S,C,\psi')$ is  flat if $\alpha^\ast: \mathbf M^C_{\mathcal R}(\psi)\longrightarrow \mathbf M^C_{\mathcal S}(\psi')$ is exact. Accordingly, 
we will say that an entwined $C$-representation $\mathscr R:\mathscr X\longrightarrow \mathscr Ent_C$  is flat if $\alpha^\ast=\mathscr R_\alpha^\ast:\mathbf M^C_{\mathscr R_x}(\psi_x)
\longrightarrow \mathbf M^C_{\mathscr R_y}(\psi_y)$ is exact for each $\alpha:x\longrightarrow y$ in $\mathscr X$.

\begin{defn}\label{D6.1}  Let  $\mathscr R:\mathscr X\longrightarrow \mathscr Ent_C$ be an entwined $C$-representation of 
$\mathscr X$. Suppose that $\mathscr R$ is flat. Let $\mathscr M$ be an entwined module over $\mathscr R$. We will say that $\mathscr M$ is cartesian if for each $\alpha :x\longrightarrow y$
in $\mathscr X$, the morphism $\mathscr M^\alpha:\alpha^\ast\mathscr M_x\longrightarrow\mathscr M_y$  in $\mathbf M^C_{\mathscr R_y}(\psi_y)$ is an isomorphism.  

\smallskip
We will denote by $Cart^C-\mathscr R$ the full subcategory of $Mod^C-\mathscr R$ consisting of cartesian modules. 
\end{defn} It is clear that $Cart^C-\mathscr R$ is an abelian category, with filtered colimits and finite limits coming from $Mod^C-\mathscr R$.

\smallskip We will now give conditions so that $Cart-\mathscr R$ is a Grothendieck category. For this, we will need some intermediate results. First, we recall (see, for instance, \cite{AR}) that an object $M$ in a Grothendieck category $\mathcal A$ is said to be finitely generated if the functor ${\mathcal A}(M,\_\_)$ satisfies
\begin{equation}
\underset{i\in I}{\varinjlim}\textrm{ } {\mathcal A}(M,M_i)={\mathcal A}(M,\underset{i\in I}{\varinjlim}\textrm{ } M_i)
\end{equation}
where $\{M_i\}_{i\in I}$ is any filtered system of objects in $\mathcal A$ connected by monomorphisms. 

\begin{thm}\label{P6.1}  Let $(\mathcal R,C,\psi)$ be an entwining structure with $C$ a right semiperfect coalgebra. Let $V$ be a finite dimensional projective right $C$-comodule. Then, for any $r\in \mathcal R$, the module
$V\otimes H_r$ is a finitely generated projective object in  $\mathbf M_{\mathcal R}^C(\psi)$.

\end{thm}

\begin{proof}
From Proposition \ref{P3.4}, we already know that $V\otimes H_r$ is a projective object in  $\mathbf M_{\mathcal R}^C(\psi)$. To show that it is finitely generated, we consider a filtered system $\{\mathcal M_i\}_{i\in I}$ of objects in $\mathbf M_{\mathcal R}^C(\psi)$ connected by monomorphisms and set 
$\mathcal M:=\underset{i\in I}{\varinjlim}\textrm{ } \mathcal M_i$. Since $\mathbf M_{\mathcal R}^C(\psi)$ is a Grothendieck category, we note that we have an inclusion $\eta_i:\mathcal M_i\hookrightarrow \mathcal M$ for each $i\in I$. 

\smallskip We now take a morphism $\zeta:V\otimes H_r\longrightarrow \mathcal M$ in $\mathbf M^C_{\mathcal R}(\psi)$. We choose a basis 
$\{v_1,...,v_n\}$ for $V$. For each $1\leq k\leq n$, we now have a morphism in $\mathbf M_{\mathcal R}$ given by 
\begin{equation}
\zeta_k:H_r\longrightarrow V\otimes H_r \qquad H_r(s)=\mathcal R(s,r)\ni f\mapsto v_k\otimes f\in (V\otimes H_r)(s)
\end{equation} Then, each composition $\zeta\circ \zeta_k:H_r\longrightarrow\mathcal M$ is a morphism in $\mathbf M_{\mathcal R}$. Since $H_r$ is a finitely generated object in $\mathbf M_{\mathcal R}$, we can now choose $j\in I$ such that every $\zeta\circ \zeta_k$ factors through $\eta_j:\mathcal M_j\hookrightarrow \mathcal M$. We now construct the following pullback diagram in $\mathbf M^C_{\mathcal R}(\psi)$
\begin{equation}\label{eq6.3}
\begin{CD}
\mathcal N @>>> \mathcal M_j \\
@V\iota VV @VV\eta_jV \\
V\otimes H_r @>\zeta>> \mathcal M\\
\end{CD}
\end{equation} Then, $\iota:\mathcal N\longrightarrow\mathcal M$ is a monomorphism in $\mathbf M^C_{\mathcal R}(\psi)$. From the construction of finite limits in $\mathbf M^C_{\mathcal R}(\psi)$, it follows that for each $s\in \mathcal R$, we have a pullback diagram in $Vect_K$
\begin{equation}\label{eq6.4}
\begin{CD}
\mathcal N(s) @>>> \mathcal M_j(s) \\
@V\iota(s) VV @VV\eta_j(s)V \\
(V\otimes H_r)(s) @>\zeta(s)>> \mathcal M(s)\\
\end{CD}
\end{equation} By assumption, we know that $\zeta(s)(v_k\otimes f)\in Im(\eta_j(s))$ for any basis element $v_k$ and any $f\in H_r(s)$. It follows that $Im(\zeta(s))\subseteq Im(\eta_j(s))$ and hence the pullback $\mathcal N(s)=(V\otimes H_r)(s)$. In other words, $\mathcal N=V\otimes H_r$. The result is now clear.  
\end{proof}

\begin{lem}\label{L6.3} Let $\alpha:(\mathcal R,C,\psi)\longrightarrow (\mathcal S,C,\psi')$ be a flat morphism in $\mathscr Ent_C$. Let $\mathcal M\in 
\mathbf M^C_{\mathcal R}(\psi)$. 

\smallskip (a) There exists a family $\{r_i\}_{i\in I}$ of objects of $\mathcal R$ and a family $\{V_i\}_{i\in I}$ of finite dimensional projective $C$-comodules such that there is an epimorphism in $\mathbf M^C_{\mathcal S}(\psi')$
\begin{equation}
\eta: \underset{i\in I}{\bigoplus}\textrm{ } (V_i\otimes H_{\alpha(r_i)})\longrightarrow \alpha^\ast\mathcal M
\end{equation}

\smallskip
(b) Let $s\in \mathcal S$ and let $W$ be a finite dimensional projective in $Comod-C$. Let $\zeta:W\otimes H_s\longrightarrow \alpha^\ast\mathcal M$ be a morphism in $\mathbf M^C_{\mathcal S}(\psi')$. Then, there exists a finite set $\{r_1,...,r_n\}$ of objects of $\mathcal R$,  a finite family 
$\{V_1,...,V_n\}$ of finite dimensional projective $C$-comodules and a morphism $\eta'':\underset{k=1}{\overset{n}{\bigoplus}}V_k\otimes H_{r_k}\longrightarrow \mathcal M$ in $\mathbf M^C_{\mathcal R}(\psi)$   such that $\zeta$ factors through $\alpha^\ast\eta''$.

\end{lem}

\begin{proof} (a) 
From the proof of Theorem \ref{T3.5}, we know that there exists an epimorphism in $\mathbf M^C_{\mathcal R}(\psi)$  \begin{equation}
\eta' : \underset{i\in I}{\bigoplus}V_i\otimes H_{r_i}\longrightarrow \mathcal M
\end{equation} where each $r_i\in \mathcal R$ and each $V_i$ is a finite dimensional projective $C$-comodule. Since $\alpha^\ast:\mathbf M^C_{\mathcal R}(\psi)\longrightarrow \mathbf M^C_{\mathcal S}(\psi')$ is a left adjoint, it induces an epimorphism $\alpha^\ast(\eta')$ in $\mathbf M^C_{\mathcal S}(\psi')$. From the definition in \eqref{ke2.2} and the construction in Proposition \ref{P2.2}, it is clear that $\alpha^\ast(V_i\otimes H_{r_i})=V_i\otimes \alpha^\ast H_{r_i}=V_i\otimes H_{\alpha(r_i)}$. This proves (a). 

\smallskip
(b) We consider the epimorphism $\alpha^\ast\eta'=\eta: \underset{i\in I}{\bigoplus}\textrm{ } (V_i\otimes H_{\alpha(r_i)})\longrightarrow \alpha^\ast\mathcal M$ constructed in (a). From Proposition \ref{P6.1}, we know that $W\otimes H_s$ is a finitely generated projective object in $\mathbf M^C_{\mathcal S}(\psi')$. 
As such $\zeta:W\otimes H_s\longrightarrow \alpha^\ast\mathcal M$ can be lifted to a morphism $\zeta': W\otimes H_s\longrightarrow \underset{i\in I}{\bigoplus}\textrm{ } (V_i\otimes H_{\alpha(r_i)})$ and $\zeta'$ factors through a finite direct sum of objects from the family $\{V_i\otimes H_{\alpha(r_i)}\}_{i\in I}$.  The result is now clear.
\end{proof}

\begin{lem}\label{L6.4}  Let $\alpha:(\mathcal R,C,\psi)\longrightarrow (\mathcal S,C,\psi')$ be a flat morphism in $\mathscr Ent_C$. Let $\kappa_1$ be any cardinal such that
\begin{equation}\kappa_1\geq max\{
\mbox{$\mathbb N$, $|Mor(\mathcal R)|$,  $|C|$, $|K|$}\}
\end{equation}
 Let $\mathcal M\in 
\mathbf M^C_{\mathcal R}(\psi)$ and let $A\subseteq el(\alpha^\ast\mathcal M)$ be a set of elements such that $|A|\leq \kappa_1$. Then, there is a submodule $\mathcal N\hookrightarrow \mathcal M$ in $\mathbf M^C_{\mathcal R}(\psi)$ with $|\mathcal N|\leq \kappa_1$ such that $A\subseteq 
el(\alpha^\ast\mathcal N)$. 

\end{lem}

\begin{proof} We consider some element $a\in A\subseteq el(\alpha^\ast\mathcal M)$. Then, we can choose a morphism $\zeta^a:W^a\otimes H_{s^a}
\longrightarrow \alpha^\ast\mathcal M$ in $\mathbf M_{\mathcal S}^C(\psi')$ such that $a\in el(Im(\zeta^a))$, where $s^a\in \mathcal S$ and $W^a$ is a finite dimensional projective in
$Comod-C$. Using Lemma \ref{L6.3}(b), we can now choose a finite set $\{r^a_1,...,r^a_{n^a}\}$ of objects of $\mathcal R$,  a finite family 
$\{V^a_1,...,V^a_{n^a}\}$ of finite dimensional projective $C$-comodules and a morphism $\eta^{a''}:\underset{k=1}{\overset{n^a}{\bigoplus}}V_k^a\otimes H_{r_k^a}\longrightarrow \mathcal M$ in $\mathbf M^C_{\mathcal R}(\psi)$ such that $\zeta^a$ factors through $\alpha^\ast\eta^{a''}$. We now set
\begin{equation}
\mathcal N:=Im\left(\eta'':=\underset{a\in A}{\bigoplus}\eta^{a''}:\underset{a\in A}{\bigoplus}\textrm{ }\underset{k=1}{\overset{n^a}{\bigoplus}}V_k^a\otimes H_{r_k^a}\longrightarrow \mathcal M\right)
\end{equation} Since $\alpha$ is flat and $\alpha^\ast$ is a left adjoint, we obtain 
\begin{equation}
\alpha^\ast\mathcal N=Im\left(\alpha^\ast\eta''=\underset{a\in A}{\bigoplus}\alpha^\ast\eta^{a''}:\underset{a\in A}{\bigoplus}\textrm{ }\underset{k=1}{\overset{n^a}{\bigoplus}}V_k^a\otimes H_{\alpha(r_k^a)}\longrightarrow\alpha^\ast \mathcal M\right)
\end{equation} Since each $a\in el(Im(\zeta^a))$ and $\zeta^a$ factors through $\alpha^\ast\eta^{a''}$, we get $A\subseteq 
el(\alpha^\ast\mathcal N)$. 

\smallskip
It remains to show that $|\mathcal N|\leq \kappa_1$. Since $\mathcal N$ is a quotient of $\underset{a\in A}{\bigoplus}\textrm{ }\underset{k=1}{\overset{n^a}{\bigoplus}}V_k^a\otimes H_{r_k^a}$ and $|A|\leq \kappa_1$, it suffices to show that each $|V_k^a\otimes H_{r_k^a}|\leq \kappa_1$. This is clear from the definition of $\kappa_1$, using the fact that each $V_k^a$ is finite dimensional. 

\end{proof}

\begin{rem}\label{Rem6.5} \emph{By considering $\alpha=id$ in Lemma \ref{L6.4}, we obtain the following simple consequence: if $A\subseteq el(\mathcal M)$ is any subset with $|A|\leq \kappa_1$, there is a submodule $\mathcal N\hookrightarrow \mathcal M$ in $\mathbf M^C_{\mathcal R}(\psi)$ with $|\mathcal N|\leq \kappa_1$ such that $A\subseteq 
el(\mathcal N)$.  }

\end{rem}

\begin{lem}\label{L6.6} Let $\alpha:(\mathcal R,C,\psi)\longrightarrow (\mathcal S,C,\psi')$ be a flat morphism in $\mathscr Ent_C$ and let $\mathcal M\in 
\mathbf M^C_{\mathcal R}(\psi)$. Let $\kappa_2$ be any cardinal such that $\kappa_2 \geq max\{
\mbox{$\mathbb N$, $|Mor(\mathcal R)|$, $|Mor(\mathcal S)|$, $|C|$, $|K|$}\}$ and let $A\subseteq el(\mathcal M)$ and $B\subseteq el(\alpha^\ast\mathcal M)$ be subsets with 
$|A|$, $|B|\leq \kappa_2$. Then, there exists a submodule $\mathcal N\subseteq \mathcal M$ in $\mathbf M^C_{\mathcal R}(\psi)$ such that 

\smallskip
(1) $|\mathcal N|\leq \kappa_2$, $|\alpha^\ast\mathcal N|\leq \kappa_2$

\smallskip
(2) $A\subseteq el(\mathcal N)$ and $B\subseteq el(\alpha^\ast\mathcal N)$.

\end{lem}

\begin{proof}
Applying Lemma \ref{L6.4} (and Remark \ref{Rem6.5}), we obtain submodules $\mathcal N_1$, $\mathcal N_2\subseteq \mathcal M$ such that 

\smallskip
(1) $|\mathcal N_1|, |\mathcal N_2|\leq \kappa_2$

\smallskip
(2) $A\subseteq el(\mathcal N_1)$, $B\subseteq el(\alpha^\ast\mathcal N_2)$. 

\smallskip
We set $\mathcal N:=(\mathcal N_1+\mathcal N_2)\subseteq \mathcal M$. Then, $(\mathcal N_1+\mathcal N_2)$ is a quotient of
$\mathcal N_1\oplus\mathcal N_2$ and hence $|\mathcal N|\leq \kappa_2$. Also, it is clear that $A\subseteq  el(\mathcal N_1) \subseteq el(\mathcal N)$. Since $\alpha$ is flat, we get $B\subseteq el(\alpha^\ast\mathcal N_2)\subseteq el(\alpha^\ast\mathcal N)$.

\smallskip
It remains to show that $|\alpha^\ast\mathcal N|\leq \kappa_2$. By the definition in \eqref{ke2.2},  we know that $\alpha^*(\mathcal N)(s)$ is a quotient of \begin{equation}\label{ke6.9}
\left(\underset{r\in \mathcal R}{\bigoplus}\mathcal N(r)\otimes \mathcal S(s,\alpha(r))\right)
\end{equation} for each $s\in \mathcal S$. Since $\kappa_2 \geq  |Mor(\mathcal R)|, |Mor(\mathcal S)|$, it follows from \eqref{ke6.9} that $|\alpha^*(\mathcal N)(s)|\leq\kappa_2$. Again since $\kappa_2 \geq  |Mor(\mathcal S)|$, we get $|\alpha^\ast\mathcal N|\leq \kappa_2$.

\end{proof}

We will now show that $Cart^C-\mathscr R$ has a generator when $\mathscr R:\mathscr X\longrightarrow \mathscr Ent_C$ is a flat representation of the poset $\mathscr X$.  This will be done using   induction on $\mathbb N\times Mor(\mathscr X)$  in a manner similar to the proof of  
\cite[Proposition 3.25]{EV}. As in Section 4, we set
\begin{equation}
\kappa =sup\{
\mbox{$|\mathbb N|$, $|C|$, $|K|$, $|Mor(\mathscr X)|$,  $|Mor(\mathscr R_x)|$, $x\in \mathscr X$}\}
\end{equation} Let $\mathscr M$ be a cartesian module over $\mathscr R:\mathscr X\longrightarrow \mathscr Ent_C$. We now consider an element $m\in el_{\mathscr X}(\mathscr M)$. Suppose that $m\in \mathscr M_x(r)$ for some $x\in \mathscr X$ and
$r\in \mathscr R_x$. As in the proof of Theorem \ref{T4.9}, we fix a finite dimensional projective $C$-comodule $V$ and a morphism $\eta: V\otimes H_r
\longrightarrow \mathscr M_x$ in $\mathbf M^C_{\mathscr R_x}(\psi_x)$ such that $m$ is an element of the image of $\eta$. Corresponding to $\eta$, we define $\mathscr N\subseteq \mathscr M$ as in  \eqref{eq4.6}. It is clear that $m\in el_{\mathscr X}(\mathscr N)$. By Lemma \ref{L4.8}, we know that $|\mathscr N|\leq \kappa$. 

\smallskip
Next, we choose a well ordering of the set $Mor(\mathscr X)$ and consider the induced lexicographic ordering of $\mathbb N\times Mor(\mathscr X)$. Corresponding to each pair $(n,\alpha:y\longrightarrow z)\in \mathbb N\times Mor(\mathscr X)$, we will now define a subobject $\mathscr P(n,\alpha)
\hookrightarrow \mathscr M$ in $Mod^C-\mathscr R$ satisfying the following conditions.

\smallskip
(1) $m\in el_{\mathscr X}(\mathscr P(1,\alpha_0))$, where $\alpha_0$ is the least element of $Mor(\mathscr X)$.

\smallskip
(2) $\mathscr P(n,\alpha)\subseteq \mathscr P(m,\beta)$, whenever $(n,\alpha)\leq (m,\beta)$ in  $\mathbb N\times Mor(\mathscr X)$

\smallskip
(3) For each $(n,\alpha:y\longrightarrow z)\in \mathbb N\times Mor(\mathscr X)$, the morphism $\mathscr P(n,\alpha)^\alpha:\alpha^\ast \mathscr P(n,\alpha)_y
\longrightarrow \mathscr P(n,\alpha)_z$ is an isomorphism in $\mathbf M^C_{\mathscr R_z}(\psi_z)$. 

\smallskip
(4) $|\mathscr P(n,\alpha)|\leq \kappa$. 

\smallskip
For $(n,\alpha:y\longrightarrow z)\in \mathbb N\times Mor(\mathscr X)$, we start the process of constructing the module $\mathscr P(n,\alpha)$ as follows: we set
\begin{equation}\label{6.11dp}
A^0_0(w):=\left\{
\begin{array}{ll}
\mathscr N_w & \mbox{if $n=1$ and $\alpha=\alpha_0$}\\
\underset{(m,\beta)<(n,\alpha)}{\bigcup}\textrm{ }\mathscr P(m,\beta)_w& \mbox{otherwise} \\
\end{array}\right.
\end{equation} for each $w\in \mathscr X$. It is clear that  each $A^0_0(w)\subseteq el(\mathscr M_w)$ and $|A^0_0(w)|\leq \kappa$.

\smallskip
 Since $\mathscr M$ is cartesian, we know that $\alpha^\ast\mathscr M_y=\mathscr M_z$. Since $\alpha:(\mathscr R_y,C,\psi_y)\longrightarrow (\mathscr R_z,C,\psi_z)$ is flat in $\mathscr Ent_C$, we use Lemma \ref{L6.6} with $A^0_0(y)\subseteq el(\mathscr M_y)$ and $A^0_0(z)\subseteq el(\alpha^\ast\mathscr M_y)=
el( \mathscr M_z)$ to obtain $A^0_1(y)\hookrightarrow \mathscr M_y$ in $\mathbf M^C_{\mathscr R_y}(\psi_y)$ such that
\begin{equation}\label{card6.12}
|A^0_1(y)|\leq \kappa \qquad |\alpha^\ast A^0_1(y)|\leq \kappa \qquad A^0_0(y) \subseteq el(A^0_1(y))\qquad A^0_0(z)\subseteq el(\alpha^\ast A^0_1(y))
\end{equation} We now set $A^0_1(z):=\alpha^\ast A^0_1(y)$. Then, \eqref{card6.12} can be rewritten as 
\begin{equation}\label{card6.13}
|A^0_1(y)|\leq \kappa \qquad |A^0_1(z)|\leq \kappa \qquad A^0_0(y) \subseteq el(A^0_1(y))\qquad A^0_0(z)\subseteq el(A^0_1(z))
\end{equation} We observe here that since $\mathscr X$ is a poset, then $y=z$ implies $\alpha:y\longrightarrow z$ is the identity and hence $A^0_1(y)=
A^0_1(z)$.  For any $w\ne y,z$ in $\mathscr X$, we set $A^0_1(w)=A^0_0(w)$. Combining with \eqref{card6.13}, we have  $A^0_0(w)\subseteq A^0_1(w)$ for every $w\in \mathscr X$ and each $|A^0_1(w)|\leq \kappa$.

\begin{lem}\label{L6.61} Let $B\subseteq el_{\mathscr X}(\mathscr M)$ with $|B|\leq \kappa$. Then, there is a submodule $\mathscr Q\hookrightarrow 
\mathscr M$ in $Mod^C-\mathscr R$ such that $B\subseteq el_{\mathscr X}(\mathscr Q)$  and $|\mathscr Q|\leq \kappa$.
\end{lem}
\begin{proof}
For any $m\in B\subseteq el_{\mathscr X}(\mathscr M)$ we can choose, as in the proof of Theorem \ref{T4.9}, a subobject $\mathscr Q_m\subseteq \mathscr M$ such that $m\in el_{\mathscr X}(\mathscr Q_m)$ and $|\mathscr Q_m|\leq \kappa$. Then, we set $\mathscr Q:=\underset{m\in B}{\sum}
\mathscr Q_m$. In particular, $\mathscr Q$ is a quotient of $ \underset{m\in B}{\bigoplus}
\mathscr Q_m$. Since $|B|\leq \kappa$, the result follows. 
\end{proof}

Using Lemma \ref{L6.61}, we now choose a submodule $\mathscr Q^0(n,\alpha)\hookrightarrow \mathscr M$ in $Mod^C-\mathscr R$ such that $\underset{w\in 
\mathscr X}{\bigcup}\textrm{ }A^0_1(w)\subseteq el_{\mathscr X}(\mathscr Q^0(n,\alpha))$ and $|\mathscr Q^0(n,\alpha)|\leq \kappa$. In particular,
$A^0_1(w)\subseteq \mathscr Q^0(n,\alpha)_w$ for each $w\in \mathscr X$. 

\smallskip
We now iterate this construction. Suppose we have constructed a submodule $\mathscr Q^l(n,\alpha)\hookrightarrow \mathscr M$ for every
$l\leq m$ such that $\underset{w\in 
\mathscr X}{\bigcup}\textrm{ }A^l_1(w)\subseteq el_{\mathscr X}(\mathscr Q^l(n,\alpha))$ and $|\mathscr Q^l(n,\alpha)|\leq \kappa$. Then, we set $A^{m+1}_0(w):=\mathscr Q^m(n,\alpha)_w$ for each $w\in \mathscr X$. We then use Lemma \ref{L6.6} with $A^{m+1}_0(y)\subseteq el(\mathscr M_y)$ and $A^{m+1}_0(z)\subseteq el(\alpha^\ast\mathscr M_y)=
el( \mathscr M_z)$ to obtain $A^{m+1}_1(y)\hookrightarrow \mathscr M_y$ in $\mathbf M^C_{\mathscr R_y}(\psi_y)$ such that
\begin{equation}\label{card6.14}
|A^{m+1}_1(y)|\leq \kappa \qquad |\alpha^\ast A^{m+1}_1(y)|\leq \kappa \qquad A^{m+1}_0(y) \subseteq el(A^{m+1}_1(y))\qquad A^{m+1}_0(z)\subseteq el(\alpha^\ast A^{m+1}_1(y))
\end{equation} We now set $A^{m+1}_1(z):=\alpha^\ast A^{m+1}_1(y)$. Then, \eqref{card6.14} can be rewritten as 
\begin{equation}\label{card6.15}
|A^{m+1}_1(y)|\leq \kappa \qquad |A^{m+1}_1(z)|\leq \kappa \qquad A^{m+1}_0(y) \subseteq el(A^{m+1}_1(y))\qquad A^{m+1}_0(z)\subseteq el(A^{m+1}_1(z))
\end{equation} For any $w\ne y,z$ in $\mathscr X$, we set $A^{m+1}_1(w)=A^{m+1}_0(w)$. Combining with \eqref{card6.15}, we have  $A^{m+1}_0(w)\subseteq A^{m+1}_1(w)$ for every $w\in \mathscr X$ and each $|A^{m+1}_1(w)|\leq \kappa$.

\smallskip
Using Lemma \ref{L6.61}, we now choose a submodule $\mathscr Q^{m+1}(n,\alpha)\hookrightarrow \mathscr M$ in $Mod^C-\mathscr R$ such that $\underset{w\in 
\mathscr X}{\bigcup}\textrm{ }A^{m+1}_1(w)\subseteq el_{\mathscr X}(\mathscr Q^{m+1}(n,\alpha))$ and $|\mathscr Q^{m+1}(n,\alpha)|\leq \kappa$. In particular,
$A^{m+1}_1(w)\subseteq \mathscr Q^{m+1}(n,\alpha)_w$ for each $w\in \mathscr X$. 

\smallskip
Finally, we set
\begin{equation}\label{6.16ep}
\mathscr P(n,\alpha):=\underset{m\geq 0}{\varinjlim}\textrm{ }\mathscr Q^m(n,\alpha)
\end{equation}  in $Mod^C-\mathscr R$.

\begin{lem}\label{L6.62}
The family $\{\mbox{$\mathscr P(n,\alpha)$ $\vert$ $(n,\alpha)\in \mathbb N\times Mor(\mathscr X)$}\}$ satisfies the following conditions.

\smallskip
(1) $m\in el_{\mathscr X}(\mathscr P(1,\alpha_0))$, where $\alpha_0$ is the least element of $Mor(\mathscr X)$.

\smallskip
(2) $\mathscr P(n,\alpha)\subseteq \mathscr P(m,\beta)$, whenever $(n,\alpha)\leq (m,\beta)$ in  $\mathbb N\times Mor(\mathscr X)$

\smallskip
(3) For each $(n,\alpha:y\longrightarrow z)\in \mathbb N\times Mor(\mathscr X)$, the morphism $\mathscr P(n,\alpha)^\alpha:\alpha^\ast \mathscr P(n,\alpha)_y
\longrightarrow \mathscr P(n,\alpha)_z$ is an isomorphism in $\mathbf M^C_{\mathscr R_z}(\psi_z)$. 

\smallskip
(4) $|\mathscr P(n,\alpha)|\leq \kappa$. 
\end{lem}

\begin{proof}
The conditions (1) and (2) are immediate from the definition in \eqref{6.11dp}. The condition (4) follows from \eqref{6.16ep} and the fact that each $|\mathscr Q^{m+1}(n,\alpha)|\leq \kappa$.

\smallskip
To prove (3), we notice that $\mathscr P(n,\alpha)_y$  may be expressed  as the  filtered union
\begin{equation} 
A^0_1(y)\hookrightarrow \mathscr Q^0(n,\alpha)_y\hookrightarrow A^1_1(y)\hookrightarrow \mathscr Q^1(n,\alpha)_y\hookrightarrow \dots 
\hookrightarrow A^{m+1}_1(y)\hookrightarrow\mathscr Q^{m+1}(n,\alpha)_y\hookrightarrow ... 
\end{equation}
of objects in $\mathbf M^C_{\mathscr R_y}(\psi_y)$.  Since $\alpha^\ast$ is exact and a left adjoint, we can express $\alpha^\ast\mathscr P(n,\alpha)_y$  as the  filtered union
\begin{equation} \label{c6.18}
\alpha^\ast A^0_1(y)\hookrightarrow \alpha^\ast\mathscr Q^0(n,\alpha)_y\hookrightarrow \alpha^\ast A^1_1(y)\hookrightarrow \alpha^\ast\mathscr Q^1(n,\alpha)_y\hookrightarrow \dots 
\hookrightarrow \alpha^\ast A^{m+1}_1(y)\hookrightarrow\alpha^\ast\mathscr Q^{m+1}(n,\alpha)_y\hookrightarrow ... 
\end{equation}
 in $\mathbf M^C_{\mathscr R_z}(\psi_z)$.  Similarly, $\mathscr P(n,\alpha)_z$  may be expressed  as the  filtered union
\begin{equation} \label{c6.19}
A^0_1(z)\hookrightarrow \mathscr Q^0(n,\alpha)_z\hookrightarrow A^1_1(z)\hookrightarrow \mathscr Q^1(n,\alpha)_z\hookrightarrow \dots 
\hookrightarrow A^{m+1}_1(z)\hookrightarrow\mathscr Q^{m+1}(n,\alpha)_z\hookrightarrow ... 
\end{equation} in $\mathbf M^C_{\mathscr R_z}(\psi_z)$. By definition, we know that $A^m_1(z)=\alpha^\ast A^m_1(y)$ for each $m\geq 0$. From \eqref{c6.18} and \eqref{c6.19}, it is clear that the filtered colimit of the isomorphisms $\alpha^\ast A^m_1(y)=A^m_1(z)$ induces an isomorphism $\mathscr P(n,\alpha)^\alpha:\alpha^\ast \mathscr P(n,\alpha)_y
\longrightarrow \mathscr P(n,\alpha)_z$. 

\end{proof}

\begin{lem}\label{L6.7} Let $\mathscr M$ be a cartesian module over a flat representation $\mathscr R:\mathscr X\longrightarrow \mathscr Ent_C$. Choose $m\in el_{\mathscr X}(\mathscr M)$. Let  $\kappa =max\{
\mbox{$|\mathbb N|$, $|C|$, $|K|$, $|Mor(\mathscr X)|$,  $|Mor(\mathscr R_x)|$, $x\in \mathscr X$}\}$. Then, there is a cartesian submodule $\mathscr P\subseteq \mathscr M$ with $m\in el_{\mathscr X}(\mathscr P)$ such that $|\mathscr P|\leq \kappa$. 
\end{lem}

\begin{proof} It is clear that $\mathbb N\times Mor(\mathscr X)$ with the lexicographic ordering  is filtered. We set 
\begin{equation}
\mathscr P:=\underset{(n,\alpha)\in \mathbb N\times Mor(\mathscr X)}{\bigcup}\textrm{ }\mathscr P(n,\alpha)\subseteq \mathscr M
\end{equation} in $Mod^C-\mathscr R$. It is immediate that $m\in el_{\mathscr X}(\mathscr P)$.  Since each $|\mathscr P(n,\alpha)|\leq \kappa$, it is clear that $|\mathscr P|\leq \kappa$. 

\smallskip
We now consider a morphism $\beta:z\longrightarrow w$ in $\mathscr X$. Then, the family $\{(m,\beta)\}_{m\geq 1}$ is cofinal in  $\mathbb N\times Mor(\mathscr X)$ and hence it follows that
\begin{equation}
\mathscr P:=\underset{m\geq 1}{\varinjlim}\textrm{ }\mathscr P(m,\beta)
\end{equation} Since each $\mathscr P(m,\beta)^\beta:\beta^\ast \mathscr P(m,\beta)_z
\longrightarrow \mathscr P(m,\beta)_w$ is an isomorphism, the filtered colimit $\mathscr P^\beta:\beta^\ast\mathscr P_z\longrightarrow
\mathscr P_w$ is an isomorphism. 

\end{proof}

\begin{Thm}\label{T6.10} Let $C$ be a right semiperfect coalgebra over a field $K$. Let $\mathscr X$ be a poset and let 
$\mathscr R:\mathscr X\longrightarrow \mathscr Ent_C$ be an entwined $C$-representation of $\mathscr X$. Suppose that $\mathscr R$
is  flat. Then, $Cart^C-\mathscr R$ is a Grothendieck category.
\end{Thm}

\begin{proof}
It is already clear that $Cart^C-\mathscr R$ satisfies the (AB5) condition. From Lemma \ref{L6.7}, it is clear that any $\mathscr M\in Cart^C-\mathscr R$
can be expressed as a sum of a family $\{\mathscr P_m\}_{m\in el_{\mathscr X}(\mathscr M)}$ of cartesian submodules such that each $|\mathscr P_m|
\leq \kappa$. As such, isomorphism classes of cartesian modules $\mathscr P$ with $|\mathscr P|\leq \kappa$ form a family of generators
for $Cart^C-\mathscr R$. 
\end{proof}

\section{Separability of the forgetful functor}

Let $(\mathcal R,C,\psi)$ be an entwining structure. We consider the forgetful functor $\mathcal F:\mathbf M^C_{\mathcal R}(\psi)\longrightarrow \mathbf M_{\mathcal R}$. By \cite[Lemma 2.4 \& Lemma 3.1]{BBR}, we know that $\mathcal F$  has a right adjoint $\mathcal G:\mathbf M_{\mathcal R}\longrightarrow \mathbf M^C_{\mathcal R}(\psi)$ given by setting $\mathcal G(
\mathcal N):=\mathcal N\otimes C$, i.e. $\mathcal G(\mathcal N)(r):=\mathcal N(r)\otimes C$ for each $r\in \mathcal R$. The right $\mathcal R$-module structure on $\mathcal G(\mathcal N)$ is given by $(n\otimes c)\cdot f:=nf_\psi\otimes c^\psi$ for $f\in \mathcal R(r',r)$, $n\in \mathcal N(r)$ and $c\in C$.

\smallskip
We continue with $\mathscr X$ being a poset, $C$ being a right semiperfect coalgebra and let $\mathscr R:\mathscr X\longrightarrow \mathscr Ent_C$ be an entwined $C$-representation. We denote by $\mathscr Lin$ the category
of small $K$-linear categories. Then, for each $x\in \mathscr X$, we may replace
the entwining structure $(\mathscr R_x,C,\psi_x)$ by the $K$-linear category $\mathscr R_x$ to obtain a functor that we continue to denote by $\mathscr R:\mathscr X\longrightarrow
\mathscr Lin$. We consider modules  over $\mathscr R:\mathscr X\longrightarrow
\mathscr Lin$ in the sense of Estrada and Virili \cite[Definition 3.6]{EV} and denote their category by $Mod-\mathscr R$. Explicitly, an object $\mathscr N$ in $Mod-\mathscr R$ consists of a module
$\mathscr N_x\in \mathbf M_{\mathscr R_x}$ for each $x\in \mathscr X$ as well as compatible morphisms $\mathscr N_\alpha:\mathscr N_x\longrightarrow \alpha_\ast\mathscr N_y$
(equivalently $\mathscr N^\alpha:\alpha^\ast\mathscr N_x\longrightarrow \mathscr N_y$) for each  $\alpha:x\longrightarrow y$ in $\mathscr X$. The module $\mathscr N$ is said to be cartesian if each  $\mathscr N^\alpha:\alpha^\ast\mathscr N_x\longrightarrow \mathscr N_y$  is an isomorphism. We denote by $Cart-\mathscr R$ the full subcategory of cartesian
modules on $\mathscr R$. 

\smallskip For each $x\in \mathscr X$, we have a forgetful functor $\mathscr F_x:\mathbf M^C_{\mathscr R_x}(\psi_x)\longrightarrow \mathbf M_{\mathscr R_x}$
having right adjoint $\mathscr G_x:  \mathbf M_{\mathscr R_x}\longrightarrow \mathbf M^C_{\mathscr R_x}(\psi_x)$. From the proofs of Propositions \ref{P2.2} and \ref{P2.3}, it is clear that we have  commutative
diagrams
\begin{equation}\label{cd7.1}
\begin{CD}
\mathbf M^C_{\mathscr R_y}(\psi_y) @>\alpha_\ast >> \mathbf M^C_{\mathscr R_x}(\psi_x)\\
@V\mathscr F_yVV @VV\mathscr F_xV \\
\mathbf M_{\mathscr R_y}@>\alpha_\ast >> \mathbf M_{\mathscr R_x}\\
\end{CD} \qquad  \begin{CD}
\mathbf M^C_{\mathscr R_x}(\psi_x) @>\alpha^\ast >> \mathbf M^C_{\mathscr R_y}(\psi_y)\\
@V\mathscr F_xVV @VV\mathscr F_yV \\
\mathbf M_{\mathscr R_x}@>\alpha^\ast >> \mathbf M_{\mathscr R_y}\\
\end{CD}\qquad 
\begin{CD}
\mathbf M_{\mathscr R_y} @>\alpha_\ast >> \mathbf M_{\mathscr R_x}\\
@V\mathscr G_yVV @VV\mathscr G_xV \\
\mathbf M_{\mathscr R_y}^C(\psi_y)@>\alpha_\ast >> \mathbf M^C_{\mathscr R_x}(\psi_x)\\
\end{CD}
\end{equation} for each $\alpha:x\longrightarrow y$ in $\mathscr X$. 

\begin{thm} Let $\mathscr R:\mathscr X\longrightarrow \mathscr Ent_C$ be an entwined $C$-representation. Then, the collection
$\{\mathscr F_x:\mathbf M^C_{\mathscr R_x}(\psi_x)\longrightarrow \mathbf M_{\mathscr R_x}\}_{x\in \mathscr X}$ (resp. the collection $\{\mathscr G_x:  \mathbf M_{\mathscr R_x}\longrightarrow \mathbf M^C_{\mathscr R_x}(\psi_x)\}_{x\in \mathscr X}$ ) together defines a functor $\mathscr F:Mod^C-\mathscr R\longrightarrow Mod-\mathscr R$ (resp. a functor 
$\mathscr G:Mod-\mathscr R\longrightarrow Mod^C-\mathscr R$).
\end{thm}

\begin{proof} We consider $\mathscr M\in Mod^C-\mathscr R$ and set $\mathscr F(\mathscr M)_x:=\mathscr F_x(\mathscr M_x)\in \mathbf M_{\mathscr R_x}$. For a morphism
$\alpha:x\longrightarrow y$, we obtain from \eqref{cd7.1} a morphism $\mathscr F(\mathscr M)_\alpha:=\mathscr F_x(\mathscr M_\alpha):\mathscr F_x(\mathscr M_x)
\longrightarrow \mathscr F_x(\alpha_\ast\mathscr M_y)=\alpha_\ast\mathscr F_y(\mathscr M_y)$. This shows that $\mathscr F(\mathscr M)$ is an object of $Mod-\mathscr R$. Similarly, it follows from
\eqref{cd7.1} that for any $\mathscr N\in Mod-\mathscr R$, we have $\mathscr G(\mathscr N)\in Mod^C-\mathscr R$ obtained by setting $\mathscr G(\mathscr N)_x:=\mathscr G_x(\mathscr N_x)
=\mathscr N_x\otimes C$. 
\end{proof}

\begin{thm}\label{P7.2}  Let   $\mathscr R:\mathscr X\longrightarrow \mathscr Ent_C$ be an entwined $C$-representation. Then, the functor $\mathscr F:Mod^C-\mathscr R\longrightarrow Mod-\mathscr R$ 
has a right adjoint, given by $\mathscr G:Mod-\mathscr R\longrightarrow Mod^C-\mathscr R$.
\end{thm}
\begin{proof}
We consider $\mathscr M\in Mod^C-\mathscr R$ and $\mathscr N\in Mod-\mathscr R$ along with a morphism $\eta:\mathscr F(\mathscr M)\longrightarrow \mathscr N$ in $Mod-\mathscr R$. We will show how to construct a morphism $\zeta:\mathscr M\longrightarrow \mathscr G(\mathscr N)$ in $Mod^C-\mathscr R$ corresponding to $\eta$. 

\smallskip
For each $x\in \mathscr X$, we consider $\eta_x:\mathscr F(\mathscr M)_x=\mathscr F_x(\mathscr M_x)\longrightarrow \mathscr N_x$ in $\mathbf M_{\mathscr R_x}$. By \cite[Lemma 3.1]{BBR}, we already know that $(\mathscr F_x,\mathscr G_x)$ is a pair of adjoint functors, which gives us $\mathbf M_{\mathscr R_x}(\mathscr F_x(\mathscr M_x),\mathscr N_x)\cong \mathbf M^C_{\mathscr R_x}(
\mathscr M_x,\mathscr G_x(\mathscr N_x))$. Accordingly, we define $\zeta_x:\mathscr M_x\longrightarrow \mathscr G_x(\mathscr N_x)=\mathscr N_x\otimes C$ by setting
$\zeta_x(m'):=\eta_x(r)(m'_0)\otimes m'_1$ for $m'\in \mathscr M_x(r)$, $r\in \mathscr R_x$. We now consider the diagrams
\begin{equation}\label{cd7.2}
\begin{CD}
\mathscr F_x(\mathscr M_x) @>\eta_x>> \mathscr N_x \\
@V\mathscr F_x(\mathscr M_\alpha)VV @VV\mathscr N_\alpha V\\
\alpha_\ast\mathscr F_y(\mathscr M_y) @>\alpha_\ast(\eta_y)>> \alpha_\ast\mathscr N_y\\
\end{CD}
\qquad \Rightarrow \qquad 
\begin{CD}
\mathscr M_x @>\zeta_x>> \mathscr G_x(\mathscr N_x)\\
@V\mathscr M_\alpha VV @VV\mathscr G_x(\mathscr N_\alpha)V\\
\alpha_\ast\mathscr M_y @>\alpha_\ast(\zeta_y)>> \alpha_\ast \mathscr G_y(\mathscr N_y) \\
\end{CD}
\end{equation}
The left hand side diagram in \eqref{cd7.2} is commutative because $\eta:\mathscr F(\mathscr M)\longrightarrow \mathscr N$ is a morphism in $Mod-\mathscr R$. In order to prove that
we have a morphism $\zeta:\mathscr M\longrightarrow \mathscr G(\mathscr N)$ in $Mod^C-\mathscr R$, it suffices to show that this  implies the commutativity of the right hand side diagram in \eqref{cd7.2}.

\smallskip
We consider $m\in el(\mathscr M_x)$. Then, we have $\mathscr G_x(\mathscr N_\alpha)(\zeta_x(m))=\mathscr N_\alpha(\eta_x(m_0))\otimes m_1$. On the other hand, we have
$\alpha_\ast(\zeta_y)(\mathscr M_\alpha(m))=\eta_y((\mathscr M_\alpha(m))_0)\otimes (\mathscr M_\alpha(m))_1$. Since $\mathscr M_\alpha$ is $C$-colinear, we have $(\mathscr M_\alpha(m))_0
\otimes (\mathscr M_\alpha(m))_1=\mathscr M_\alpha(m_0)\otimes m_1$. It follows that $\alpha_\ast(\zeta_y)(\mathscr M_\alpha(m))=\eta_y(\mathscr M_\alpha(m_0))\otimes m_1$. From the left hand side commutative diagram in \eqref{cd7.2}, we get $\eta_y(\mathscr M_\alpha(m_0))=\mathscr N_\alpha(\eta_x(m_0))$, which shows that the right hand diagram in 
\eqref{cd7.2} is commutative.

\smallskip
Similarly, we may show that a morphism $\zeta':\mathscr M\longrightarrow \mathscr G(\mathscr N)$ in $Mod^C-\mathscr R$  induces a morphism $\eta':\mathscr F(\mathscr M)\longrightarrow \mathscr N$ in $Mod-\mathscr R$ and that these two associations are inverse to each other. This proves the result.
\end{proof}

We now recall that a functor $F:\mathcal A\longrightarrow \mathcal B$ is said to be separable if the natural transformation $\mathcal A(\_\_,\_\_)\longrightarrow
\mathcal B(F(\_\_),F(\_\_))$ is a split monomorphism (see \cite{NBO}, \cite{Raf}). If $F$ has a right adjoint $G:\mathcal B\longrightarrow \mathcal A$, then $F$ is separable if and only if there exists a natural transformation 
$\upsilon \in Nat(GF,1_{\mathcal A})$ satisfying $\upsilon\circ \mu=1_{\mathcal A}$, where $\mu$ is the unit of the adjunction (see \cite[Theorem 1.2]{Raf}).

\smallskip
We now consider the forgetful functor $\mathscr F:Mod^C-\mathscr R\longrightarrow Mod-\mathscr R$ as well as its right adjoint $\mathscr G:Mod-\mathscr R\longrightarrow
Mod^C-\mathscr R$ constructed in Proposition \ref{P7.2}. We will need an alternate description for the natural transformations $\mathscr G\mathscr F\longrightarrow 1_{Mod^C-\mathscr R}$.

\begin{thm}\label{P7.25}   A  natural transformation $\upsilon\in Nat(\mathscr G\mathscr F,1_{Mod^C-\mathscr R})$ corresponds to  a collection 
of natural transformations $\{\upsilon_x\in Nat(\mathscr G_x\mathscr F_x,1_{\mathbf M^C_{\mathscr R_x}(\psi_x)})\}_{x\in \mathscr X}$ such that for any $\alpha:x\longrightarrow y$ in
$\mathscr X$ and object $\mathscr M\in Mod^C-\mathscr R$, we have a commutative diagram
\begin{equation}\label{cd7.3}
\begin{CD}
\mathscr G_x\mathscr F_x(\mathscr M_x) @>\upsilon_x(\mathscr M_x)>> \mathscr M_x\\
@V\mathscr G_x\mathscr F_x(\mathscr M_\alpha) VV @VV\mathscr M_\alpha V \\
\alpha_\ast\mathscr G_y\mathscr F_y(\mathscr M_y) @>\alpha_\ast\upsilon_y(\mathscr M_y)>> \alpha_\ast\mathscr M_y\\
\end{CD}
\end{equation} in $\mathbf M^C_{\mathscr R_x}(\psi_x)$. 
\end{thm} 

\begin{proof} We consider $\upsilon \in Nat(\mathscr G\mathscr F,1_{Mod^C-\mathscr R})$. For  $x\in \mathscr X$, we define the natural transformation $\upsilon_x\in Nat(\mathscr G_x\mathscr F_x,1_{\mathbf M^C_{\mathscr R_x}(\psi_x)})$ by setting
\begin{equation}\label{eq7.35b}
\upsilon_x(\mathcal M):=\upsilon(ex^C_x(\mathcal M))_x:\mathscr G_x\mathscr F_x(\mathcal M)=\mathscr G_x\mathscr F_x((ex^C_x(\mathcal M))_x)\longrightarrow (ex_x^C(\mathcal M))_x=\mathcal M\end{equation} for $\mathcal M\in \mathbf M^C_{\mathscr R_x}(\psi_x)$. We now consider  $\mathscr M\in Mod^C-\mathscr R$. For $\alpha:x\longrightarrow y$ in $\mathscr X$, the morphism $\upsilon(\mathscr M):\mathscr G\mathscr F(\mathscr M)\longrightarrow
\mathscr M$ in $Mod^C-\mathscr R$ leads to a commutative diagram
\begin{equation}\label{7.36b}
\begin{CD}
(\mathscr G\mathscr F(\mathscr M))_x = \mathscr G_x\mathscr F_x(\mathscr M_x) @>\upsilon(\mathscr M)_x>> \mathscr M_x \\
@V\mathscr G_x\mathscr F_x(\mathscr M_\alpha) VV @VV\mathscr M_\alpha V \\
\alpha_\ast(\mathscr G\mathscr F(\mathscr M))_y=\alpha_\ast\mathscr G_y\mathscr F_y(\mathscr M_y) @> \alpha_\ast(\upsilon(\mathscr M)_y)>> \alpha_\ast\mathscr M_y\\
\end{CD}
\end{equation} We now claim that $\upsilon(\mathscr M)_x=(\upsilon(ex^C_x(\mathscr M_x)))_x=\upsilon_x(\mathscr M_x)$ for each $x\in \mathscr X$. For this, we consider the canonical morphism $\zeta: ex_x^C(\mathscr M_x)=ex_x^C(ev_x^C(\mathscr M))\longrightarrow \mathscr M$ in $Mod^C-\mathscr R$ corresponding to the adjoint pair $(ex_x^C,ev_x^C)$ in Proposition \ref{P5.3}. It is clear that $ev_x^C(\zeta)=id$. Then, we have commutative diagrams
\begin{equation}\label{7.37b}
\begin{array}{ccc}
\begin{CD}
\mathscr G\mathscr F(ex^C_x(\mathscr M_x)) @>\upsilon(ex^C_x(\mathscr M_x))>> ex_x^C(\mathscr M_x)\\
@V\mathscr G\mathscr F(\zeta)VV @VV\zeta V\\
\mathscr G\mathscr F(\mathscr M) @>\upsilon(\mathscr M)>> \mathscr M \\
\end{CD} & \qquad \Rightarrow \qquad & \begin{CD}
\mathscr G_x\mathscr F_x(\mathscr M_x) @>(\upsilon(ex^C_x(\mathscr M_x)))_x>> \mathscr M_x\\
@Vid VV @VVid V\\
\mathscr G_x\mathscr F_x(\mathscr M_x) @>\upsilon(\mathscr M)_x>> \mathscr M_x \\
\end{CD} \\
\end{array}
\end{equation} This proves that $\upsilon(\mathscr M)_x=(\upsilon(ex^C_x(\mathscr M_x)))_x=\upsilon_x(\mathscr M_x)$ for each $x\in \mathscr X$. The commutativity
of the diagram \eqref{cd7.3} now follows from \eqref{7.36b}.

\smallskip Conversely, given a  collection 
of natural transformations $\{\upsilon_x\in Nat(\mathscr G_x\mathscr F_x,1_{\mathbf M^C_{\mathscr R_x}(\psi_x)})\}_{x\in \mathscr X}$ satisfying \eqref{cd7.3} for each $\mathscr M
\in Mod^C-\mathscr R$, we get $\upsilon(\mathscr M):\mathscr G\mathscr F(\mathscr M)\longrightarrow \mathscr M$ in $Mod^C-\mathscr R$ by setting $\upsilon(\mathscr M)_x=
\upsilon_x(\mathscr M_x)$ for each $x\in \mathscr X$. From \eqref{cd7.3}, it is clear that $\upsilon\in Nat(\mathscr G\mathscr F,1_{Mod^C-\mathscr R})$.

\end{proof}

More explicitly, the diagram in \eqref{cd7.3} shows that for each $\alpha:x\longrightarrow y$ in $\mathscr X$ and $r\in \mathscr R_x$, we have a commutative diagram
\begin{equation}\label{cd7.4}
\begin{CD}
\mathscr M_x(r)\otimes C=(\mathscr G_x\mathscr F_x(\mathscr M_x))(r) @>(\upsilon_x(\mathscr M_x))(r)>> \mathscr M_x(r)\\
@V(\mathscr G_x\mathscr F_x(\mathscr M_\alpha))(r) VV @VV\mathscr M_\alpha(r) V \\
\mathscr M_y(\alpha(r))\otimes C=(\mathscr G_y\mathscr F_y(\mathscr M_y))(\alpha(r))=(\alpha_\ast\mathscr G_y\mathscr F_y(\mathscr M_y))(r) @>(\alpha_\ast\upsilon_y(\mathscr M_y))(r)>=(\upsilon_y(\mathscr M_y))(\alpha(r))> (\alpha_\ast\mathscr M_y)(r)=\mathscr M_y(\alpha(r))\\
\end{CD}
\end{equation} We note that all morphisms in \eqref{cd7.4} are $C$-colinear. 
We now give another interpretation of the space $ Nat(\mathscr G\mathscr F,1_{Mod^C-\mathscr R})$. For this, we consider a collection 
$\theta:=\{\theta_x(r):C\otimes C\longrightarrow \mathscr R_x(r,r)\}_{x\in \mathscr X,r\in \mathscr R_x}$ of $K$-linear maps satisfying the following conditions. 

\smallskip
(1) Fix $x\in \mathscr X$ and $r\in \mathscr R_x$. Then, for $c$, $d\in C$, we have
\begin{equation}\label{theta1}
\theta_x(r)(c\otimes d_1)\otimes d_2=(\theta_x(r)(c_2\otimes d))_{\psi_x}\otimes {c_1}^{\psi_x}
\end{equation}
(2) Fix $x\in \mathscr X$ and $c$, $d\in C$. Then, for $f:s\longrightarrow r$ in $\mathscr R_x$, we have
\begin{equation}\label{theta2}
(\theta_x(r)(c\otimes d))\circ f=f_{{\psi_x}_{\psi_x}}\circ (\theta_x(s)(c^{\psi_x}\otimes d^{\psi_x}))
\end{equation}
(3) Fix $c$, $d\in C$. Then, for any $\alpha:x\longrightarrow y$ in $\mathscr X$ and $r\in \mathscr R_x$, we have
\begin{equation}\label{theta3}
\alpha(\theta_x(r)(c\otimes d))=\theta_y(\alpha(r))(c\otimes d)
\end{equation} The space of all such $\theta$ will be denoted by $V_1$. 

\begin{thm}\label{P7.3} Let $\theta\in V_1$. Then, $\theta$ induces a natural transformation $\upsilon\in Nat(\mathscr G\mathscr F,1_{Mod^C-\mathscr R})$, such that for each $x\in \mathscr X$,  $\upsilon_x\in Nat(\mathscr G_x\mathscr F_x,1_{\mathbf M^C_{\mathscr R_x}(\psi_x)})$ is given by 
\begin{equation}\label{eq7.8}
\upsilon_x(\mathcal M):\mathcal M\otimes C\longrightarrow \mathcal  M\qquad (m\otimes c)\mapsto \mathcal M(\theta_x(r)(m_1\otimes c))(m_0)
\end{equation} for any $\mathcal  M\in \mathbf M^C_{\mathscr R_x}(\psi_x)$, $r\in \mathscr R_x$, $m\in \mathcal M(r)$ and $c\in C$. 
\end{thm}

\begin{proof} From \cite[Proposition 3.6]{BBR}, it follows that each $\upsilon_x$ as defined in \eqref{eq7.8} by the collection $\theta_x:=
\{\theta_x(r):C\otimes C\longrightarrow \mathscr R_x(r,r)\}_{r\in \mathscr R_x}$ gives a natural transformation $\upsilon_x\in Nat(\mathscr G_x\mathscr F_x,1_{\mathbf M^C_{\mathscr R_x}(\psi_x)})$. To prove the result, it therefore suffices to show the commutativity of the diagram \eqref{cd7.4} for any $\mathscr M\in Mod^C-\mathscr R$. Accordingly, for $\alpha:
x\longrightarrow y$ in $\mathscr X$ and $r\in \mathscr R_x$, we have
\begin{equation}\label{eq7.9}
((\mathscr M_\alpha(r))\circ (\upsilon_x(\mathscr M_x))(r))(m\otimes c)=(\mathscr M_\alpha(r))(\mathscr M_x(\theta_x(r)(m_1\otimes c))(m_0))
\end{equation} for $m\otimes c\in \mathscr M_x(r)\otimes C$. On the other hand, we have
\begin{equation}\label{eq7.10}
\begin{array}{ll}
(((\upsilon_y(\mathscr M_y))(\alpha(r)))\circ ((\mathscr G_x\mathscr F_x(\mathscr M_\alpha))(r)))(m\otimes c)&=\mathscr M_y(\theta_y(\alpha(r))(\mathscr M_\alpha(m)_1\otimes c ))(\mathscr M_\alpha(r)(m))_0\\
&=\mathscr M_y(\theta_y(\alpha(r))(m_1\otimes c))(\mathscr M_\alpha(r)(m_0))\\
&=\mathscr M_y(\alpha(\theta_x(r)(m_1\otimes c)))(\mathscr M_\alpha(r)(m_0))\\
\end{array}
\end{equation}
The second equality in \eqref{eq7.10} follows from the $C$-colinearity of $\mathscr M_\alpha(r)$ and the third equality follows by applying condition \eqref{theta3}. We now notice that for any $f\in \mathscr R_x(r,r)$, we have a commutative diagram
\begin{equation}\label{cd7.11}
\begin{CD}
\mathscr M_x(r) @>\mathscr M_\alpha(r)>> \mathscr M_y(\alpha(r))\\
@V\mathscr M_x(f)VV @VV\mathscr M_y(\alpha(f))V \\
\mathscr M_x(r) @>\mathscr M_\alpha(r)>> \mathscr M_y(\alpha(r))\\
\end{CD} 
\end{equation} Applying \eqref{cd7.11}   to $f=\theta_x(r)(m_1\otimes c)\in \mathscr R_x(r,r)$, we obtain from \eqref{eq7.10} that
\begin{equation}\label{eq7.12}
(((\upsilon_y(\mathscr M_y))(\alpha(r)))\circ ((\mathscr G_x\mathscr F_x(\mathscr M_\alpha))(r)))(m\otimes c)=(\mathscr M_\alpha(r))(\mathscr M_x(\theta_x(r)(m_1\otimes c))(m_0))
\end{equation} This proves the result.
\end{proof}

 Fix $x\in \mathscr X$ and $r\in \mathscr R_x$. We now set
\begin{equation}
\mathscr H_y^{(x,r)}:=\left\{\begin{array}{ll}
\mathscr R_y(\_\_,\alpha(r)) \otimes C & \mbox{if $\alpha:x\longrightarrow y$} \\
0  & \mbox{if $x\not\leq y$}\\
\end{array}\right.
\end{equation} for each $y\in \mathscr X$.

\begin{lem}\label{L7.4}
For each $x\in \mathscr X$ and $r\in \mathscr R_x$, the collection $\mathscr H^{(x,r)}:=\{\mathscr H_y^{(x,r)}\}_{y\in \mathscr X}$ determines an object of $Mod^C-\mathscr R$. 
\end{lem}

\begin{proof}
For each $y\in \mathscr X$, it follows by \cite[Lemma 2.4]{BBR} that $\mathscr H_y^{(x,r)}$ is an object of $\mathbf M^C_{\mathscr R_y}(\psi_y)$. We consider $\beta:y\longrightarrow z$
in $\mathscr X$ and suppose we have $\alpha:x\longrightarrow y$, i.e., $x\leq y$. Then, for $r'\in \mathscr R_y$, we have an obvious morphism 
\begin{equation}\beta(\_\_)\otimes C: \mathscr H_y^{(x,r)}(r')=
\mathscr R_y(r',\alpha(r))\otimes C\longrightarrow \beta_\ast(\mathscr R_z(\_\_,\beta\alpha(r))\otimes C)(r')=\mathscr R_z(\beta(r'),\beta\alpha(r))\otimes C
\end{equation} which is $C$-colinear. To prove  that $\mathscr H_y^{(x,r)}\longrightarrow \beta_\ast\mathscr H_z^{(x,r)}$ is a morphism in $\mathbf M^C_{\mathscr R_y}(
\psi_y)$, it remains to show that for any $g:r''\longrightarrow r'$ in $\mathscr R_y$, the following diagram commutes
\begin{equation}\label{cd7.15}
\begin{CD}
\mathscr R_y(r',\alpha(r))\otimes C @>\cdot g>> \mathscr R_y(r'',\alpha(r))\otimes C\\
@V\beta(\_\_)\otimes CVV @VV\beta(\_\_)\otimes CV \\
\mathscr R_z(\beta(r'),\beta\alpha(r))\otimes C @>\cdot \beta(g)>>\mathscr R_z(\beta(r''),\beta\alpha(r))\otimes C  \\
\end{CD}
\end{equation} For $f\otimes c\in \mathscr R_y(r',\alpha(r))\otimes C$, we have
\begin{equation*}
(\beta(\_\_)\otimes C)((f\otimes c)\cdot g)=(\beta(\_\_)\otimes C)(fg_{\psi_y}\otimes c^{\psi_y})=\beta(f)\beta(g_{\psi_y})\otimes c^{\psi_y}=\beta(f)\beta(g)_{\psi_z}
\otimes c^{\psi_z}=(\beta(f)\otimes c)\cdot \beta(g)
\end{equation*} This shows that \eqref{cd7.15} is commutative. Finally, if $x\not\leq y$, then $0=\mathscr H_y^{(x,r)}\longrightarrow \beta_\ast\mathscr H_z^{(x,r)}$ is obviously a morphism in $\mathbf M^C_{\mathscr R_y}(
\psi_y)$. This proves the result.
\end{proof}

\begin{thm}\label{P7.5} Let $\upsilon\in Nat(\mathscr G\mathscr F,1_{Mod^C-\mathscr R})$. For each $x\in \mathscr X$ and $r\in \mathscr R_x$, define $\theta_x(r):
C\otimes C\longrightarrow \mathscr R_x(r,r)$ by setting
\begin{equation}\label{eq7.16}
\theta_x(r)(c\otimes d):=((id\otimes \varepsilon_C)\circ (\upsilon_x(\mathscr H^{(x,r)}_x)(r)))(id_r\otimes c\otimes d)
\end{equation}
for $c$, $d\in C$. Then, the collection $\theta:=\{\theta_x(r):C\otimes C\longrightarrow \mathscr R_x(r,r)\}_{x\in \mathscr X,r\in \mathscr R_x}$ is an element of $V_1$. 
\end{thm}

\begin{proof} From the definition in \eqref{eq7.16}, we have explicitly that
\begin{equation}\label{eq7.17}
\theta_x(r)(c\otimes d)=((id\otimes \varepsilon_C)\circ (\upsilon_x(\mathscr R_x(\_\_,r)\otimes C)(r)))(id_r\otimes c\otimes d)
\end{equation} Then, it follows from \cite[Proposition 3.5]{BBR} that $\theta_x(r)$ satisfies the conditions in \eqref{theta1} and \eqref{theta2}. It remains to verify the condition \eqref{theta3}. For this we take $\alpha:x\longrightarrow y$ in $\mathscr X$ and consider the commutative diagram
\begin{equation}\label{cd7.18}
\begin{CD}
\mathscr R_x(r',r)\otimes C\otimes C @>\upsilon_x(\mathscr H^{(x,r)}_x)(r')>> \mathscr R_x(r',r)\otimes C @>id\otimes \varepsilon_C>> \mathscr R_x(r',r)\\
@V\alpha(\_\_)\otimes C\otimes CVV @VV\alpha(\_\_)\otimes CV @VV\alpha(\_\_)V\\
\mathscr R_y(\alpha(r'),\alpha(r))\otimes C\otimes C @>\upsilon_y(\mathscr H^{(x,r)}_y)(\alpha(r'))>>\mathscr R_y(\alpha(r'),\alpha(r))\otimes C @>id\otimes \varepsilon_C>>\mathscr R_y(\alpha(r'),\alpha(r))\\
\end{CD}
\end{equation} for any $r,r'\in \mathscr R_x$. Since $\upsilon\in Nat(\mathscr G\mathscr F,1_{Mod^C-\mathscr R})$, the commutativity of the left hand side square in \eqref{cd7.18} follows from \eqref{cd7.4}. It is clear that the right hand square in \eqref{cd7.18} is commutative. 

\smallskip
We notice that $\mathscr H_y^{(y,\alpha(r))}=\mathscr H_y^{(x,r)}$ in $\mathbf M_{\mathscr R_y}^C(\psi_y)$. Applying \eqref{cd7.18} with $r'=r\in\mathscr R_x$ and $id_r\otimes c\otimes d\in \mathscr R_x(r,r)\otimes C\otimes C$, it follows from \eqref{eq7.17} that $\alpha(\theta_x(r)(c\otimes d))=\theta_y(\alpha(r))(c\otimes d)$. This proves \eqref{theta3}. 

\end{proof}

\begin{thm}\label{P7.6}   $Nat(\mathscr G\mathscr F,1_{Mod^C-\mathscr R})$ is isomorphic to $V_1$.
\end{thm}

\begin{proof} From Proposition \ref{P7.3} and Proposition \ref{P7.5}, we see that we have maps  $\psi:V_1\longrightarrow
Nat(\mathscr G\mathscr F,1_{Mod^C-\mathscr R})$ and $\phi:Nat(\mathscr G\mathscr F,1_{Mod^C-\mathscr R})\longrightarrow V_1$  in opposite directions.

\smallskip
We consider $\upsilon\in Nat(\mathscr G\mathscr F,1_{Mod^C-\mathscr R})$. By Proposition \ref{P7.5}, $\upsilon$ induces an element $\theta\in V_1$. Applying Proposition \ref{P7.3},
$\theta$ induces an element in $Nat(\mathscr G\mathscr F,1_{Mod^C-\mathscr R})$, which we denote by $\upsilon'$. 
Then, $\upsilon$ and $\upsilon'$ are determined respectively by natural transformations $\{\upsilon_x\in Nat(\mathscr G_x\mathscr F_x,1_{\mathbf M^C_{\mathscr R_x}(\psi_x)})\}_{x\in \mathscr X}$ and $\{\upsilon'_x\in Nat(\mathscr G_x\mathscr F_x,1_{\mathbf M^C_{\mathscr R_x}(\psi_x)})\}_{x\in \mathscr X}$ satisfying compatibility conditions as in \eqref{cd7.3}. From \cite[Proposition 3.7]{BBR}, it follows that $\upsilon'_x=\upsilon_x$ for each $x\in \mathscr X$. Hence, $\upsilon'=\upsilon$ and $\psi\circ \phi=id$. Similarly, we can show that $\phi\circ \psi=id$. 
\end{proof}

\begin{Thm}\label{T7.7}  Let $\mathscr X$ be a partially ordered set.  Let $C$ be a right semiperfect $K$-coalgebra and let $\mathscr R:\mathscr X\longrightarrow 
\mathscr Ent_C$ be an entwined $C$-representation. Then, the functor $\mathscr F:Mod^C-\mathscr R\longrightarrow Mod-\mathscr R$ is separable if and only if there exists
$\theta\in V_1$ such that 
\begin{equation}\label{eq7.19d}
\theta_x(r)(c_1\otimes c_2)=\varepsilon_C(c)\cdot id_r
\end{equation} for every $x\in \mathscr X$, $r\in\mathscr R_x$ and $c\in C$. 

\end{Thm}
\begin{proof} We suppose that $\mathscr F:Mod^C-\mathscr R\longrightarrow Mod-\mathscr R$ is separable. As mentioned before, this implies that there exists $\upsilon\in Nat(\mathscr G\mathscr F,1_{Mod^C-\mathscr R})$ such that
$\upsilon\circ \mu=1_{Mod^C-\mathscr R}$, where $\mu$ is the unit of the adjunction $(\mathscr F,\mathscr G)$. We set $\theta=\phi(\upsilon)$, where $\phi:Nat(\mathscr G\mathscr F,1_{Mod^C-\mathscr R})\longrightarrow V_1$  is the isomorphism described in the proof of Proposition \ref{P7.6}. In particular, for every $x\in \mathscr X$, $r\in\mathscr R_x$, we have
$\upsilon(\mathscr H^{(x,r)})\circ \mu(\mathscr H^{(x,r)})=id$. 
From \eqref{eq7.16}, it now follows that for every  $c\in C$, we have
\begin{equation}
\begin{array}{ll}
\theta_x(r)(c_1\otimes c_2)&=((id\otimes \varepsilon_C)\circ (\upsilon_x(\mathscr H^{(x,r)}_x)(r)))(id_r\otimes c_1\otimes c_2)\\
&=((id\otimes \varepsilon_C)\circ (\upsilon_x(\mathscr H^{(x,r)}_x)(r))\circ \mu(\mathscr H^{(x,r)})_x(r))(id_r\otimes c)\\
&=(id\otimes \varepsilon_C)(id_r\otimes c)=\varepsilon_C(c)\cdot id_r\\
\end{array}
\end{equation} Conversely, suppose that there exists $\theta\in V_1$ satisfying the condition in \eqref{eq7.19d}. We set $\upsilon:=\psi(\theta)$, where $\psi:V_1\longrightarrow
Nat(\mathscr G\mathscr F,1_{Mod^C-\mathscr R})$  is the other isomorphism described in the proof of Proposition \ref{P7.6}. We consider $\mathscr M\in Mod^C-\mathscr R$. By \eqref{eq7.8}, we know that
\begin{equation}\label{eq7.21}
\upsilon_x(\mathscr M_x):\mathscr M_x\otimes C\longrightarrow \mathscr M_x\qquad (m\otimes c)\mapsto \mathscr M_x(\theta_x(r)(m_1\otimes c))(m_0)
\end{equation} for any $x\in \mathscr X$, $r\in \mathscr R_x$, $m\in \mathscr M_x(r)$ and $c\in C$. 
We claim that $\upsilon\circ \mu=1_{Mod^C-\mathscr R}$. For this, we see that
\begin{equation}
\begin{array}{ll}
((\upsilon(\mathscr M)\circ \mu(\mathscr M))_x(r))(m)&=(\upsilon_x(\mathscr M_x)(r))(m_0\otimes m_1) \\
&= \mathscr M_x(\theta_x(r)(m_{01}\otimes m_1))(m_{00})\\
&= \mathscr M_x(\theta_x(r)(m_{11}\otimes m_{12}))(m_{0})\\
&= \varepsilon_C(m_1)m_0=m\\
\end{array}
\end{equation} This proves the result.

\end{proof}

We now turn to cartesian modules over entwined $C$-representations. For this, we assume additionally that 
$\mathscr R:\mathscr X\longrightarrow \mathscr Ent_C$ is flat. Then, it follows from Theorem \ref{T6.10} that $Cart^C-\mathscr R$ is a Grothendieck category. In particular, by taking
$C=K$, we note that $Cart-\mathscr R$ is also a Grothendieck category.

\begin{thm}\label{P7.8} Let $\mathscr X$ be a poset, $C$ be a right semiperfect $K$-coalgebra and $\mathscr R:\mathscr X\longrightarrow \mathscr Ent_C$ be an entwined $C$-representation
that is also flat. Then, the functor $\mathscr F:Mod^C-\mathscr R\longrightarrow Mod-\mathscr R$ restricts to a functor $\mathscr F^c:Cart^C-\mathscr R\longrightarrow Cart-\mathscr R$. Additionally, $\mathscr F^c$ has a right adjoint $\mathscr G^c:Cart-\mathscr R\longrightarrow Cart^C-\mathscr R$. 
\end{thm}

\begin{proof} We consider $\mathscr M\in Cart^C-\mathscr R$. We claim that $\mathscr F(\mathscr M)\in Mod-\mathscr R$ actually lies in the subcategory  $Cart-\mathscr R$. Indeed, for $\alpha:x\longrightarrow y$ in $\mathscr X$, we have $\mathscr F(M)_\alpha:\mathscr F_x(\mathscr M_x)=\mathscr M_x\longrightarrow\alpha_\ast\mathscr M_y=\alpha_\ast \mathscr F_y(\mathscr M_y)$ in $\mathbf M_{\mathscr R_x}$. By adjunction, this corresponds to a morphism $\alpha^\ast\mathscr M_x\longrightarrow \mathscr M_y$ in $\mathbf M_{\mathscr R_y}$. But since $\mathscr M\in Cart^C-\mathscr R$, we already know that  $\alpha^\ast\mathscr M_x\longrightarrow \mathscr M_y$ is an isomorphism. Hence, $\mathscr F^c(\mathscr M):=
\mathscr F(\mathscr M)\in Cart-\mathscr R$. 

\smallskip
We also notice that $Cart^C-\mathscr R$ is closed under taking colimits in $Mod^C-\mathscr R$. Then  $\mathscr F^c:Cart^C-\mathscr R\longrightarrow Cart-\mathscr R$ preserves colimits and we know from Theorem \ref{T6.10} that both $Cart^C-\mathscr R$ and $Cart-\mathscr R$ are Grothendieck categories. It now follows from \cite[Proposition 8.3.27]{KSch} that $\mathscr F^c$ has a right adjoint. 
\end{proof}

\begin{thm}\label{P7.9}  Let $\mathscr X$ be a poset, $C$ be a right semiperfect $K$-coalgebra and $\mathscr R:\mathscr X\longrightarrow \mathscr Ent_C$ be an entwined $C$-representation
that is also flat. Suppose there exists
$\theta\in V_1$ such that 
\begin{equation}\label{eq7.19}
\theta_x(r)(c_1\otimes c_2)=\varepsilon_C(c)\cdot id_r
\end{equation} for every $x\in \mathscr X$, $r\in\mathscr R_x$ and $c\in C$. 
Then, $\mathscr F^c:Cart^C-\mathscr R\longrightarrow Cart-\mathscr R$ is separable.
\end{thm} 

\begin{proof} From Theorem \ref{T7.7}, it follows that $\mathscr F:Mod^C-\mathscr R\longrightarrow Mod-\mathscr R$ is separable. In other words, for any $\mathscr M$, $\mathscr N\in 
Mod^C-\mathscr R$, the canonical morphism $Mod^C-\mathscr R(\mathscr M,\mathscr N)\longrightarrow Mod-\mathscr R(\mathscr F(\mathscr M),\mathscr F(\mathscr N))$ is a split monomorphism. Since $Cart^C-\mathscr R$ and $Cart-\mathscr R$ are full subcategories of $Mod^C-\mathscr R$ and $Mod-\mathscr R$ respectively and $\mathscr F^c$ is a restriction
of $\mathscr F$, the result follows.

\end{proof} 

\section{Separability of the functor $\mathscr G:Mod-\mathscr R\longrightarrow Mod^C-\mathscr R$}

We continue with $\mathscr X$ being a poset, $C$ being a right semiperfect coalgebra and let $\mathscr R:\mathscr X\longrightarrow \mathscr Ent_C$ be an entwined $C$-representation. In this section, we will give conditions for the right adjoint $\mathscr G:Mod-\mathscr R\longrightarrow Mod^C-\mathscr R$ to be separable. 

\smallskip Putting $C=K$ in Proposition \ref{P5.3}, we see that for each $x\in \mathscr X$, there is a functor $ex_x:\mathbf M_{\mathscr R_x}\longrightarrow Mod-\mathscr R$ having
right adjoint $ev_x:Mod-\mathscr R\longrightarrow \mathbf M_{\mathscr R_x}$. 
In a manner similar to Proposition \ref{P7.25}, we now can show that a natural transformation $\omega\in Nat(1_{Mod-\mathscr R},\mathscr F\mathscr G)$ consists of a collection of natural transformations $\{\omega_x\in
Nat(1_{\mathbf M_{\mathscr R_x}},\mathscr F_x\mathscr G_x)\}_{x\in \mathscr X}$ such that for any $\alpha:x\longrightarrow y$ in $\mathscr X$ and any $\mathscr N\in Mod-\mathscr R$, we have the following commutative diagram
\begin{equation}\label{eq8.1}
\begin{CD}
\mathscr N_x @>\omega_x(\mathscr N_x)>> \mathscr F_x\mathscr G_x(\mathscr N_x)\\
@V\mathscr N_\alpha VV @VV\mathscr F_x\mathscr G_x(\mathscr N_\alpha)V \\
\alpha_\ast\mathscr N_y @>\alpha_\ast\omega_y(\mathscr N_y)>> \alpha_\ast\mathscr F_y\mathscr G_y(\mathscr N_y)\\
\end{CD}
\end{equation} Here, $\omega_x\in
Nat(1_{\mathbf M_{\mathscr R_x}},\mathscr F_x\mathscr G_x)$ is determined by setting 
\begin{equation}\label{Yeq7.35b}
\omega_x(\mathcal N):=\omega(ex_x(\mathcal N))_x: (ex_x(\mathcal N))_x=\mathcal N\longrightarrow \mathscr F_x\mathscr G_x(\mathcal N)=\mathscr F_x\mathscr G_x((ex_x(\mathcal N))_x)\end{equation} for $\mathcal N\in \mathbf M_{\mathscr R_x}$. As in the proof of Proposition \ref{P7.25}, we can also show that
\begin{equation}\label{Ye8.15}
\omega_x(\mathscr N_x)=\omega(ex_x(\mathscr N_x))_x=\omega(\mathscr N)_x
\end{equation} for any $\mathscr N\in Mod-\mathscr R$ and $x\in \mathscr X$. More explicitly, for each $x\in \mathscr X$ and $r\in \mathscr R_x$, we have a commutative diagram
\begin{equation}\label{eq8.2}
\begin{CD}
\mathscr N_x(r) @>(\omega_x(\mathscr N_x))(r)>> (\mathscr F_x\mathscr G_x(\mathscr N_x))(r)=\mathscr N_x(r)\otimes C\\
@V\mathscr N_\alpha(r)VV @VV(\mathscr F_x\mathscr G_x(\mathscr N_\alpha))(r)V\\
\mathscr N_y(\alpha(r))=(\alpha_\ast\mathscr N_y)(r)@>(\alpha_\ast\omega_y(\mathscr N_y))(r)>=\omega_y(\mathscr N_y)(\alpha(r))> ( \alpha_\ast\mathscr F_y\mathscr G_y(\mathscr N_y))(r)=\mathscr N_y(\alpha(r))\otimes C \\
\end{CD}
\end{equation} We will now give another interpretation for the space $Nat(1_{Mod-\mathscr R},\mathscr F\mathscr G)$. For this, we consider a collection $\eta=\{\eta_x(s,r):H_r^x(s)=\mathscr R_x(s,r)\longrightarrow H_r^x(s)\otimes C=\mathscr R_x(s,r)\otimes C:f\mapsto \hat{f}\otimes c_f\}_{x\in \mathscr X,r,s\in \mathscr R_x}$ of $K$-linear maps satisfying the following conditions:

\smallskip
(1) Fix $x\in \mathscr X$. Then, for $s' \xrightarrow{h} s \xrightarrow{f} r \xrightarrow{g}  r'$ in $\mathscr R_x$, we have
\begin{equation}\label{8.3eta}
\eta_x(s',r')(gfh)=\sum \widehat{gfh}\otimes c_{gfh}=g\hat{f}h_{\psi_x}\otimes c_f^{\psi_x}\in \mathscr R_x(s',r')\otimes C
\end{equation}
(2) For $\alpha:x\longrightarrow y$ in $\mathscr X$ and $f\in \mathscr R_x(s,r)$ we have
\begin{equation}\label{8.4eta}
\alpha(\hat{f})\otimes c_f=\widehat{\alpha(f)}\otimes c_{\alpha(f)} \in \mathscr R_y(\alpha(s),\alpha(r))\otimes C
\end{equation} The space of all such $\eta$ will be denoted by $W_1$. 
We note that condition (1) is equivalent to saying that for each $x\in \mathscr X$, the element $\eta_x=\{\eta_x(s,r):\mathscr R_x(s,r)\longrightarrow \mathscr R_x(s,r)\otimes C:f\mapsto \hat{f}\otimes c_f\}_{r,s\in \mathscr R_x}\in Nat(H^x,H^x\otimes C)$, i.e., $\eta_x$ is a morphism in the category of $\mathscr R_x$-bimodules (functors $\mathscr R_x^{op}
\otimes \mathscr R_x\longrightarrow Vect_K$). Here $H^x$ is the canonical $\mathscr R_x$-bimodule that takes a pair of objects $(s,r)\in Ob(\mathscr R_x^{op}
\otimes \mathscr R_x)$ to $\mathscr R_x(s,r)$. Further, $H^x\otimes C$ is the $\mathscr R_x$-bimodule defined by setting
\begin{equation}
(H^x\otimes C)(s,r)=\mathscr R_x(s,r)\otimes C\qquad (H^x\otimes C)(h,g)(f\otimes c)=gfh_{\psi_x}\otimes c^{\psi_x}
\end{equation}  for $s' \xrightarrow{h} s \xrightarrow{f} r \xrightarrow{g}  r'$ in $\mathscr R_x$ and $c\in C$.

\begin{lem}\label{Lem8.1} There is a canonical morphism $Nat(1_{Mod-\mathscr R},\mathscr F\mathscr G)\longrightarrow W_1$. 

\end{lem}

\begin{proof} As mentioned above, any $\omega\in Nat(1_{Mod-\mathscr R},\mathscr F\mathscr G)$ corresponds to  a collection of natural transformations $\{\omega_x\in
Nat(1_{\mathbf M_{\mathscr R_x}},\mathscr F_x\mathscr G_x)\}_{x\in \mathscr X}$ satisfying \eqref{eq8.1}. From the proof of \cite[Proposition 3.10]{BBR}, we know that each $\omega_x\in
Nat(1_{\mathbf M_{\mathscr R_x}},\mathscr F_x\mathscr G_x)$ corresponds to $\eta_x\in  Nat(H^x,H^x\otimes C)$ determined by setting
\begin{equation}\label{eqr8.6}
\eta_x(s,r):H_r^x(s)=\mathscr R_x(s,r)\longrightarrow H_r^x(s)\otimes C=\mathscr R_x(s,r)\otimes C \qquad \eta_x(s,r):=\omega_x(H^x_r)(s)
\end{equation} for $r$, $s\in \mathscr R_x$. Here, $H^x_r$ is the right $\mathscr R_x$-module $H_r^x:=\mathscr R_x(\_\_,r):\mathscr R_x^{op}\longrightarrow Vect_K$. We now consider
$\alpha:x\longrightarrow y$ in $\mathscr X$ and some $f\in \mathscr R_x(s,r)$. By applying Lemma \ref{L5.1} with $C=K$, we have $ex_x(H^x_r)\in Mod-\mathscr R$
which satisfies $(ex_x(H^x_r))_y=\alpha^\ast H^x_r=H^y_{\alpha(r)}$. Setting $\mathscr N=ex_x(H^x_r)$ in \eqref{eq8.2}, we have 
\begin{equation}\label{eq8.7}
\begin{CD}
\mathscr N_x(s)=H^x_r(s) @>(\omega_x(H^x_r))(s)>=\eta_x(s,r)> (\mathscr F_x\mathscr G_x(\mathscr N_x))(s)=H^x_r(s)\otimes C\\
@V\mathscr N_\alpha(s)VV @VV(\mathscr F_x\mathscr G_x(\mathscr N_\alpha))(s)V\\
\mathscr N_y(\alpha(s))=H^y_{\alpha(r)}(\alpha(s))@>\eta_y(\alpha(s),\alpha(r))>=\omega_y(H^y_{\alpha(r)})(\alpha(s))>\mathscr N_y(\alpha(s))\otimes C =H^y_{\alpha(r)}(\alpha(s))
\otimes C\\
\end{CD}
\end{equation} It follows that that the collection $\eta_x(s,r)$ satisfies condition \eqref{8.4eta}. This proves the result. 
\end{proof}

\begin{thm}\label{Pro8.2} The spaces $Nat(1_{Mod-\mathscr R},\mathscr F\mathscr G)$ and $ W_1$ are isomorphic.
\end{thm}

\begin{proof} We consider an element $\eta\in W_1$. As mentioned before, this gives a collection $\{\eta_x\in  Nat(H^x,H^x\otimes C)\}_{x\in\mathscr X}$ satisfying the compatibility condition in \eqref{8.4eta}. From the proof of \cite[Proposition 3.10]{BBR}, it follows that each $\eta_x$ corresponds to a natural transformation $\omega_x\in
Nat(1_{\mathbf M_{\mathscr R_x}},\mathscr F_x\mathscr G_x)$ which satisfies $\omega_x(H^x_r)(s)=\eta_x(s,r)$ for $r$, $s\in \mathscr R_x$. We claim that the collection
$\{\omega_x\}_{x\in \mathscr X}$ satisfies the compatibility condition in \eqref{eq8.1} for each $\mathscr N\in Mod-\mathscr R$, thus determining an element $\omega\in Nat(1_{Mod-\mathscr R},\mathscr F\mathscr G)$. 

\smallskip
We start with $\mathscr N=ex_x(H^x_r)$ for some $x\in \mathscr X$ and $r\in \mathscr R_x$. We consider a morphism $\alpha:y\longrightarrow z$ in $\mathscr X$. If $x\not\leq y$, then
$\mathscr N_y=0$ and the condition in \eqref{eq8.1} is trivially satisfied. Otherwise, let $\beta:x\longrightarrow y$ in $\mathscr X$ and set $s=\beta(r)$. In particular, $\mathscr N_y=\beta^\ast H^x_r=H^y_{\beta(r)}=H^y_s$ and $\mathscr N_z=H^z_{\alpha\beta(r)}=H^z_{\alpha(s)}$. Applying the condition \eqref{8.4eta}, we see that the following diagram is commutative for any $s'\in \mathscr R_y$: 
\begin{equation}
\begin{CD}
\mathscr R_y(s',s)=\mathscr N_y(s')@>\eta_y(s',s)>=\omega_y(\mathscr N_y)(s')> \mathscr N_y(s')\otimes C=\mathscr R_y(s',s)\otimes C\\
@V\mathscr N_\alpha(s')VV @VV(\mathscr F_y\mathscr G_y(\mathscr N_\alpha))(s')V \\
\mathscr R_z(\alpha(s'),\alpha(s))=\mathscr N_z(\alpha(s'))@>\eta_z(\alpha(s'),\alpha(s))>=\omega_z(\mathscr N_z)(\alpha(s'))> \mathscr N_z(\alpha(s'))\otimes C=\mathscr R_z(\alpha(s'),
\alpha(s))\otimes C \\
\end{CD}
\end{equation}
In other words, the condition in \eqref{eq8.1} is satisfied for $\mathscr N=ex_x(H^x_r)$. 
From Theorem \ref{T5.5}, we know that the collection 
\begin{equation}\label{8gen}
\{\mbox{$ex_x(H_r^x)$ $\vert$ $x\in \mathscr X$, $r\in \mathscr R_x$}\}
\end{equation} is a set of generators for $Mod-\mathscr R$. Accordingly, for any $\mathscr N'\in Mod-\mathscr R$, we can choose an epimorphism $\phi:\mathscr N\longrightarrow \mathscr N'$ where $\mathscr N$ is a direct sum of copies of objects in \eqref{8gen}. Then, $\mathscr N$ satisfies \eqref{eq8.1} and we have commutative diagrams 
\begin{equation}
\begin{CD}
\mathscr N_y  @>\mathscr N_\alpha>> \alpha_\ast\mathscr N_z @>\alpha_\ast\omega_z(\mathscr N_z)>> \alpha_\ast\mathscr F_z\mathscr G_z(\mathscr N_z) \\
@V\phi_yVV @V\alpha_\ast\phi_zVV @VV\alpha_\ast\mathscr F_z\mathscr G_z(\phi_z)V \\
\mathscr N'_y  @>\mathscr N'_\alpha>> \alpha_\ast\mathscr N'_z @>\alpha_\ast\omega_z(\mathscr N'_z)>> \alpha_\ast\mathscr F_z\mathscr G_z(\mathscr N'_z) \\
\end{CD}
\end{equation}
\begin{equation}
\begin{array}{c}
 \begin{CD}
\mathscr N_y  @>\omega_y(\mathscr N_y)>> \mathscr F_y\mathscr G_y(\mathscr N_y) @>\mathscr F_y\mathscr G_y(\mathscr N_\alpha)>> \alpha_\ast\mathscr F_z\mathscr G_z(\mathscr N_z) \\
@V\phi_yVV @V\mathscr F_y\mathscr G_y(\phi_y)VV @VV\alpha_\ast\mathscr F_z\mathscr G_z(\phi_z)V \\
\mathscr N'_y  @>\omega_y(\mathscr N'_y)>> \mathscr F_y\mathscr G_y(\mathscr N'_y) @>\mathscr F_y\mathscr G_y(\mathscr N'_\alpha)>> \alpha_\ast\mathscr F_z\mathscr G_z(\mathscr N'_z) \\
\end{CD}\qquad \qquad 
\begin{CD}
\mathscr N_y @>\omega_y(\mathscr N_y)>> \mathscr F_y\mathscr G_y(\mathscr N_y)\\
@V\mathscr N_\alpha VV @VV\mathscr F_y\mathscr G_y(\mathscr N_\alpha)V \\
\alpha_\ast\mathscr N_z @>\alpha_\ast\omega_z(\mathscr N_z)>> \alpha_\ast\mathscr F_z\mathscr G_z(\mathscr N_z)\\
\end{CD}\\
\end{array}
\end{equation}
for any $\alpha:y\longrightarrow z$ in $\mathscr X$. Since $\phi_y:\mathscr N_y\longrightarrow \mathscr N'_y$ is an epimorphism, it follows that $\mathscr N'$ also satisfies the condition in \eqref{eq8.1}. This gives a morphism $W_1\longrightarrow Nat(1_{Mod-\mathscr R},\mathscr F\mathscr G)$. It may be verified that this is inverse to the morphism $Nat(1_{Mod-\mathscr R},\mathscr F\mathscr G)\longrightarrow W_1$ in Lemma \ref{Lem8.1}, which proves the result. 

\end{proof}

We will now give conditions for the functor $\mathscr G:Mod-\mathscr R\longrightarrow Mod^C-\mathscr R$ to be separable. Since $\mathscr G$ has a left adjoint, it follows (see \cite[Theorem 1.2]{Raf}) that $\mathscr G$ is separable if and only if there exists a natural transformation $\omega\in Nat(1_{Mod-\mathscr R},\mathscr F\mathscr G)$ such that $\nu\circ \omega=1_{Mod-\mathscr R}$, where
$\nu$ is the counit of the adjunction. 

\begin{Thm}\label{T8.3} Let $\mathscr X$ be a partially ordered set, $C$ be a right semiperfect $K$-coalgebra and let $\mathscr R:\mathscr X\longrightarrow \mathscr Ent_C$ be an entwined $C$-representation. Then,  the functor
$\mathscr G:Mod-\mathscr R\longrightarrow Mod^C-\mathscr R$  is separable if and only if there exists $\eta\in W_1$ such that 
\begin{equation}\label{cond8.3}
\begin{CD} id=(id\otimes \varepsilon_C)\circ \eta_x(s,r):\mathscr R_x(s,r)@>\eta_x(s,r)>> \mathscr R_x(s,r)\otimes C@>(id\otimes \varepsilon_C)>> \mathscr R_x(s,r)\end{CD}
\end{equation}
for each $x\in \mathscr X$ and $s$, $r\in \mathscr R_x$.
\end{Thm}

\begin{proof} First, we suppose that $\mathscr G$ is separable, i.e., there exists a natural transformation $\omega\in Nat(1_{Mod-\mathscr R},\mathscr F\mathscr G)$ such that $\nu\circ \omega=1_{Mod-\mathscr R}$. Using Proposition \ref{Pro8.2}, we consider $\eta\in W_1$ corresponding to $\omega$.

\smallskip
By definition, the counit $\nu$ of the adjunction $(\mathscr F,\mathscr G)$ is described as follows: for any $\mathscr N\in Mod-\mathscr R$, we have
\begin{equation}
\nu(\mathscr N)_x(s):\mathscr N_x(s)\otimes C\longrightarrow \mathscr N_x(s) \qquad n\otimes c\mapsto n\varepsilon_C(c) 
\end{equation} for each $x\in \mathscr X$, $s\in \mathscr R_x$. We choose $x\in\mathscr X$, $r\in \mathscr R_x$ and set $\mathscr N=ex_x(H^x_r)$. Since $\nu\circ \omega=1_{Mod-\mathscr R}$, it now follows from \eqref{eqr8.6} that
\begin{equation}\label{cond8.13}
id=\nu(ex_x(H^x_r))_x(s)\circ \omega(ex_x(H^x_r))_x(s)=(id\otimes \varepsilon_C)\circ \omega_x(H^x_r)(s)=(id\otimes \varepsilon_C)\circ \eta_x(s,r)
\end{equation}
Conversely, suppose that we have $\eta\in W_1$ such that the condition in \eqref{cond8.3} is satisfied. Using the isomorphism in Proposition \ref{Pro8.2}, we obtain the natural transformation $\omega\in Nat(1_{Mod-\mathscr R},\mathscr F\mathscr G)$ corresponding to $\eta$. Then, it is clear from \eqref{cond8.13} that $\nu(\mathscr N)\circ \omega(\mathscr N)=id$
for $\mathscr N=ex_x(H^x_r)$. Since $\{\mbox{$ex_x(H_r^x)$ $\vert$ $x\in \mathscr X$, $r\in \mathscr R_x$}\}$ is a set of generators for $Mod-\mathscr R$, it follows that for any
$\mathscr N'\in Mod-\mathscr R$, there is an epimorphism $\phi:\mathscr N\longrightarrow \mathscr N'$ such that $\nu(\mathscr N)\circ \omega(\mathscr N)=id$. We now consider the commutative diagram
\begin{equation}\label{cd8.14}
\begin{CD}
\mathscr N @>\omega(\mathscr N)>> \mathscr F\mathscr G(\mathscr N) @>\nu(\mathscr N)>> \mathscr N \\
@V\phi VV @V\mathscr F\mathscr G(\phi) VV @VV\phi V\\
\mathscr N' @>\omega(\mathscr N')>> \mathscr F\mathscr G(\mathscr N') @>\nu(\mathscr N')>> \mathscr N' \\
\end{CD}
\end{equation} Since the upper horizontal composition in \eqref{cd8.14} is the identity and $\phi$ is an epimorphism, it follows that $\nu(\mathscr N')\circ \omega(\mathscr N')=id$. This proves the result. 
\end{proof}

\section{$(\mathscr F,\mathscr G)$ as a Frobenius pair}

In Sections 7 and 8, we have given conditions for the functor $\mathscr F:Mod^C-\mathscr R\longrightarrow Mod-\mathscr R$ and its right adjoint $\mathscr G:Mod-\mathscr R
\longrightarrow Mod^C-\mathscr R$ to be separable. In this section, we will give necessary and sufficient conditions for $(\mathscr F,\mathscr G)$ to be a Frobenius pair, i.e.,
$\mathscr G$ is both a right and a left adjoint of $\mathscr F$. First, we note that it follows from the characterization of Frobenius pairs (see for instance, \cite[$\S$ 1]{uni}) that
$(\mathscr F,\mathscr G)$ is a Frobenius pair if and only if there exist $\upsilon\in Nat(\mathscr G\mathscr F,1_{Mod^C-\mathscr R})$ and $\omega\in Nat(1_{Mod-\mathscr R},\mathscr F
\mathscr G)$ 
 such that 
 \begin{equation}\label{eq9.1}
\mathscr F(\upsilon(\mathscr M))\circ\omega(\mathscr F(\mathscr M))= id_{\mathscr F(\mathscr M)}\qquad \upsilon(\mathscr G(\mathscr N))\circ \mathscr G(\omega(\mathscr N))=id_{\mathscr G(\mathscr N)}
 \end{equation}
 for any $\mathscr M\in Mod^C-\mathscr R$ and $\mathscr N\in Mod-\mathscr R$. Equivalently, for each $x\in \mathscr X$, we must have
 \begin{equation}\label{eq9.15}
 \begin{array}{c}
(\mathscr F(\upsilon(\mathscr M)))_x\textrm{ }\circ\textrm{ }\omega(\mathscr F(\mathscr M))_x= \mathscr F_x(\upsilon_x(\mathscr M_x))\textrm{ }\circ\textrm{ }\omega_x(\mathscr F_x(\mathscr M_x))= id_{\mathscr F_x(\mathscr M_x)}\\
 \upsilon(\mathscr G(\mathscr N))_x\textrm{ }\circ\textrm{ } \mathscr G(\omega(\mathscr N))_x=\upsilon_x(\mathscr G_x(\mathscr N_x))\textrm{ }\circ\textrm{ }\mathscr G_x(\omega_x(\mathscr N_x))=id_{\mathscr G_x(\mathscr N_x)} \\
 \end{array}
\end{equation}  for any $\mathscr M\in Mod^C-\mathscr R$ and $\mathscr N\in Mod-\mathscr R$. 
 
 \begin{Thm}\label{T9.1} Let $\mathscr X$ be a partially ordered set, $C$ be a right semiperfect $K$-coalgebra and let $\mathscr R:\mathscr X\longrightarrow \mathscr Ent_C$ be an entwined $C$-representation. Let $\mathscr F:Mod^C-\mathscr R\longrightarrow Mod-\mathscr R$  be the forgetful functor and  $\mathscr G:Mod-\mathscr R
\longrightarrow Mod^C-\mathscr R$ its right adjoint. Then, $(\mathscr F,\mathscr G)$ is a Frobenius pair if and only if 
there exist $\theta\in V_1$ and $\eta\in W_1$ such that
\begin{equation}\label{eq9.2}
\varepsilon_C(d)f=\sum \widehat{f}\circ \theta_x(r)(c_f\otimes d)\qquad \varepsilon_C(d)f=\sum \widehat{f_{\psi_x}} 
\circ \theta_x(r)(d^{\psi_x}\otimes c_f)
\end{equation} for every $x\in \mathscr X$, $r\in \mathscr R_x$, $f\in \mathscr R_x(r,s)$ and $d\in C$, where $\eta_x(r,s)(f)=\widehat{f}\otimes c_f$. 
 
 \end{Thm}

\begin{proof} We suppose there exist $\theta\in V_1$ and $\eta\in W_1$  satisfying \eqref{eq9.2} and consider $\mathscr M\in Mod^C-\mathscr R$, $\mathscr N\in Mod-\mathscr R$. Using the isomorphisms in Proposition \ref{P7.6} and Proposition \ref{Pro8.2}, we obtain $\upsilon\in Nat(\mathscr G\mathscr F,1_{Mod^C-\mathscr R})$ and $\omega\in Nat(1_{Mod-\mathscr R},\mathscr F
\mathscr G)$  corresponding to $\theta$ and $\eta$ respectively. 

\smallskip For fixed 
$x\in \mathscr X$, it follows that $\theta_x=\{\theta_x(r):C\otimes C \longrightarrow \mathscr R_x(r,r)\}_{r\in \mathscr R_x}$ and 
the $\mathscr R_x$-bimodule morphism $\eta_x\in Nat(H^x,H^x\otimes C)$ satisfy the conditions in \cite[Theorem 3.14]{BBR}. Hence, we have \begin{equation}\label{eq9.4}
\mathscr F_x(\upsilon_x(\mathcal M))\textrm{ }\circ\textrm{ }\omega_x(\mathscr F_x(\mathcal  M))= id_{\mathscr F_x(\mathcal M)}\qquad \upsilon_x(\mathscr G_x(\mathcal N))\textrm{ }\circ\textrm{ }\mathscr G_x(\omega_x(\mathcal N))=id_{\mathscr G_x(\mathcal N)} 
\end{equation}  for any $\mathcal M\in \mathbf M^C_{\mathscr R_x}(\psi_x)$ and $\mathcal N\in \mathbf M_{\mathscr R_x}$.  In particular, 
\eqref{eq9.15} holds for $\mathscr M_x\in \mathbf M^C_{\mathscr R_x}(\psi_x)$ and $\mathscr N_x\in \mathbf M_{\mathscr R_x}$. 

\smallskip
Conversely, suppose that $(\mathscr F,\mathscr G)$ is a Frobenius pair. Then, there exist  $\upsilon\in Nat(\mathscr G\mathscr F,1_{Mod^C-\mathscr R})$ and $\omega\in Nat(1_{Mod-\mathscr R},\mathscr F
\mathscr G)$   satisfying \eqref{eq9.15} for each $x\in \mathscr X$. Again using the  isomorphisms in Proposition \ref{P7.6} and Proposition \ref{Pro8.2}, we obtain corresponding $\theta\in V_1$ and $\eta\in W_1$. 

\smallskip
We now consider $\mathcal M\in \mathbf M^C_{\mathscr R_x}(\psi_x)$ and $\mathcal N\in \mathbf M_{\mathscr R_x}$. Applying \eqref{eq9.15} with $\mathscr M=ex^C_x(\mathcal M)$
and $\mathscr N=ex_x(\mathcal N)$, we have
\begin{equation}
 \mathscr F_x(\upsilon_x(\mathcal M))\textrm{ }\circ\textrm{ }\omega_x(\mathscr F_x(\mathcal M))= id_{\mathscr F_x(\mathcal M)}\qquad \upsilon_x(\mathscr G_x(\mathcal N))\textrm{ }\circ\textrm{ }\mathscr G_x(\omega_x(\mathcal N))=id_{\mathscr G_x(\mathcal N)}
\end{equation} It now follows from \cite[Theorem 3.14]{BBR} that  $\theta_x=\{\theta_x(r):C\otimes C \longrightarrow \mathscr R_x(r,r)\}_{r\in \mathscr R_x}$ and the $\mathscr R_x$-bimodule morphism 
$\eta_x\in Nat(H^x,H^x\otimes C)$ satisfy \eqref{eq9.2}. This proves the result.

\end{proof}

\begin{cor}\label{C9.2} Let $(\mathscr F,\mathscr G)$ be a Frobenius pair. Then, for each $x\in \mathscr X$, $(\mathscr F_x,\mathscr G_x)$ is a Frobenius pair of adjoint functors.
\end{cor} 

\begin{proof}
This is immediate from \eqref{eq9.4}. 
\end{proof}

We consider $\alpha:x\longrightarrow y$ in $\mathscr X$. In \eqref{cd7.1}, we observed directly that the functors  $\{\mathscr F_x:\mathbf M^C_{\mathscr R_x}(\psi_x)\longrightarrow \mathbf M_{\mathscr R_x}\}_{x\in \mathscr X}$ commute with both $\alpha^\ast$ and 
$\alpha_\ast$, while the functors 
$\{\mathscr G_x:  \mathbf M_{\mathscr R_x}\longrightarrow \mathbf M^C_{\mathscr R_x}(\psi_x)\}_{x\in \mathscr X}$ commute only with $\alpha_\ast$. We will now give a sufficient condition
for the functors $\{\mathscr G_x:  \mathbf M_{\mathscr R_x}\longrightarrow \mathbf M^C_{\mathscr R_x}(\psi_x)\}_{x\in \mathscr X}$ to  commute with $\alpha^\ast$.

\begin{lem}\label{L9.3}  Let $(\mathscr F,\mathscr G)$ be a Frobenius pair. Then, for any $\alpha:x\longrightarrow y$ in $\mathscr X$, we have a commutative diagram
\begin{equation}\label{eq9.6}
 \begin{CD}
\mathbf M _{\mathscr R_x}@>\alpha^\ast >> \mathbf M_{\mathscr R_y}\\
@V\mathscr G_xVV @VV\mathscr G_yV \\
\mathbf M^C_{\mathscr R_x}(\psi_x)@>\alpha^\ast >> \mathbf M^C_{\mathscr R_y}(\psi_y)\\
\end{CD}
\end{equation}

\end{lem} 

\begin{proof} 
For $\mathcal M\in \mathbf M _{\mathscr R_x}$, we will show that $\mathscr G_y\alpha^\ast(\mathcal M)=\alpha^\ast\mathscr G_x(\mathcal M)\in  \mathbf M^C_{\mathscr R_y}(\psi_y)$. From Corollary \ref{C9.2} we know that each $(\mathscr F_x,\mathscr G_x)$ is a Frobenius pair of adjoint functors. Using this fact and the commutative diagrams in \eqref{cd7.1}, we now have that for any $\mathcal N\in \mathbf M^C_{\mathscr R_y}(\psi_y)$:
\begin{equation}
\mathbf M^C_{\mathscr R_y}(\psi_y)(\mathscr G_y\alpha^\ast(\mathcal M),\mathcal N)=\mathbf M_{\mathscr R_x}(\mathcal M,\alpha_\ast\mathscr F_y(\mathcal N))=\mathbf M_{\mathscr R_x}(\mathcal M,\mathscr F_x\alpha_\ast(\mathcal N))=\mathbf M^C_{\mathscr R_y}(\psi_y)(\alpha^\ast\mathscr G_x(\mathcal M),\mathcal N)
\end{equation}
\end{proof}

\begin{thm}\label{P9.4}  Let $(\mathscr F,\mathscr G)$ be a Frobenius pair.  Suppose that  $\mathscr R:\mathscr X\longrightarrow \mathscr Ent_C$  is flat. Then, $\mathscr G:Mod-\mathscr R\longrightarrow Mod^C-\mathscr R$ restricts to a functor
$\mathscr G^c:Cart-\mathscr R\longrightarrow Cart^C-\mathscr R$. 
\end{thm} 

\begin{proof}
For  any $\mathscr N\in Cart-\mathscr R$, we claim that $\mathscr G(\mathscr N)\in Mod^C-\mathscr R$ actually lies in $Cart^C-\mathscr R$. By definition of $\mathscr G$, we have  for any $\alpha:x\longrightarrow y$, a morphism 
$\mathscr G(\mathscr N)_\alpha=\mathscr G_x(\mathscr N_\alpha):\mathscr G_x(\mathscr N_x)\longrightarrow \mathscr G_x(\alpha_\ast(\mathscr N_y))=\alpha_\ast(\mathscr G_y(\mathscr N_y))$ in $\mathbf M^C_{\mathscr R_x}(\psi_x)$ which corresponds to a morphism $\mathscr G(\mathscr N)^\alpha:\alpha^\ast(\mathscr G_x(\mathscr N_x))\longrightarrow 
\mathscr G_y(\mathscr N_y)$ in $\mathbf M^C_{\mathscr R_y}(\psi_y)$. Since $(\mathscr F,\mathscr G)$ is a Frobenius pair, it follows  from Lemma \ref{L9.3} that $\mathscr G_y\alpha^\ast(\mathscr N_x)=\alpha^\ast\mathscr G_x(\mathscr N_x)\in  \mathbf M^C_{\mathscr R_y}(\psi_y)$. Since $\mathcal N$ is cartesian, we know that $\alpha^\ast\mathscr N_x$ is isomorphic to 
$\mathscr N_y$ and hence $\mathscr G(\mathscr N)^\alpha=\mathscr G_y(\mathscr N^\alpha)$ is an isomorphism. 
\end{proof}

\begin{cor}\label{C9.5} Let $(\mathscr F,\mathscr G)$ be a Frobenius pair.  Suppose that  $\mathscr R:\mathscr X\longrightarrow \mathscr Ent_C$  is flat.  Then, $(\mathscr F^c,\mathscr G^c)$ is a Frobenius pair of adjoint functors between $Cart^C-\mathscr R$ and
$Cart-\mathscr R$.

\end{cor} 

\begin{proof} From Proposition \ref{P7.8}, we know that $\mathscr F:Mod^C-\mathscr R\longrightarrow Mod-\mathscr R$ restricts to a functor $\mathscr F^c:Cart^C-\mathscr R\longrightarrow 
Cart-\mathscr R$. From Proposition \ref{P9.4}, we know that $\mathscr G:Mod-\mathscr R\longrightarrow Mod^C-\mathscr R$ restricts to a functor
$\mathscr G^c:Cart-\mathscr R\longrightarrow Cart^C-\mathscr R$ on the full subcategories of cartesian modules.  Since $\mathscr G$ is both right and left adjoint to $\mathscr F$, it is clear that $\mathscr G^c$ is both right and left adjoint to $\mathscr F^c$. 

\end{proof} 

\section{Constructing entwined representations}

In this final section, we will give examples of  how to construct entwined representations and describe modules over them. Let $(\mathcal R,C,\psi)$ be an entwining structure. Then, we consider the   
$K$-linear category $(C,\mathcal R)_\psi$ defined as follows
\begin{equation}
Ob((C,\mathcal R)_\psi)=Ob(\mathcal R)\qquad (C,\mathcal R)_\psi(s,r):=Hom_K(C,\mathcal R(s,r))
\end{equation} for $s$, $r\in \mathcal R$. The composition in  $(C,\mathcal R)_\psi$ is as follows: given $\phi:C\longrightarrow \mathcal R(s,r)$ and $\phi':C\longrightarrow \mathcal R(t,s)$ respectively
in $(C,\mathcal R)_\psi(s,r)$ and $(C,\mathcal R)_\psi(t,s)$, we set
\begin{equation}
\phi\ast\phi': C\longrightarrow \mathcal R(t,r)\qquad c\mapsto \sum \phi(c_2)_\psi\circ \phi'(c_1^\psi)
\end{equation}

\begin{lem}\label{L99.1}
Let $(\mathcal R,C,\psi)$ be an entwining structure. Then, there is a canonical functor $P_\psi: \mathbf M_{\mathcal R}^C(\psi)\longrightarrow \mathbf M_{(C,\mathcal R)_\psi}$. 
\end{lem}

\begin{proof}
We consider $\mathcal M\in \mathbf M_{\mathcal R}^C(\psi)$. We will define $\mathcal N=P_\psi(\mathcal M)\in  \mathbf M_{(C,\mathcal R)_\psi}$ by setting $\mathcal N(r):=\mathcal M(r)$
for each $r\in (C,\mathcal R)$. Given $\phi:C\longrightarrow \mathcal R(s,r)$ in $(C,\mathcal R)_\psi(s,r)$, we define $m\ast \phi\in \mathcal N(s)=\mathcal M(s)$ by setting $m\ast \phi=
\sum m_0\phi(m_1)$. Here, $\rho_{\mathcal M(r)}(m)=\sum m_0\otimes m_1$ is the right $C$-comodule structure on $\mathcal M(r)$. 

\smallskip
For $\phi':C\longrightarrow \mathcal R(t,s)$ in  $(C,\mathcal R)_\psi(t,s)$, we now have
\begin{equation}
\begin{array}{ll}
m\ast (\phi\ast\phi') &  = \sum  m_0(\phi\ast \phi')(m_1) 
=\sum m_0\phi(m_{12})_\psi\phi'(m_{11}^\psi)=\sum m_0\phi(m_{2})_\psi\phi'(m_{1}^\psi) \\
&\\  & \\
(m\ast \phi)\ast \phi' & =\sum (m\ast \phi)_0\phi'((m\ast \phi)_1 ) 
 =\sum (m_0 \phi(m_1))_0\phi'((m_0 \phi(m_1))_1 ) \\
 &=\sum (m_{00}\phi(m_1)_\psi)\phi'(m_{01}^\psi)=\sum m_0\phi(m_{2})_\psi\phi'(m_{1}^\psi)  \\
\end{array}
\end{equation} This proves the result. 
\end{proof}

\begin{lem}\label{L99.2} Let $(\alpha,id):(\mathcal R,C,\psi)\longrightarrow (\mathcal S,C,\psi')$ be a morphism of entwining structures. Then, $P_{\psi}\circ (\alpha,id)_\ast=\alpha_\ast
\circ P_{\psi'}:\mathbf M_{\mathcal S}^C(\psi')\longrightarrow \mathbf M_{(C,\mathcal R)}$.
\end{lem}

\begin{proof}
We begin with $\mathcal N\in \mathbf M_{\mathcal S}^C(\psi')$. From the construction in Lemma \ref{L99.1}, it is clear that for any $r\in (C,\mathcal R)_\psi$, we have $(P_{\psi}\circ (\alpha,id)_\ast)(\mathcal N)(r)=(\alpha_\ast
\circ P_{\psi'})(\mathcal N)(r)=\mathcal N(\alpha(r))$. We set $\mathcal N_1:=(P_{\psi}\circ (\alpha,id)_\ast)(\mathcal N)$ and $\mathcal N_2:=(\alpha_\ast
\circ P_{\psi'})(\mathcal N)$ and consider $n\in \mathcal N_1(r)=\mathcal N_2(r)$ as well as $\phi:C\longrightarrow \mathcal R(s,r)$ in $(C,\mathcal R)_\psi(s,r)$. Then, in both $\mathcal N_1(s)$ and $\mathcal N_2(s)$, we have $n\ast \phi=\sum n_0 \alpha(\phi(n_1))$. This proves the result. 
\end{proof}

Now let $\mathscr X$ be a small category and $\mathscr R:\mathscr X\longrightarrow \mathscr Ent_C$ an entwined $C$-representation. By replacing each entwining structure
$(\mathscr R_x,C,\psi_x)$ with the category $(C,\mathscr R_x)_{\psi_x}$, we obtain an induced representation $(C,\mathscr R)_\psi:\mathscr X\longrightarrow \mathscr Lin$ (we recall that
$\mathscr Lin$ is the category of small $K$-linear categories).

\begin{thm}\label{P99.3} There is a canonical functor $Mod^C-\mathscr R\longrightarrow Mod-(C,\mathscr R)_\psi$. 
\end{thm}

\begin{proof} By definition, an object $\mathscr M\in Mod^C-\mathscr R$ consists of a collection $\{\mathscr M_x\in \mathbf M^C_{\mathscr R_x}(\psi_x)\}_{x\in \mathscr X}$ and for each
 $\alpha:x\longrightarrow y$ in $\mathscr X$, a morphism $\mathscr M_\alpha:\mathscr M_x\longrightarrow \alpha_\ast\mathscr M_y$ in $\mathbf M^C_{\mathscr R_x}(\psi_x)$. Applying the functors $P_{\psi_x}:\mathbf M^C_{\mathscr R_x}(\psi_x)\longrightarrow \mathbf M_{(C,\mathscr R_x)_{\psi_x}}$ for $x\in \mathscr X$ and using Lemma \ref{L99.2}, the result is now clear.
\end{proof}

Now let $C$ be finitely generated as a $K$-vector space and let $C^\ast$ denote its $K$-linear dual. Then, the canonical map $C^\ast\otimes V\longrightarrow Hom_K(C,V)$ is an isomorphism
for any vector space $V$. For an entwining structure $(\mathcal R,C,\psi)$, the category $(C,\mathcal R)_\psi$ can now be rewritten as $(C^\ast\otimes \mathcal R)_\psi$ where $(C^\ast\otimes \mathcal R)_\psi (s,r)=C^\ast\otimes \mathcal R(s,r)$ for $s$, $r\in Ob((C^\ast\otimes \mathcal R)_\psi)=Ob(\mathcal R)$. Given $c^\ast\otimes f\in C^\ast\otimes \mathcal R(s,r)$ and
$d^\ast\otimes g\in C^\ast\otimes \mathcal R(t,s)$, the composition in $(C^\ast\otimes \mathcal R)_\psi$ is expressed as 
\begin{equation}\label{eq99.47}
(c^\ast\otimes f)\circ (d^\ast \otimes g):C\longrightarrow \mathcal R(t,r)\qquad x\mapsto \sum c^\ast(x_2)d^\ast (x_1^\psi)(f_\psi\circ g)
\end{equation} for $x\in C$. It is important to note that when $f$ and $g$ are identity maps, the composition in \eqref{eq99.47} simplifies to 
\begin{equation}\label{eq99.48}
(c^\ast\otimes id_r)\circ (d^\ast \otimes id_r):C\longrightarrow \mathcal R(t,r)\qquad x\mapsto \sum c^\ast(x_2)d^\ast (x_1)id_r
\end{equation} In other words, for the canonical morphism  $C^\ast\longrightarrow C^\ast\otimes\mathcal R(r,r)$ given by $c^\ast\mapsto c^\ast\otimes id_r$
to be a morphism of algebras, we must use the opposite of the usual convolution product on $C^\ast$. 

\smallskip 
Similarly, given an entwined $C$-representation $\mathscr R:\mathscr X\longrightarrow \mathscr Ent_C$ with $C$ finitely generated as a $K$-vector space, we can  replace the  induced representation $(C,\mathscr R)_\psi:\mathscr X\longrightarrow \mathscr Lin$ by $(C^\ast\otimes\mathscr R)_\psi$. Then, $Mod-(C,\mathscr R)_\psi$ may be replaced by $Mod-(C^\ast\otimes \mathscr R)_\psi$. 

\begin{thm}\label{P99.4} Let $\mathscr X$ be a small category and $\mathscr R:\mathscr X\longrightarrow \mathscr Ent_C$ an entwined $C$-representation. Suppose that $C$ is finitely generated as a $K$-vector space. Then, the categories $Mod^C-\mathscr R$ and $Mod-(C^\ast\otimes \mathscr R)_\psi$ are equivalent. 
\end{thm}

\begin{proof} By Proposition \ref{P99.3}, we already know that any object in $Mod^C-\mathscr R$ may be equipped with a $(C^\ast\otimes\mathscr R)_\psi$-module structure. For the converse, we consider some $\mathscr M\in Mod-(C^\ast\otimes\mathscr R)_\psi$ and choose some $x\in \mathscr X$.

\smallskip
We make $\mathscr M_x$ into an $\mathscr R_x$-module as follows:  for $f\in \mathscr R_x(s,r)$ and $m\in \mathscr M_x(r)$, we set  $mf\in \mathscr M_x(s)$ to be $mf:=m(\varepsilon_C\otimes f)$. By considering the canonical morphism  $C^\ast\longrightarrow C^\ast\otimes\mathscr  R_x(r,r)$, it follows that the right  $(C^\ast\otimes \mathscr R_x)_{\psi_x}(r,r)$ module
$\mathscr M_x(r)$ carries a right $C^\ast$-module structure. As observed in \eqref{eq99.48}, here the product on $C^\ast$ happens to be the opposite of the usual convolution product. Hence, the right 
$C^\ast$-module structure on $\mathscr M_x(r)$ leads to a left $C^\ast$-module structure on $\mathscr M_x(r)$ when $C^\ast$ is equipped with the usual product. Since $C$ is finite dimensional, it is well known (see, for instance, \cite[$\S$ 2.2]{book3}) that we have an induced right $C$-comodule structure on $\mathscr M_x(r)$. It may be  verified by direct computation 
that $\mathscr M_x\in \mathbf M^C_{\mathscr R_x}(\psi_x)$. Finally, for a morphism $\alpha: x\longrightarrow y$ in $\mathscr X$, the map $\mathscr M_\alpha:\mathscr M_x
\longrightarrow \alpha_\ast\mathscr M_y$ in $\mathbf M_{(C^\ast\otimes \mathscr R_x)_{\psi_x}}$ induces a morphism in $\mathbf M^C_{\mathscr R_x}(\psi_x)$. Hence, $\mathscr M\in Mod-(C^\ast\otimes\mathscr R)_\psi$ may be treated as an object of $Mod^C-\mathscr R$. It may be directly verified that this structure is the inverse of the one defined by Propostion \ref{P99.3}.

\end{proof}

\smallskip
Finally, we will give an example of constructing entwined representations starting from $B$-comodule categories, where $B$ is a bialgebra. So let $B$ be a bialgebra over $K$, having multiplication $\mu_B$, unit map $u_B$ as well as comultiplication $\Delta_B$ and counit map $\varepsilon_B$. Then, the notion of a ``$B$-comodule category,'' which behaves like a $B$-comodule algebra with many objects, is implicit in the literature. 

\begin{defn}\label{D100.1} Let $B$ be a $K$-bialgebra. We will say that a small $K$-linear category $\mathcal R$ is a right $B$-comodule category if it satisfies the following conditions:

\smallskip
(i) For any $r$, $s\in \mathcal R$, there is a coaction $\rho=\rho(r,s):\mathcal R(r,s)\longrightarrow \mathcal R(r,s)\otimes B$, $f\mapsto \sum f_0\otimes f_1$, making $\mathcal R(r,s)$ a right $B$-comodule. Further, 
$\rho(id_r)=id_r\otimes 1_B$ for each $r\in \mathcal R$. 

\smallskip
(ii) For $f\in \mathcal R(r,s)$ and $g\in \mathcal R(s,t)$, we have
\begin{equation}\label{eq100.1}
\rho(g\circ f)=(g\circ f)_0\otimes (g\circ f)_1= (g_0\circ f_0)\otimes (g_1f_1)
\end{equation} We have   suppressed the summation signs  in \eqref{eq100.1}. We will always refer to a right $B$-comodule category as a co-$B$-category. We will only consider those $K$-linear functors
between co-$B$-categories whose action on morphisms is $B$-colinear. Together, the  co-$B$-categories form a new category, which we will denote by $Cat^B$. 
\end{defn}

\begin{lem}\label{L100.2} Let $B$ be a bialgebra over $K$. Let $\mathcal R$ be a co-$B$-category and let $C$ be a right $B$-module coalgebra. The collection $\psi:=\psi_{\mathcal R}=\{\psi_{rs}:C\otimes \mathcal R(r,s)\longrightarrow \mathcal R(r,s)\otimes C\}_{r,s
\in \mathcal R}$ defined by setting
\begin{equation}
\psi_{rs}(c\otimes f)=f_\psi\otimes c^\psi=f_0\otimes cf_1 \qquad f\in \mathcal R(r,s), \textrm{ }c\in C
\end{equation}
makes $(\mathcal R,C,\psi)$ an entwining structure.

\end{lem}
\begin{proof} We consider morphisms $f$, $g$ in $\mathcal R$ so that $gf$ is defined. Then, for $c\in C$, we see that
\begin{equation}
\begin{array}{c}
(gf)_\psi\otimes c^\psi=(gf)_0\otimes c(gf)_1=(g_0f_0)\otimes c(g_1f_1)=g_\psi f_\psi\otimes c^{\psi\psi}\\
f_\psi\otimes \Delta_C(c^\psi)=f_0\otimes \Delta_C(cf_1)=f_0\otimes c_1f_1\otimes c_2f_2=f_{00}\otimes c_1f_{01}\otimes c_2f_1=f_{\psi\psi}\otimes c_1^\psi\otimes c_2^\psi\\
\varepsilon_C(c^\psi)f_\psi=\varepsilon_C(c)\varepsilon_B(f_1)f_0=\varepsilon_C(c)f\qquad \psi(c\otimes id_r)=id_r\otimes c1_B\\
\end{array}
\end{equation} This proves the result.
\end{proof}

\begin{thm}\label{P100.3} Let $B$ be a $K$-bialgebra and let $C$ be a right $B$-module coalgebra. If $\mathscr X$ is a small category, a functor
$\mathscr R':\mathscr X\longrightarrow Cat^B$ induces an entwined $C$-representation of $\mathscr X$
\begin{equation}
\mathscr R:\mathscr X\longrightarrow \mathscr Ent_C\qquad x\mapsto (\mathscr R_x,C,\psi_x):=(\mathscr R'_x,C,\psi_{\mathscr R'_x})
\end{equation}

\end{thm}
\begin{proof}
It may be easily verified that the entwining structures constructed in Lemma \ref{L100.2} are functorial with respect to $B$-colinear functors between $B$-comodule categories. This proves the result. 
\end{proof}

We now consider a representation $\mathscr R':\mathscr X\longrightarrow Cat^B$ as in Proposition \ref{P100.3} and the corresponding entwined $C$-representation 
$\mathscr R:\mathscr X\longrightarrow \mathscr Ent_C$. By considering the underlying $K$-linear category of any co-$B$-category, we obtain an induced representation that we continue
to denote by $\mathscr R':\mathscr X\longrightarrow Cat^B\longrightarrow \mathscr Lin$.  We conclude by showing how entwined modules over $\mathscr R$ are related to modules over $\mathscr R'$ in the sense of Estrada and Virili \cite{EV}. 

\begin{thm} \label{P100.4}  Let $B$ be a $K$-bialgebra and let $C$ be a right $B$-module coalgebra. Let $\mathscr X$ be a small category, $\mathscr R':\mathscr X\longrightarrow Cat^B$ a functor and let $\mathscr R:\mathscr X\longrightarrow\mathscr Ent_C$ be the corresponding entwined $C$-representation. Then, a module $\mathscr M$ over $\mathscr R$
consists of the following data:

\smallskip
(1) A module $\mathscr M$ over the induced representation $\mathscr R':\mathscr X\longrightarrow Cat^B\longrightarrow \mathscr Lin$.

\smallskip
(2) For each $x\in \mathscr X$ and $r\in \mathscr R_x$ a right $C$-comodule structure $\rho_r^x:\mathscr M_x(r)\longrightarrow \mathscr M_x(r)\otimes C$ such that
\begin{equation*}
\rho_{s}^x(mf)= \big(mf\big)_0 \otimes \big(mf\big)_{1}=m_0f_0\otimes  {m_1}f_1
\end{equation*}
for every   $f \in \mathscr{R}_x(s,r)$  and $m \in \mathscr{M}_x(r).$ 

\smallskip
(3) For each morphism $\alpha:x\longrightarrow y$ in $\mathscr X$, the morphism $\mathscr M_\alpha(r):\mathscr M_x(r)\longrightarrow (\alpha_\ast\mathscr M_y)(r)$ is $C$-colinear for
each $r\in \mathscr R_x$. 

\end{thm}

\begin{proof} We consider a datum as described by the three conditions above.
The conditions (1) and (2) ensure that each $\mathscr M_x\in \mathbf M^C_{\mathscr R_x}(\psi_x)$. For each $x\in \mathscr X$, there is a forgetful functor $\mathscr F_x:\mathbf M^C_{\mathscr R_x}(\psi_x)\longrightarrow \mathbf M_{\mathscr R_x}$. Let $\alpha:x\longrightarrow y$ be a morphism in $\mathscr X$. From \eqref{cd7.1}, we know that   $(\alpha,id)_\ast: \mathbf M_{\mathscr R_y}^C(\psi_y)
\longrightarrow \mathbf M_{\mathscr R_x}^C(\psi_x)$ and $\alpha_\ast: \mathbf M_{\mathscr R_y}
\longrightarrow \mathbf M_{\mathscr R_x}$ are well behaved with respect to these forgetful functors. 
For each $r\in \mathscr R_x$, if $\mathscr M_\alpha(r):\mathscr M_x(r)\longrightarrow (\alpha_\ast\mathscr M_y)(r)$  is also $C$-colinear, it follows that $\mathscr M_\alpha$ is a morphism in $\mathbf M_{\mathscr R_x}^C(\psi_x)$. The result is now clear. 
\end{proof}

\end{document}